\documentclass{article}
\usepackage{url}
\usepackage{graphicx, float}
\usepackage{amsmath,amssymb,amsfonts}%
\usepackage{amsthm}%
\usepackage{mathtools}
\usepackage{authblk}

\mathtoolsset{showonlyrefs=true}
\usepackage{hyperref}
\usepackage{amsrefs}
\usepackage[margin=1.0in]{geometry}

\newcommand{\R}{\mathbb{R}}

\newcommand{\D}{\mathcal{D}}
\newcommand{\E}{\mathcal{E}}

\newcommand{\sgn}{\mathrm{sgn}}
\newtheorem{theorem}{Theorem}[section]
\newtheorem{lemma}[theorem]{Lemma}
\newtheorem{proposition}[theorem]{Proposition}
\newtheorem{corollary}[theorem]{Corollary}

\theoremstyle{definition}

\theoremstyle{remark}
\newtheorem{remark}[theorem]{Remark}

\numberwithin{equation}{section}

\title{Quantitative Hydrodynamic Stability for Couette Flow on Unbounded Domains with Navier Boundary Conditions}
\author[1]{Ryan Arbon}
\author[1]{Jacob Bedrossian}
\affil[1]{Department of Mathematics, University of California, Los Angeles, 90095, CA. 

Email: rarbon@math.ucla.edu, jacob@math.ucla.edu}

\begin{document}
\maketitle

\begin{abstract}
We prove a stability threshold theorem for 2D Navier-Stokes on three unbounded domains: the whole plane $\R \times \R$,  the half plane $\R \times [0,\infty)$ with Navier boundary conditions, and the infinite channel $\R \times [-1, 1]$ with Navier boundary conditions. Starting with the Couette shear flow, we consider initial perturbations $\omega_{in}$ which are of size $\nu^{1/2}(1+\ln(1/\nu)^{1/2})^{-1}$ in an anisotropic Sobolev space with an additional low frequency control condition for the planar cases. We then demonstrate that such perturbations exhibit inviscid damping of the velocity, as well as enhanced dissipation at $x$-frequencies $|k| \gg \nu$ with decay time-scale $O(\nu^{-1/3}|k|^{-2/3})$. On the plane and half-plane, we show Taylor dispersion for $x$-frequencies $|k| \ll \nu$ with decay time-scale $O(\nu |k|^{-2})$, while on the channel we show low frequency dispersion for $|k| \ll \nu$ with decay time-scale $O(\nu^{-1})$. Generalizing the work of \cite{bedrossian2023stability}  done on $\mathbb{T} \times [-1,1]$, the key contribution of this paper is to perform new nonlinear computations at low frequenc ies with wave number $|k| \lesssim \nu$ and at intermediate frequencies with wave number $\nu \lesssim |k| \leq 1$, and to provide the first enhanced dissipation result for a fully-nonlinear shear flow on an unbounded $x$-domain. Additionally, we demonstrate that the results of this paper apply equally to solutions of the perturbed $\beta$-plane equations from atmospheric dynamics.
\end{abstract}

\maketitle

\pagestyle{plain}

\setcounter{tocdepth}{1}
\tableofcontents

\section{Introduction}

Consider the incompressible 2D Navier-Stokes equations on $(x,y) \in \R \times D$, where $D = \R$, $D = [0,\infty)$, or $D = [-1,1]$. These correspond to the plane, the half-plane, and the infinite channel respectively. On the half-plane and on the infinite channel we impose inhomogeneous Navier boundary conditions. The governing equations are
\begin{equation}\label{velocity_equations}
    \begin{cases}
        \partial_t v + (v \cdot \nabla) v + \nabla P = \nu \Delta v,\\
        \nabla \cdot v =  0, \; \; v_2(t,x,y \in \partial D) = 0, \; \; \partial_y v_1(t,x, y \in \partial D) = 1, \; \\
        v(t=0, x,y) = v_{in}(x,y), \;\; (x,y) \in \R \times D,
    \end{cases}
\end{equation}
where $v = (v_1, v_2)$ is the fluid velocity, $P$ is the scalar pressure, and $\nu > 0$ is the kinematic viscosity. The boundary conditions arise physically from considering a fluid with Navier boundary conditions, and then applying a constant force to the (components of) the boundary $\partial D$. See \cite{masmoudi2003} for a derivations of Navier boundary conditions from kinetic theory.\\

We may re-cast \eqref{velocity_equations} into the vorticity formulation by defining $\nabla^{\perp} = ( -\partial_y, \partial_x)$ and setting $\Omega = \nabla^\perp v$, which we term the vorticity. In this formulation, the Navier boundary conditions on $v$ become inhomogeneous Dirichlet boundary conditions on $\Omega$. The full vorticity formulation is:
\begin{equation}\label{couette_first_system}
\begin{cases}
    \partial_t \Omega + v \cdot \nabla \Omega = \nu \Delta \Omega,\\
    v = \nabla^{\perp} (\Delta)^{-1} \Omega, \; \;  \Omega(t,x,y \in \partial D)=1,\\
    \Omega(t = 0,x,y)= \Omega_{in}(x,y), \; \; (x,y) \in \R \times D,
\end{cases}
\end{equation}
where  $\Delta^{-1}$ has homogeneous Dirichlet conditions on $y = 0$ when $D =[0,\infty)$ and $\Delta^{-1}$ has homogeneous Dirichlet conditions on $y = \pm 1$ when $D = [-1,1]$.\\

Physically, \eqref{couette_first_system} when $D = [-1,1]$ describes the flow of a two-dimensional fluid trapped between upper and lower flat boundaries which experience a constant force in opposite directions. The corresponding three-dimensional problem involves the movement of two plates which contain a slab of fluid. Although the three-dimensional problem is of much more immediate physical interest, the two-dimensional infinite channel is mathematically relevant and may prove an important stepping stone to analyzing the three-dimensional case \cite{stabilityoverview}. Correspondingly, $D = [0,\infty)$ models a single boundary moving according to a constant force. Meanwhile, the two-dimensional planar system $D = \R$ has potential physical relevance in atmospheric modeling. In geophysical contexts, one common approximation to the full spherical fluid equations in a rotating reference frame is the $\beta$-plane approximation, where one models the fluid dynamics by a tangent plane approximation, with additional consideration for the Coriolis force arising from the rotation \cite{vallis2017}. The resulting vorticity formulation is expressed as
\begin{equation}\label{beta_plane_full_system}
\begin{cases}
    \partial_t \Omega + v \cdot \nabla \Omega +  \mathfrak{b}v_2 = \nu \Delta \Omega,\\
        v = \nabla^{\perp} (\Delta)^{-1} \Omega,\\
    \Omega(t = 0,x,y)= \Omega_{in}(x,y), \; \; (x,y) \in \R \times \R,
\end{cases}
\end{equation}
where $\mathfrak{b} \neq 0$ is the Coriolis parameter. Note that we break from convention here, using the symbol $\mathfrak{b}$ rather than $\beta$, since $\beta$ will be used throughout the paper in a different context. We therefore term \eqref{beta_plane_full_system} the $\mathfrak{b}$-plane system. See \cite{stabilityoverview} for additional discussion of physical relevance.\\

The Couette flow $\Omega_0 = 1$ with velocity $v_0= (y,0)$ is a steady state solution to \eqref{couette_first_system} in all cases (and is a steady state solution to \eqref{beta_plane_full_system}). We shall be chiefly concerned with examining the \textit{quantitative asymptotic stability} of $\Omega_0$. As described in \cite{stabilityoverview}, given two normed function spaces $X$ and $Y$, a solution (traditionally an equilibrium solution) $v_E$ to the Navier-Stokes equations on a given domain is asymptotically stable with exponent $\gamma$ if, for any initial data $v_{in}$, we have that $||v_{in} - v_E||_X \ll \nu^\gamma$ implies $||v(t) - v_E||_Y \ll 1$ for all times $t>0$ and $||v(t) - v_E||_Y \to 0$ as $t \to \infty$.  See \cite{drazin1982} for an overview of hydrodynamic stability in a broad sense, and see \cite{stabilityoverview} for a discussion about shear flows in particular. However, in each of the previous references, the $x$-domain is taken to be periodic, rather than the unbounded $x$-domain considered in this paper.\\

In general, one has more freedom to choose the $||\cdot||_X$ norm than the $||\cdot||_Y$ norm, since the $||\cdot||_Y$ norm is constrained by the linearized problem \cite{stabilityoverview}. For shear flows in particular, one wants quantitative stability thresholds to preserve \textit{inviscid damping}, \textit{enhanced dissipation}, and at low enough frequencies \textit{Taylor dispersion}. Quantitative statements about convergence, and in particular proving enhanced dissipation and inviscid damping for the Navier-Stokes and Euler equations, have been a large topic of recent research, both in 2D \cites{bedrossian2016enhanced, bedrossian2018, chen2020transition2d, cotizelati2020, gallay2018, ionescu2020, ionescu2023, li2023, masmoudi2019} and 3D \cites{bedrossian2020dynamic, bedrossian2022dynamicsabove, bedrossian2017sobolev, chen2020transition3d}. The introduction of \cite{li2023} contains a table summarizing the currently known threshold values $\gamma$ corresponding to different regularity classes and boundary conditions in the context of 2D Couette flow.\\

Inviscid damping refers to time-decay of the perturbed velocity in the Euler equations (either linearized or nonlinear). See \cites{bedrossian2019vortex, ionescu2024linearinv, jia2020gevrey, jia2023monotone} for discussion of inviscid damping in the linearized case. See \cite{wei_zhang_zhu_2020} for a discussion of linearized inviscid damping for the $\mathfrak{b}$-plane system \eqref{beta_plane_full_system}, and see \cite{wang_zhang_zhu_2023} for nonlinear inviscid damping of 2D Couette flow in Gevrey class on $\mathbb{T} \times \R$ for \eqref{beta_plane_full_system}. Enhanced dissipation refers to the increased dissipation exhibited by a passive scalar advected by a shear flow, causing rapid decay of frequencies in the $x$ direction. For linear passive scalars in 2D and for $x$-freqencies $|k| \gg \nu$,  the decay is like $\exp(\nu^{1/3} |k|^{2/3} t)$ \cite{cotizelati2023}. In the context of linearized periodic Couette flow, enhanced dissipation was observed by Kelvin with dissipation time scale $\nu^{-1/3}$ \cite{kelvin1887}. See \cites{bedrossiancotizelati2017, stabilityoverview, constantin2008, zlatos2010} for mathematical discussions and \cites{bernoff1994, lundgren1982, rhines1983} for physics-based discussion. Taylor dispersion is a low-frequency phenomenon also observed in passive scalars, where frequencies $|k| \ll \nu$ disperse (in 2D) as $\exp(\nu^{-1} |k|^2 t)$. See \cites{aris1956, taylor1953, taylor1954} for analyses of Taylor dispersion in long tubes or pipes from a physical perspective. See for instance \cite{cotizelati2023} for a unified mathematical treatment of enhanced dissipation and Taylor dispersion for linear passive scalar advection-diffusion equations in infinite cylinders of arbitrary dimension. As noted in \cite{cotizelati2023}, when a PDE system on a bounded domain has Dirichlet boundary conditions, it is often possible through the Poincar\'e  inequality to obtain low-frequency dispersion on time scales like $\nu^{-1}$ in-line with the heat equation, rather than the standard Taylor dispersion rates. 
For Couette flow, this enhanced dissipation is captured by showing decay of frequencies $k$ with $|k| \gg \nu$ on time scales like $k^{-2/3}\nu^{-1/3}$. We shall see Taylor dispersion effects on the plane and the half-plane, while on the infinite channel we shall see that the Dirichlet boundary conditions of the perturbation equation yield decay on time scales like $\nu^{-1}$.

In this paper, we shall prove inviscid damping, enhanced dissipation, and low frequency dispersion for Couette flow in the following sense. Consider an initial vorticity $\Omega_{in} = \Omega_0 + \omega_{in}$, where on the half-plane we assume $\omega_{in}(x,y = 0) = 0$ and on the infinite channel we assume that $\omega_{in}(x,y = \pm 1) = 0$. Then the corresponding solution to \eqref{couette_first_system} can be written as $\Omega = 1 + \omega(t,x,y)$, where $\omega$ satisfies the respective perturbation equation system:
\begin{equation}\label{couette_system}
\begin{cases}
        \partial_t \omega + y \partial_x \omega + u \cdot \nabla \omega = \nu \Delta \omega,\\
    u = \nabla^\perp (\Delta)^{-1} \omega,\\
    \omega(0), = \omega_{\textrm{in}}, \; \omega(t,x,y \in \partial D) = 0,
\end{cases}
\end{equation}
where $\phi = \Delta^{-1} \omega$. The perturbation system in the $\mathfrak{b}$-plane case only differs slightly, and is given by
\begin{equation}\label{beta_system}
\begin{cases}
        \partial_t \omega + y \partial_x \omega + u \cdot \nabla \omega + \mathfrak{b} \partial_x \phi = \nu \Delta \omega,\\
    u = \nabla^\perp (\Delta)^{-1} \omega,\\
    \omega(0) = \omega_{\textrm{in}}, \; \omega(t,x,y \in \partial D) = 0.
\end{cases}
\end{equation}
We prove a quantitative stability threshold for the initial perturbation $\omega_{in}$. Namely, we show that for a suitable anisotropic Sobolev norm $|| \cdot ||_X$, that if $||\omega_{in}||_X \lesssim \nu^{1/2}(1+\ln(1/\nu)^{1/2})^{-1}$, then the corresponding solution $\omega$ of \eqref{couette_system} or \eqref{beta_system} exhibits enhanced dissipation, and the perturbation velocity $u = \nabla^{\perp} (\Delta)^{-1} \omega$ undergoes inviscid damping. We show that if $D = \R$ or $[0,\infty)$, then $\omega$ undergoes Taylor dispersion, and if $D = [-1,1]$, then $\omega$ undergoes dispersion inline with the heat equation. The more  technically interesting result is the planar case, which we state first. Given a function $f$, we denote its Fourier transform by
$$f_k(t,y) \coloneqq \frac{1}{2\pi}\int_{\R} f(t,x,y) e^{-ikx}dx.$$
We use similar notation for the Fourier transform of the intial datum $\omega_{in}$, writing $\omega_{in, \, k}$. 
\begin{theorem}\label{main_theorem}
     Suppose $\omega_{in}$ is initial datum for \eqref{couette_system} with $D = \R$. Then for all $m \in (1/2, 1)$ and $J \in [1,\infty)$, there exists a constant $\delta_{1}$$ >0$ independent of $\nu$ such that if
    $$\sum_{0 \leq j \leq 1}|| \langle \partial_x \rangle ^m \langle \frac{\partial_x}{\nu}\rangle^{-j/3}\partial_y^j \omega_{in}||_{L_{x,y}^2} + || \omega_{in, \, k} ||_{L_k^\infty L_y^2}= \epsilon \leq \delta_1 \frac{\nu^{1/2}}{1+\ln(1/\nu)^{1/2}},$$
    then for all $c >0$ sufficiently small (independent of $\nu$ and $\delta_1$) and all $\nu \in (0,1)$, the corresponding solution $\omega$ to \eqref{couette_system} satisfies
\begin{equation}
    \begin{split}
        &\sum_{0 \leq j \leq 1}|| \langle c \lambda^{pl}(\nu, \partial_x) t \rangle^J \langle \partial_x \rangle ^m \langle \frac{\partial_x}{\nu}\rangle^{-j/3} \partial_y^j \omega||_{L_{x,y}^2} + || \omega_k ||_{L_k^\infty L_y^2} \leq 2\epsilon, \quad \forall t\in [0,\infty); \\
        &\sum_{0 \leq j \leq 1}|| \langle c \lambda^{pl}(\nu, \partial_x) t \rangle^J \langle \partial_x \rangle^m \langle \frac{\partial_x}{\nu}\rangle^{-j/3} \partial_y^j \partial_x u||_{L_t^2 L_{x,y}^2} + ||ku_k||_{L_k^\infty L_{t,y}^2}\leq 2\epsilon.
    \end{split}
\end{equation}
    Here $\lambda^{pl}(\nu,\partial_x)$ is a Fourier multiplier defined on the Fourier side as
    $$\lambda^{pl}(\nu, k) = \begin{cases}
        \nu^{1/3}|k|^{2/3}, & |k| \geq \nu\\
        \frac{|k|^2}{\nu}, & |k| \leq \nu.
    \end{cases}$$
\end{theorem}
The statement of the theorem for the case of the half-plane is virtually identical:
\begin{theorem}\label{half_theorem}
     Suppose $\omega_{in}$ is initial datum for \eqref{couette_system} with $D = [0,\infty)$, such that $\omega_{in}(y = 0) = 0$. Then for all $m \in (1/2, 1)$ and $J \in [1,\infty)$, there exists a constant $\delta_{1}$$ >0$ independent of $\nu$ such that if
    $$\sum_{0 \leq j \leq 1}|| \langle \partial_x \rangle ^m \langle \frac{\partial_x}{\nu}\rangle^{-j/3}\partial_y^j \omega_{in}||_{L_{x,y}^2} + || \omega_{in,k} ||_{L_k^\infty L_y^2}= \epsilon \leq \delta_1 \frac{\nu^{1/2}}{1+\ln(1/\nu)^{1/2}},$$
    then for all $c >0$ sufficiently small (independent of $\nu$ and $\delta_1$) and all $\nu \in (0,1)$, the corresponding solution $\omega$ to \eqref{couette_system} satisfies
\begin{equation}
    \begin{split}
        &\sum_{0 \leq j \leq 1}|| \langle c \lambda^{pl}(\nu, \partial_x) t \rangle^J \langle \partial_x \rangle ^m \langle \frac{\partial_x}{\nu}\rangle^{-j/3} \partial_y^j \omega||_{L_{x,y}^2} + || \omega_k ||_{L_k^\infty L_y^2} \leq 2\epsilon, \quad \forall t\in [0,\infty); \\
        &\sum_{0 \leq j \leq 1}|| \langle c \lambda^{pl}(\nu, \partial_x) t \rangle^J \langle \partial_x \rangle^m \langle \frac{\partial_x}{\nu}\rangle^{-j/3} \partial_y^j \partial_x u||_{L_t^2 L_{x,y}^2} + ||ku_k||_{L_k^\infty L_{t,y}^2}\leq 2\epsilon.
    \end{split}
\end{equation}
\end{theorem}

The theorem in the case of the infinite channel is slightly simpler, and is as follows:

\begin{theorem}\label{channel_theorem}
Suppose that $\omega_{in}$ is initial datum for \eqref{couette_system} with $D = [-1,1]$, such that $\omega_{in}(y = \pm 1) = 0$. Then for all $m \in (1/2, 1)$, there exists a constant $\delta_2$ independent of $\nu$ such that if
    $$\sum_{0 \leq j \leq 1}|| \langle \partial_x \rangle ^m \langle \frac{\partial_x}{\nu}\rangle^{-j/3} \partial_y^j \omega_{in}||_{L^2} = \epsilon \leq \delta_2 \frac{\nu^{1/2}}{1+\ln(1/\nu)^{1/2}},$$
    then for all $c >0$ sufficiently small (independent of $\nu$ and $\delta_2$) and all $\nu \in (0,1)$, the corresponding solution $\omega$ to \eqref{couette_system} satisfies
    \begin{equation}
        \begin{split}
            &\sum_{0 \leq j \leq 1}|| e^{c \lambda^{ch}(\nu, \partial_x) t } \langle \partial_x \rangle ^m \langle \frac{\partial_x}{\nu}\rangle^{-j/3} \partial_y^j \omega||_{L_{x,y}^2} \leq \epsilon, \quad \forall t\in [0,\infty);\\
            & \sum_{0 \leq j \leq 1}|| e^{ c \lambda^{ch}(\nu,\partial_x) t} \langle \partial_x \rangle^m \langle \frac{\partial_x}{\nu}\rangle^{-j/3} \partial_y^j \partial_x u||_{L_t^2 L_{x,y}^2} \leq \epsilon.
        \end{split}
    \end{equation}
    Here $\lambda^{ch}(\nu,\partial_x)$ is a Fourier multiplier defined on the Fourier side as
    $$\lambda^{ch}(\nu, k) = \begin{cases}
        \nu^{1/3}|k|^{2/3}, & |k| \geq \nu\\
        \nu, & |k| \leq \nu.
    \end{cases}$$
\end{theorem}
The super-scripts in $\lambda^{pl}$ and $\lambda^{ch}$ denote \textit{planar} and \textit{channel}, respectively.\\

We will see that a minor adjustment to the proof of Theorem \ref{main_theorem} yields the following corollary concerning the $\mathfrak{b}$-plane.
\begin{corollary}\label{beta_corollary}
Let $\omega_{in}$ be initial data for \eqref{beta_system}. Then the statement of Theorem \ref{main_theorem} applies to $\omega$ solving \eqref{beta_system}.
\end{corollary}
\begin{remark}
    The proofs of Theorem \ref{main_theorem} and Corollary \ref{beta_corollary} differ only at the linear level. We will therefore only discuss the differences in Section \ref{lin_beta_plane}
\end{remark}

We shall prove Theorems \ref{main_theorem}, \ref{half_theorem}, and \ref{channel_theorem} using an energy method similar to that used in proving Theorem 1.1 in \cite{bedrossian2023stability}, with adjustments to treat the low-frequency Taylor dispersion regime coming from \cite{cotizelati2023}. Just as \cite{bedrossian2023stability}, we will first perform linear estimates based on a modified hypocoercive scheme \cite{cotizelati2023} and then prove a bootstrap estimate to propagate these linear estimate to the non-linear level. The key contribution of this paper is in the handling of frequencies with $|k| < 1$, including the interactions between low ($|k| \leq \nu$), intermediate ($\nu \leq |k| \leq 1$), and high ($|k| \geq 1$) frequencies arising in the non-linearity. The intermediate frequencies require particular attention. A singular integral operator, denoted $\mathfrak{J}_k$, was used in \cite{bedrossian2023stability} to prove inviscid damping estimated on $\mathbb{T} \times [-1,1]$. We are able to adapt this in order to obtain planar inviscid damping estimates. Furthermore, we note that \cite{bedrossian2023stability} considered perturbations of \textit{near}-Couette flow on the periodic channel, while we are concerned with perturbations of Couette flow. We are at present unable to extend the result to \textit{near}-Couette flow on the infinite channel, since estimates for the inviscid damping operator lead to blow-up at low frequencies at even the linear level (contrast Lemma \ref{basic_hypo_est} with Lemma 4.1 of \cite{bedrossian2023stability}). These low freqencies, or correspondingly large wavelengths, are known to produce instabilities in general classes of shear flows which do not include Couette flow \cites{friedlander, grenier}.\\

To the authors' knowledge, this is the first quantitative stability result giving enhanced dissipation and inviscid damping for fully nonlinear Navier or Euler shear flows where the $x$ variable is unbounded and non-periodic, in contrast with previous results which have domains such as $\mathbb{T} \times \R$ \cites{bedrossian2016enhanced, bedrossian2018, cotizelati2020, li2023, masmoudi2019}  and $\mathbb{T} \times [-1,1]$ modulo rescaling \cites{bedrossian2023stability, chen2020transition2d, ionescu2020, ionescu2023}. There do exist stability results for Couette flow in the 3D infinite channel, but they are either not fully quantitative \cite{romanov1974}, or are quantitative with exponent $\gamma = 3$ or $4$ (depending on the velocity component); however, these results do not consider enhanced dissipation, inviscid damping, or Taylor dispersion \cites{kreiss1994, liefvendahl2002}. There are comparable non-linear stability results in kinetic theory which prove enhanced dissipation, Taylor dispersion as applicable, and Landau-type damping (similar to inviscid damping). See for instance the recent works \cites{albritton2023, bedrossian2022boltzmann, dietert2023}. Additionally, we note the difference in decay rates between Theorems \ref{main_theorem} and \ref{half_theorem} and Theorem \ref{channel_theorem}. While in Theorem \ref{channel_theorem}, we are able to obtain the exponential decay implied by the linearized problem, the decay in Theorems \ref{main_theorem} and \ref{half_theorem} is a polynomial rate of arbitrarily high order $J$, with the constant $\delta_1$ depending on $J$. This is similar to the Taylor dispersion and enhanced dissipation results obtained in \cite{bedrossian2022boltzmann} for the Boltzmann equations posed on the whole space. While the infinite channel has an unbounded $x$ domain, the ability to use the Poincar\'e inequality causes $\lambda^{ch} = \nu$ at low frequencies, rather than $\nu^{-1} k^{2}$. Sub-additivity of $k \mapsto \lambda^{ch}(\nu,k)$ enables us to propagate the exponential decay even at the non-linear level.\\

Lastly, we note that Theorems \ref{main_theorem} and \ref{half_theorem} assume a smallness condition on $||\omega_{in, \, k}||_{L_k^\infty L_y^2}$. This is used to control non-linear terms where two frequencies $k$ and $k'$ are very close ($|k-k'| \lesssim \nu$). In the case of the infinite channel, the Poincar\'e inequality is able to replace this control. Indeed, the Poincar\'e inequality simplifies several steps in the proof of Theorem \ref{channel_theorem}. Furthermore, the proof of Theorem \ref{half_theorem} is virtually identical to the proof of Theorem \ref{main_theorem}, modulo certain estimates on the operator $\mathfrak{J}_k$. For this reason, we will primarily focus on Theorem \ref{main_theorem} in Sections \ref{outline_section} -- \ref{nonlin}. In Section \ref{half_modifications}, we will discuss the required estimates to finish the proof of Theorem \ref{half_theorem}. In Section \ref{channel_modifications}, we will sketch the proof of Theorem \ref{channel_theorem}, noting the differences from \ref{main_theorem}.

\subsection{Notation}

We use the standard notation $\langle \cdot \rangle$ to denote $\langle a \rangle = \sqrt{1 + a^2}$ for any quantity $a$. Given two quantities $A, B \geq 0$, we write $A \lesssim B$ to indicate that there exists a constant $C > 0$ such that $A \leq C B$. If the constant $C$ depends on a parameter $\theta$, we shall write $A \lesssim_\theta B$. If $A \lesssim B$ and $B \lesssim A$, we write $A \approx B$.

For $D = \R$, $[0,\infty)$, or $[-1,1]$, we will use the notation $||\cdot||_2$ and $||\cdot||_\infty$ to denote the $L^2(D)$ and $L^\infty(D)$ norms, respectively. The notation $\langle \cdot, \cdot\rangle$ will denote the inner product on $L^2(D)$. We will use $||\cdot||_{L^p \to L^p}$ to denote the operator norm of an operator from $L^p(D)$ to itself. In Section \ref{half_modifications}, we will have occasion to use norms $||\cdot||_{L^p_{a}L^q_{b}}$ which denote the mixed Lebesgue norm with $(a,b) \in D \times D$. In general, if there is some ambiguity as to which variable is associated with a given norm, we will use a subscript, e.g. $L^2_k$ would refer to the $L^2$ norm in the $k$-variable.

\section{Outline}\label{outline_section}

We recall that we focus on the proof of Theorem \ref{main_theorem}, rather than Theorem \ref{half_theorem} or \ref{channel_theorem}. Hence our statements will be focused on the planar case, and throughout Sections \ref{outline_section} -- \ref{nonlin} we will assume that $D = \R$. Corresponding statements apply to the infinite channel with obvious modifications. Non-obvious modifications will be discussed in Section \ref{channel_modifications}. In the linearized version of \eqref{couette_system}, taking the Fourier transform in the $(x,k)$ Fourier pair decouples the problem into $k$-mode by $k$-mode. Hence, we will build our energy functional $\mathcal{E}$ (to be defined later in \eqref{introduce_energy}) out of a $k$-by-$k$ functional $E_k$ adapted to the linearized problem. If $\omega_k$ solves \eqref{couette_system} (or the linearized system \eqref{linearized_system}), define
\begin{equation}\label{kEnergy}
    E_k[\omega_k] \coloneqq ||\omega_k||_2^2 + c_\alpha \alpha ||\partial_y \omega_k||_2^2 - c_\beta \beta \textrm{Re} \langle ik \omega_k, \partial_y \omega_k \rangle + c_\tau  \textrm{Re}\langle \mathfrak{J}_k \omega_k, \omega_k \rangle +  c_\tau c_\alpha \alpha\textrm{Re}\langle \mathfrak{J}_k \partial_y \omega_k, \partial_y \omega_k\rangle,
\end{equation}
where
\begin{equation}
    \alpha \coloneqq
    \begin{cases}
        1, & |k| <  \nu,\\
        \nu^{2/3} |k|^{-2/3}, & |k| \geq  \nu,
    \end{cases}
\end{equation}
\begin{equation}
    \beta \coloneqq
    \begin{cases}
         \nu^{-1}, & |k| < \nu,\\
         \nu^{1/3} |k|^{-4/3}, & |k| \geq \nu, 
    \end{cases}
\end{equation}
and $c_\alpha$, $c_\beta$, $c_\tau$ are positive constants to be determined later. The operator $\mathfrak{J}_k$ is a convolution operator, which is a modification of the inviscid damping operator $\mathfrak{J}_k$ introduced in \cite{bedrossian2023stability}. Although at the linear level, distinct wave numbers are decoupled, so that one views $k$, and hence $\alpha$ and $\beta$, as constants, at the nonlinear level we shall consider $\alpha$ and $\beta$ as Fourier multipliers, and hence we define
\begin{equation}\label{definitions_of_A_and_B}
    A(k) \coloneqq
    \begin{cases}
        1, & |k| <  \nu,\\
        \nu^{1/3} |k|^{-1/3}, & |k| \geq  \nu,
    \end{cases} \; \; \; B(k) \coloneqq
    \begin{cases}
         \nu^{-1/2}, & |k| < \nu,\\
         \nu^{1/6} |k|^{-2/3}, & |k| \geq \nu. 
    \end{cases}
\end{equation}
The energy $E_k$ is a modification of the now-standard hypocoercive energy function found in, for example, \cites{bedrossiancotizelati2017, cotizelati2023}, with the addition of terms corresponding to the inviscid damping operator $\mathfrak{J}_k$ as introduced in \cite{bedrossian2023stability}. On the plane, the convolution operator $\mathfrak{J}_k : L^2(\R) \to L^2(\R)$ is defined on Schwartz functions $f\in \mathcal{S}(\R)$ by
\begin{equation}
\begin{split}
        \mathfrak{J}_k[f](y) \coloneqq |k| \textrm{p.v.} \frac{k}{|k|} \int_{-1}^1 \frac{1}{2i(y-y')} G_k(y,y') f(y') dy',
\end{split}
\end{equation}
where
\begin{equation}
\begin{split}
        G_k(y,y') &= -\frac{1}{|k|}e^{-|k(y-y')|}.
\end{split}
\end{equation}
 As shown in Lemma \ref{boundedness_of_J_k}, the operator $\mathfrak{J}_k$ is a linear operator $L^2\to L^2$ which is uniformly bounded in $k$. The origin of $\mathfrak{J}_k$ is discussed in \cite{bedrossian2023stability}, and arises from a physical side representation of a series expansion of the ``ghost multiplier" $\mathcal{M}$ employed in \cites{bedrossian2020dynamic, bedrossian2016enhanced, zillinger2014}. In particular, we note that while the original $\mathfrak{J}_k$ was adapted to the channel (and this is the version of $\mathfrak{J}_k$ which we will use in Section \ref{channel_modifications} to prove Theorem \ref{channel_theorem}), the deriviation of this $\mathfrak{J}_k$ actually passes through the definition which we will use. In particular, the kernel $\frac{k}{iy}G_k$ (abusing notation slightly) was found, and then it was noticed that $G_k$ is the fundamental solution of $\Delta_k =\partial_y^2 - k^2$ on $\R$. The operator $\mathfrak{J}_k$ in \cite{bedrossian2023stability} and Section \ref{channel_modifications} arises by replacing $G_k$ with the fundamental solution of $\partial_y^2 - k^2$ on $[-1,1]$ subject to Dirichlet boundary conditions, while the operator used in Section \ref{half_modifications} on the half-plane corresponds to letting $G_k$ be the fundamental solution of $\partial_y^2 - k^2$ on $[0,\infty)$ with Dirichlet boundary conditions.\\

To prove Theorem \ref{main_theorem}, we will estimate $\frac{d}{dt} E_k[\omega_k]$ and $\frac{d}{dt}E_k$ in the linear case in Section \ref{linear_section}. Then in Section \ref{nonlin}, we will combine the linear results with estimates on the nonlinear contributions to $\frac{d}{dt} \mathcal{E}$ to prove a bootstrap lemma, completing the proof of Theorem \ref{main_theorem}.

\subsection{Linearized Problem}

We now make the references to the linearized problem precise. Consider the following mode-by-mode linear system, arising from dropping the non-linear term from \eqref{couette_system} and taking the Fourier transform:
\begin{equation}\label{linearized_system}
    \begin{cases}
        \partial_t \omega_k + ik y \omega_k - \nu \Delta_k \omega_k = 0,\\
        \Delta_k \phi_k = \omega_k.
    \end{cases}
\end{equation}
The operator $\Delta_k$ is defined by $\Delta_k = - |k|^2 + \partial_y^2$ with associated gradient $\nabla_k = (ik, \partial_y)$. We note that $G_k$ is the fundamental solution of $\Delta_k$ on $\R$. In addition to the $k$-by-$k$ energy $E_k$ in \eqref{kEnergy}, we introduce the corresponding $k$-by-$k$ dissipation functionals,
\begin{equation}\label{k_by_k_dissipation}
    \begin{split}
        D_k[\omega_k]& \coloneqq \nu ||\nabla_k \omega_k||_2^2 + c_\alpha \alpha || \nabla_k \partial_y \omega_k||_2^2 + c_\beta \beta |k|^2 ||\omega_k||_2^2 + c_\tau  |k|^2 ||\nabla_k \phi_k||_2^2 + c_\tau c_\alpha \alpha  |k|^2 ||\nabla_k \partial_y \phi_k||_2^2\\
        &\eqqcolon D_{k,\gamma} + c_\alpha D_{k,\alpha} + c_\beta D_{k,\beta} + c_\tau D_{k,\tau} +  c_\tau c_\alpha D_{k,\tau \alpha},
    \end{split}
\end{equation}

The terms of $E_k$ involving $\mathfrak{J}_k$ will be shown (see Lemma \ref{disip_lin}) to extract the $L_t^2 \dot{H}^1_x L_y^2$ inviscid damping estimate:
$$\sum_{0 \leq j \leq 1} |k|^2 || (\alpha \partial_y)^j \nabla_k \phi_k ||_{L_t^2 L_y^2}^2 \lesssim E_k[ \omega_k(0)].$$

Before we state the primarily linear result, we introduce the shorthand $\lambda_k = \lambda^{pl}(\nu, k)$. Our main result at the linear level is:

\begin{proposition}\label{linear_proposition}
    There exist constants $c_\alpha$, $c_\beta$, and $c_\tau$, and constants $c_0 = c_0(c_\tau, c_\alpha, c_\beta)$, $c_1 = c_1(c_\tau, c_\alpha, c_\beta)$ which can be chosen independently of $\nu$ such that, for any $H^1$ solution $\omega$ to \eqref{linearized_system}, the following holds for any $k \neq 0$:
    $$\frac{d}{dt} E_k[\omega_k] \leq - c_1 D_k[\omega_k] - c_0\lambda_k E_k[\omega_k].$$
    In particular, this implies that for any $c>0$ sufficiently small
    $$\frac{d}{dt} E_k[\omega_k] \leq - 8c D_k[\omega_k] - 8c\lambda_k E_k[\omega_k],$$
    and the following inviscid damping  and enhanced dissipation estimate holds:
    $$e^{2c\lambda_k t} E_k [\omega_k(t)] + \frac{1}{4}c_\tau \int_0^t e^{2 c \lambda_k s} |k|^2 ||\nabla_k \phi_k(s)||_{2}^2 ds \leq E_k[\omega_k(0)], \; \; \forall t \geq 0.$$
\end{proposition}

Proposition \ref{linear_proposition} will be proven in Section \ref{linear_section}.

\subsection{Nonlinear Problem}

For the nonlinear problem, taking the Fourier transform of \eqref{couette_system} gives rise to the nonlinear $(u \cdot \nabla \omega)_k$ term, which combines multiple frequencies. Hence we cannot rely solely on the mode-by-mode functional. Recalling the shorthand $\lambda_k = \lambda^{pl}(\nu,k)$, we introduce the functionals $\mathcal{E}_1$, $\mathcal{E}_2$, and $\mathcal{E}$:
\begin{equation}\label{introduce_energy}
\begin{split}
     &\mathcal{E}_1 \coloneqq \int \frac{\langle c \lambda_k t \rangle^{2J}}{M_k(t)} \langle k \rangle^{2m} E_k[\omega_k] dk,\\
        &\mathcal{E}_2 \coloneqq \sup_{k \in \R}\left(||\omega_k||_2^2 + c_\tau \textrm{Re}\langle \mathfrak{J}_k \omega_k, \omega_k \rangle\right),\\
        &\mathcal{E} \coloneqq \mathcal{E}_1 + \mathcal{E}_2,
\end{split}
\end{equation}
with the (time-integrated) dissipations
\begin{equation}\label{introduce_dissipation}
\begin{split}
        &\mathcal{D}_1(t) \coloneqq \int_0^t \int_\R \frac{\langle c \lambda_k s \rangle^{2J}}{M_k(s)} \langle k \rangle^{2m}D_{k}[\omega_k] dk ds,\\
    &\D_2(t) \coloneqq \sup_{k \in \R}  \left(\int_0^t D_{k,\gamma} + c_\tau D_{k, \tau}ds\right ),\\
    &\D \coloneqq \D_1 + \D_2,
\end{split}
\end{equation}
where $M_k(t)$ is the solution to the ODE:
\begin{equation}\label{definition_of_Mk(t)}
    \begin{split}
        \dot{M}_k(t) &= c J^2 \lambda_k \frac{(c \lambda_k t)^2}{\langle c \lambda_k t \rangle^4} M_k(t)\\
        M_k(0) &= 1.
    \end{split}
\end{equation}
The constant $c$ in the preceding definitions is a positive constant which is sufficiently small, as determined by Proposition \ref{linear_proposition}. Importantly, the smallness condition on $c$ is independent of $m$, $J$, and $\nu$.\\

The time-dependent multiplier $M_k(t)$ is included in \eqref{introduce_energy} and \eqref{introduce_dissipation} to address terms in $\frac{d}{dt}\mathcal{E}_1$ which arise from the time-derivative falling on $\langle c \lambda_k t \rangle^{2J}$. Solving the ODE \eqref{definition_of_Mk(t)} explicitly, it is easy to see that $1 \leq M_k(t) \leq e^{4J^2 /3}$ for all $t \geq 0$ and all $k \in \R$. Thus $M_k(t)$ is uniformly bounded above and below in both $k$ and $t$ and $||M_k(t) f||_{L^2_k L^2_y} \approx ||f||_{L^2_k L^2_y}$ for any $f \in L^2_k L^2_y(\R \times \R)$. Similar so-called ``ghost multipliers" have been used in \cites{zillinger2014, bedrossian2020dynamic, bedrossian2022dynamicsabove, bedrossian2017sobolev, bedrossian2018}.\\

We additionally define, for $* \in \{\gamma, \tau, \alpha, \alpha\tau, \beta\}$, the quantities: 
\begin{equation}\label{different_dissipations_definition}
    \mathcal{D}_*(t) \coloneqq \int_0^t \int_\R \frac{\langle c \lambda_k s \rangle^{2J}}{M_k(s)}  \langle k \rangle^{2m}D_{k, *}[\omega_k] dk ds,
\end{equation}
as well as
\begin{equation}
    \D_{\gamma, 2}(t) \coloneqq \sup_{k \in \R} \int_0^t D_{k,\gamma} ds, \quad \D_{\tau,2}(t) = \sup_{k \in \R} \int_0^t D_{k, \tau} ds.
\end{equation}
Similarly to \cite{bedrossian2023stability}, the key estimate, which directly implies Theorem \ref{main_theorem}, is the following bootstrap lemma. Unlike \cite{bedrossian2023stability}, the presence of the supremum terms in the energy \ref{introduce_energy} prevents us from expressing the lemma in a purely differential form. Instead, we have the following:
\begin{lemma}[Bootstrap Lemma]\label{bootstrap}
There exists a constant $C >0$ depending only on $m$, $J$, and the choice of $c \in (0,1)$ from Proposition \ref{linear_proposition} such that for all $t \in [0,\infty)$:
\begin{equation}\label{derivative_of_energy}
    \mathcal{E}(t) \leq 2\mathcal{E}(0) - 4c \D(t) + \left(C \frac{(\ln(1/\nu)^{1/2} + 1)^2}{\nu}\sup_{s \in [0,t]}\mathcal{E}(s) \right)^{1/2}\D(t),
\end{equation}
for $\mathcal{E}$ as in \eqref{introduce_energy}, $\mathcal{D}$ as in \eqref{introduce_dissipation}, and $\omega$ an $H^1$ solution to \eqref{couette_system} such that $\mathcal{E}(0) < \infty$. In particular, \eqref{derivative_of_energy} holds for all times $t \geq 0$ such that $\mathcal{E}(t) < \infty$.
\end{lemma}Notice that Lemma \ref{bootstrap} implies that if $\mathcal{E}(0) \leq \frac{c^2 \nu}{4 C (\ln(1/\nu)^{1/2} + 1)^2}$, then a simple bootstrap argument yields
\begin{align}\sup_{t \in [0,\infty)} \left(\mathcal{E}(t)  +2c \mathcal{D}(t) \right)\leq 2\mathcal{E}(0),\end{align}
which implies that Theorem \ref{main_theorem} holds for times $t \in (0,\infty)$ since $1 \leq M_k(t) \leq e^{4J^2 /3}$, meaning that $\mathcal{E}^{1/2}$ is equivalent to the norm in the conclusion of the theorem, while $\mathcal{D}$ gives the velocity estimate.  The bulk of this paper is devoted to the proof of Lemma \ref{bootstrap}, with special attention payed toward the contribution of the non-linear terms. The main distinction between this paper and \cite{bedrossian2023stability} is the need to treat various combinations of high, low, and intermediate frequencies all interacting across continuous ranges, and this is the primary topic of Section \ref{nonlin}. Additionally, we have the added need to properly account for the time integrals.

\subsection{Properties of \texorpdfstring{$\mathfrak{J}_k$}{Jk}}
Before beginning the proofs of Proposition \ref{linear_proposition} and Lemma \ref{bootstrap}, we prove boundedness of $\mathfrak{J}_k$. The bound from $L^2 \to L^2$ follows immediately by noting that $\mathfrak{J}_k$ is the physical-in-$y$ representation of the Fourier multiplier $\arctan\left(\frac{\eta}{k}\right),$ where $\eta$ is the frequency conjugate to $y$. Then we see $|\arctan\left(\frac{\eta}{k}\right)| \leq \frac{\pi}{2}$, so $||\mathfrak{J}_k||_{L^2 \to L^2} \leq \frac{\pi}{2}$ uniformly in $k$. We are able to do even more, proving boundedness in (most) $L^p$ spaces.

\begin{lemma}\label{boundedness_of_J_k}
    For all $k \in \R \setminus \{0\}$, and for all $p \in (1,\infty)$, the operator $\mathfrak{J}_k$  extends to a bounded linear operator $L^p(\R) \to L^p(\R)$ satisfying the uniform bound   $$||\mathfrak{J}_k||_{L^p \to L^p} \lesssim 1,$$
    where importantly the bound is independent of $k$, but can depend on $p$.
\end{lemma}

\begin{proof}
    The proof of Lemma \ref{boundedness_of_J_k} follows by proving that $\mathfrak{I}_k$ is a convolution operator of Calder\'on-Zygmund type satisfying estimates uniform in $k$. See for instance Chapter 5 of Grafakos \cite{grafakos2014}.\\

    Let us define the convolution kernel $F(y) \coloneqq \frac{\sgn(k)}{2i} \frac{e^{-|ky|}}{y}$. Then for $f \in \mathcal{S}(\R)$, we have $\mathfrak{J}_k = \textrm{p.v}\left(F * f\right)$. We now proceed to check that $F$ satisfies sufficient hypotheses to apply the Calder\'on-Zygmund theory. Trivially we have the size condition,
    $$|F(y)| \leq \frac{1}{|y|}.$$
    Similarly, we check the H\"ormander condition:
        \begin{equation*}
        \begin{split}
            | \partial_y F(y) | & = | e^{-|ky|} \frac{|ky| + 1}{y^2} | \leq \frac{2}{y^2},
        \end{split}
    \end{equation*}
    Lastly, the cancellation condition is immediately satisfied since $F$ is an odd kernel. Thus, $\mathfrak{J}_k$ extends continuously to a bounded map $L^1 \to L^{1,\infty}$ (meaning weak $L^1$). By interpolation, $\mathfrak{J}_k$ extends continuously to all $L^p$ spaces for $p \in (1,\infty)$, with bound    $$||\mathfrak{J}_k||_{L^p \to L^{p}} \lesssim  \max(p, (p-1)^{-1}),$$
    where the implicit constant is independent of $k$.
\end{proof}

We additionally note that because $\mathfrak{J}_k$ is a convolution operator, it commutes with $\partial_y$ when acting on $H^1$ functions. This is in contrast to the case of the half-plane and the infinite channel, where $\mathfrak{J}_k$ and $\partial_y$ have a non-zero commutator (see Lemmas \ref{commutator_estimate_of_J_k_half} and \ref{commutator_estimate_of_J_k}).

\section{Linearized Estimates}\label{linear_section}

We now work toward the proof of Proposition \ref{linear_proposition}. Computing the time derivative of the energy \eqref{kEnergy}, in the linear case where $\omega$ solves \eqref{linearized_system}, we have:
\begin{equation}\label{derivative_of_linear_energy}
    \begin{split}
        \frac{1}{2}\frac{d}{dt} E_k[\omega_k] &= \frac{1}{2} \biggl( \frac{d}{dt}||\omega_k||_2^2 + c_\alpha \alpha \frac{d}{dt} ||\partial_y \omega_k||_2^2 + c_\beta \beta \frac{d}{dt} \textrm{Re} \langle i k \omega_k, \partial_y \omega_k\rangle\\
        &\hphantom{= \frac{1}{2} \biggl(}+c_\tau \frac{d}{dt} \textrm{Re}\langle \omega_k, \mathfrak{J}_k[\omega_k]\rangle + c_\tau c_\alpha \alpha \frac{d}{dt} \textrm{Re} \langle \partial_y \omega_k, \mathfrak{J}_k[\partial_y \omega_k]\rangle \biggr)\\
        &\eqqcolon L_\gamma + L_\alpha + L_\beta + L_\tau + L_{\tau \alpha}.
    \end{split}
\end{equation}
We begin by collecting relevant estimates for $L_\gamma$, $L_\alpha$, and $L_\beta$.

\begin{lemma}\label{basic_hypo_est}

Under the hypotheses of Proposition \ref{linear_proposition}, the following estimates hold:
\begin{subequations}\label{hypo_lin}
    \begin{equation}\label{gamma_lin_est}
    \frac{1}{2} \frac{d}{dt} ||\omega_k||_2^2 + \nu ||\nabla_k \omega_k||_2^2 \leq 0,
\end{equation}
\begin{equation}\label{alpha_lin_est}
    \frac{1}{2} \frac{d}{dt} \alpha ||\nabla_k\partial_y\omega_k||_2^2  + \nu \alpha ||\nabla_k \partial_y \omega_k||_2^2 \lesssim D_{k,\gamma} +  \lambda_k||\omega_k||_2^2,
\end{equation}
\begin{equation}\label{beta_lin_est}
    -\frac{d}{dt} \beta \mathrm{Re} \langle ik \omega_k, \partial_y \omega_k \rangle + \lambda_k ||\omega_k||_2^2 \lesssim D_{k,\gamma}^{1/2} D_{k,\alpha}^{1/2}.
\end{equation}
\end{subequations}

\end{lemma}

\begin{proof}

The $L_\gamma$ estimate \eqref{gamma_lin_est} is proved in exactly the same manner as Lemma 4.1 in \cite{bedrossian2023stability}. Note that although \cite{bedrossian2023stability} poses the problem on $[-1,1]$ rather than $\R$, the formal computations remain the same. The proof of the $L_\alpha$ estimate \eqref{alpha_lin_est} at frequencies $|k| \geq \nu$ is the same as the proofs of Lemmas 4.2 in \cite{bedrossian2023stability}, and the proof in the $|k| \leq \nu$ case is almost identical if one makes the substitution $\nu^{2/3}|k|^{-2/3} \mapsto \alpha$. We demonstrate this principle in the following.

The $L_\beta$ estimate \eqref{beta_lin_est} requires slightly more modification from the proof exhibited for Lemma 4.3 of \cite{bedrossian2023stability}. Through integration by parts, we can directly compute
\begin{equation}\label{beta_comp}
    \begin{split}
        -\frac{d}{dt} \beta \textrm{Re} \langle ik \omega_k, \partial_y \omega_k \rangle =& - \beta |k|^2 || \omega_{k}||_2^2- \beta \textrm{Im}  k \langle \omega_k, \nu \partial_y \Delta_k \omega_k \rangle - \beta \textrm{Im}  k \langle \partial_y \omega_k, \nu \Delta_k \omega_k \rangle.
    \end{split}
\end{equation}
The first term on the right-hand side of \eqref{beta_comp} appears in the statement of the Lemma. The final two terms of \eqref{beta_comp} can be treated with integration by parts and Cauchy-Schwarz:
\begin{equation}
    \begin{split}
         |\beta \textrm{Im}  k \langle \omega_k, \nu \partial_y \Delta_k \omega_k \rangle + \beta \textrm{Im}  k \langle \partial_y \omega_k, \nu \Delta_k \omega_k \rangle| & \lesssim  \nu \beta |k| ||\nabla_k \omega_k||_2 ||\nabla_k \partial_y \omega_k||_2.
    \end{split}
\end{equation}
For $|k| \geq \nu$, we have
\begin{equation}
\begin{split}
        \nu \beta |k| ||\nabla_k \omega_k||_2 ||\nabla_k \partial_y \omega_k||_2 &= (\nu^{1/2}||\nabla_k \omega_k||_2) ( \nu^{1/2}|k|^{-1/3}||\nabla_k \nu^{1/3}\partial_y \omega_k||_2)\lesssim D_{k,\gamma}^{1/2} D_{k,\alpha}^{1/2}.
\end{split}
\end{equation}
On the other hand, when $|k| \leq \nu$, the estimate becomes
\begin{equation}
\begin{split}
\nu \beta |k| ||\nabla_k \omega_k||_2 ||\nabla_k \partial_y \omega_k||_2 &\leq \nu||\nabla_k \omega_k||_2 ||\nabla_k \partial_y \omega_k||_2\lesssim D_{k,\gamma}^{1/2} D_{k,\alpha}^{1/2}.
\end{split}
\end{equation}
\end{proof}

Next, we state the following estimates for the $L_\tau$ and $L_{\tau \alpha}$ terms. The proof of these estimates is simpler than the corresponding estimates in \cite{bedrossian2023stability}, primarily we are considering perturbations of true Couette flow, rather than perturbations of near-Couette flow. Our proof will not use the exact form of $G_k$, and hence will be applicable in the half-plane and infinite channel cases.

\begin{lemma}\label{disip_lin}
Under the hypotheses of Lemma \ref{linear_proposition},
\begin{subequations}
\begin{equation}\label{damping_of_phi}
    \frac{d}{dt} \mathrm{Re}\langle \omega_k, \mathfrak{J}_k \omega_k\rangle + \frac{1}{2}D_{k,\tau} \lesssim D_{k,\gamma},
\end{equation}
and
\begin{equation}\label{damping_of_dy_phi}
    \frac{d}{dt} \alpha \mathrm{Re}\langle \partial_y \omega_k, \mathfrak{J}_k \partial_y \omega_k\rangle + \frac{1}{2}D_{k,\tau\alpha} \lesssim D_{k,\alpha} + (\lambda_k ||\omega_k||_2^2)^{1/2}D_{k,\gamma}^{1/2}.
\end{equation}
\end{subequations}
\end{lemma}

We note that \eqref{damping_of_phi} gives the linearized inviscid damping estimate for both the Euler equations and the Navier Stokes equations. Indeed, by uniform boundedness of $\mathfrak{J}_k$ (Lemma \ref{boundedness_of_J_k}) and by taking $\nu \to 0$, we obtain the $L^2_t \dot{H}^1_x L_y^2$ linearized inviscid damping estimate:
$$\int_0^\infty |k|^2 ||\nabla_k \phi_k||_2^2 dt \approx \int_0^\infty D_{k,\tau} dt
 \lesssim ||\omega_k(0)||_2^2.$$
Additionally, we observe that while \eqref{damping_of_dy_phi} does not directly provide an inviscid damping estimate, it does produce the good term $D_{k,\tau\alpha}$, which will be useful in the nonlinear computations found in Section \ref{nonlin}

\begin{proof}
To prove Lemma \ref{disip_lin}, we first obtain the $D_\tau$ estimate by computing directly (with self-adjointness of $\mathfrak{J}_k$):
    \begin{equation}
    \begin{split}
        \frac{d}{dt} \textrm{Re} \langle \omega_k, \mathfrak{J}_k \omega_k \rangle &= \textrm{Re}\langle \frac{d}{dt} \omega_k, \mathfrak{J}_k \omega_k \rangle + \textrm{Re}\langle   \omega_k, \frac{d}{dt} \mathfrak{J}_k\omega_k \rangle\\
        &= \textrm{Re}\langle-iky \omega_k+ \nu \Delta_k \omega_k, 
 \mathfrak{J}_k \omega_k \rangle + \textrm{Re}\langle  \omega_k, \mathfrak{J}_k (-iky \omega_k + \nu \Delta_k \omega_k)\rangle\\
        &= 2 \textrm{Re} \langle - i k y \omega_k + \nu \Delta_k \omega_k, \mathfrak{J}_k \omega_k \rangle.
    \end{split}
\end{equation}
For the second term, we have by integration by parts and commutation between $\partial_y$ and $\mathfrak{J}_k:$
\begin{equation}
    \begin{split}
         |\langle \nu \Delta_k \omega_k, \mathfrak{J}_k \omega_k \rangle| = \nu |\langle \nabla_k \omega_k, \mathfrak{J}_k \nabla_k \omega_k\rangle | \lesssim D_{k,\gamma}.
    \end{split}
\end{equation}
For the first term, we use the definition of $\mathfrak{J}_k$ and the definition of $\phi_k$ as $\Delta_k \phi_k = \omega_k.$
\begin{equation}
\begin{split}
    \textrm{Re}\langle \mathfrak{J}_k (-iky \omega_k), \omega_k \rangle+\textrm{Re}\langle \mathfrak{J}_k  \omega_k, -iky \omega_k \rangle &= - \textrm{Re} \int_\R  |k|\textrm{p.v} \left(\int_\R \sgn(k)\frac{G_k(y,y')}{2(y-y')} k y' \omega_k(y') \overline{\omega_k}(y) dy'\right) dy\\
&\quad+\textrm{Re}\int_\R \left( |k|\textrm{p.v}\int_\R \sgn(k)\frac{G_k(y,y')}{2(y-y')}\omega_k(y') \overline{k y\omega_k}(y) dy' \right)dy\\
    &= \frac{|k|^2}{2} \int_\R  \textrm{p.v} \left(\int_{\R} \frac{G_k(y,y')}{y-y'}(y-y')\omega_k(y')  dy' \right)\overline{\omega_k(y)} dy\\
    &= \frac{|k|^2}{2} \langle \Delta_k^{-1} \omega_k, \omega_k \rangle = \frac{|k|^2}{2} \langle \phi_k, \Delta_k \phi_k\rangle\\
    &= - \frac{|k|^2}{2} \langle \nabla_k \phi_k, \nabla_k \phi_k\rangle.
\end{split}
\end{equation}
The $D_{\tau \alpha}$ estimate \eqref{damping_of_dy_phi} follows in a similar manner. We compute directly
    \begin{equation}
    \begin{split}
        \alpha \frac{d}{dt} \textrm{Re} \langle \partial_y \omega_k, \mathfrak{J}_k \partial_y \omega_k \rangle &= \alpha \textrm{Re}\langle \frac{d}{dt} \partial_y \omega_k, \mathfrak{J}_k \partial_y \omega_k \rangle + \textrm{Re}\langle   \partial_y \omega_k, \frac{d}{dt} \mathfrak{J}_k \partial_y \omega_k \rangle\\
        &= 2 \alpha\textrm{Re} \langle -i k y \partial_y \omega_k-ik \omega_k + \nu \Delta_k \partial_y\omega_k, \mathfrak{J}_k \partial_y \omega_k \rangle.
    \end{split}
\end{equation}
Then we recall that $\partial_y$ commutes with $\mathfrak{J}_k$. Hence, by the $D_\tau$ computations and boundedness of $\mathfrak{J}_k$,
\begin{equation}\label{almost_dy_Jk}
    \alpha  \frac{d}{dt} \textrm{Re} \langle \partial_y \omega_k, \mathfrak{J}_k \partial_y \omega_k \rangle + \frac{1}{2}D_{k,\tau \alpha} \lesssim \alpha |k| ||\omega_k||_2||\partial_y \omega_k||_2 + D_{k,\alpha}.
\end{equation}
The estimate \eqref{almost_dy_Jk} proves \eqref{damping_of_dy_phi} upon noticing that
$$\alpha|k| ||\omega_k||_2 ||\partial_y \omega_k||_2 \lesssim \left(\lambda_k ||\omega_k||_2^2\right)^{1/2}D_{k,\gamma}^{1/2}.$$
\end{proof}

We are now prepared to prove Proposition \ref{linear_proposition}.

\subsection{Proof of Linear case}

\begin{proof} Let $K_0$ denote twice the maximal implicit constant from Lemmas \ref{basic_hypo_est} and \ref{disip_lin}, and note that $K_0$ is independent of $\nu$. By Eq. \eqref{derivative_of_linear_energy}:
\begin{equation*}
    \begin{split}
        \frac{d}{dt}E_k + c_\beta\lambda_k &||\omega_k||_2^2 + 2 D_{k,\gamma} + c_{\alpha} 2 D_{k,\alpha} + \frac{c_\tau}{2} D_{k,\tau} + \frac{c_\alpha c_\tau}{2} D_{k,\tau \alpha}\\
        &\leq c_\alpha K_0 D_{k,\gamma} + c_\alpha K_0 \lambda_k ||\omega_k||_2^2 + c_\beta K_0 D_{k,\gamma}^{1/2} D_{k,\alpha}^{1/2}\\ &\quad+ K_0 c_\tau D_{k,\gamma} + K_0 c_{\tau}c_{\alpha} D_{k,\alpha} + K_0 c_{\tau} c_{\alpha} (\lambda_k ||\omega_k||_2^2)^{1/2} D_{k,\gamma}^{1/2}.
        \end{split}
\end{equation*}
Taking, without loss of generality, $K_0 \geq 32(1 + ||\mathfrak{J}_k||_{L^2 \to L^2})$, we set the conditions that
\begin{equation}\label{coeff_conditions}
    c_\tau < \frac{1}{32 K_0}, \; c_\alpha < \min \left(\frac{1}{8 K_0}, 1\right), \; \frac{c_\alpha}{c_\beta} < \frac{1}{25 K_0}, \; \frac{c_\beta^2}{2 c_\alpha} < \frac{1}{25K_0^2}.
\end{equation}
Similar conditions were set in \cite{bedrossian2023stability}. Then by Young's product inequality:
\begin{subequations}
    \begin{equation*}
        K_0 c_\beta D_{k,\gamma}^{1/2} D_{k,\alpha}^{1/2} \leq \frac{1}{10} D_{k,\gamma} + 
        \frac{1}{5}c_\alpha D_{k,\alpha},
    \end{equation*}
    and
    \begin{equation*}
        K_0 c_\tau c_\alpha (\lambda_k ||\omega_k||_2^2)^{1/2} D_{k,\lambda}^{1/2} \leq \frac{1}{10} c_\beta \lambda_k ||\omega_k||_2^2 + \frac{5}{2} \frac{c_\tau^2 c_\alpha^2 K_0^2}{c_\beta} D_{k,\gamma} \leq \frac{c_\beta}{10} \lambda_k ||\omega_k||_2^2 + \frac{1}{10} D_{k,\gamma}.
    \end{equation*}
\end{subequations}
Now using the conditions on the coefficients set in \eqref{coeff_conditions}, we have:
\begin{equation*}
    \begin{split}
        \frac{d}{dt}E_k &+ c_\beta\lambda_k ||\omega_k||_2^2 + 2 D_{k,\gamma} + c_{\alpha} 2 D_{k,\alpha} + \frac{c_\tau}{2} D_{k,\tau} + \frac{c_\alpha c_\tau}{2} D_{k,\tau \alpha}\\
        &\leq \frac{1}{8} D_{k,\gamma} + \frac{c_\beta}{8} \lambda_k ||\omega_k||_2^2 + \frac{1}{10} D_{k,\gamma} + 
        \frac{1}{5}c_\alpha D_{k,\alpha} + \frac{1}{32} D_{k,\gamma} + \frac{1}{32}c_{\alpha} D_{k,\alpha} + \frac{c_\beta}{10} \lambda_k ||\omega_k||_2^2 + \frac{1}{10} D_{k,\gamma},
        \end{split}
\end{equation*}
which simplifies to
\begin{equation*}
    \begin{split}
        \frac{d}{dt}E_k &+ c_\beta \lambda_k \left(1 -\frac{1}{8} - \frac{1}{10}\right)||\omega_k||_2^2 + \left(2 - \frac{1}{8} - \frac{1}{5}\right) D_{k,\gamma} + c_{\alpha}\left(2 - \frac{1}{5} - \frac{1}{32}\right) D_{k,\alpha}\\
        &+ c_\tau\left(\frac{1}{2} - \frac{1}{32} \right) D_{k,\tau} + c_\alpha c_\tau\frac{1}{2} D_{k,\tau \alpha} \leq 0.
        \end{split}
\end{equation*}

Note that $||\omega_k||_2 + \alpha ||\partial_y \omega_k||_2 \approx E_k$ so long as we additionally impose $c_\beta^{2} \leq \frac{1}{4}c_\alpha +\frac{1}{4}(1-c_\tau)$. Furthermore, we observe that $\alpha \lambda_k \leq \nu$, and so $D_{k,\gamma} \geq \nu||\partial_y\omega_k||_2^2 \geq \lambda_k \alpha ||\partial_y \omega_k||_2^2$. Thus, with the final assumption that $c_\beta \leq 1$, we obtain
\begin{equation}
    \begin{split}
        0 &\geq \frac{d}{dt}E_k + \frac{c_\beta}{2}\lambda_k\left(||\omega_k||_2^2 + \alpha ||\partial_y \omega_k||_2^2\right) + \frac{1}{4}D_{k,\gamma} + \frac{c_\alpha}{2}D_{k,\alpha} + \frac{c_\beta}{4}D_{k,\beta} + \frac{c_\tau}{4}D_{k,\tau} + \frac{c_\alpha c_\tau}{2} D_{k,\tau}\\
        &\geq \frac{d}{dt}E_k + c_0(c_\tau, c_\alpha, c_\beta) \lambda_k E_k + \frac{1}{4} D_{k},
    \end{split}
\end{equation}
which is the desired estimate. This is satisfied so long as $c_\alpha$, $c_\beta$, and $c_\tau$ meet all stated assumptions. The exact constant $c_0$ depends on the proportionality constants between $E_k$ and $||\omega_k||_2 + \alpha ||\partial_y \omega_k||_2$, which is in turn determined by the choice of coefficients in the energy. Note that none of the constants depend on $k$ or $\nu$. An example of explicit choices of constants satisfying all necessary assumptions is given in Section 4.5 of \cite{bedrossian2023stability}.
\end{proof}

\subsection{Discussion of the Linearized \texorpdfstring{$\mathfrak{b}$}{beta}-plane}\label{lin_beta_plane}
The linearization of \eqref{beta_system} is given by
\begin{equation}\label{beta_linearized_system}
    \begin{cases}
        \partial_t \omega_k + ik y \omega_k + \mathfrak{b} ik \phi_k - \nu \Delta_k \omega_k = 0,\\
        \Delta_k \phi_k = \omega_k.
    \end{cases}
\end{equation}
Now if $\omega_k$ solves \eqref{beta_linearized_system}, then the only new terms in $\frac{d}{dt}E_k$ are those arising from $ \mathfrak{b} \phi_k$. However, we observe that through integration by parts
\begin{equation}
\begin{split}
    &\mathrm{Re} \langle \omega_k, ik \phi_k \rangle =  -\mathrm{Re} \langle \nabla_k \phi_k, ik \nabla_k \phi_k \rangle = 0,\\
    &\mathrm{Re} \langle \partial_y \omega_k, ik \partial_y \phi_k \rangle =  -\mathrm{Re} \langle \nabla_k \partial_y \phi_k, ik \nabla_k \partial_y \phi_k \rangle = 0,\\
    &\mathrm{Re}\langle ik \omega_k, ik\partial_y\phi_k \rangle + \mathrm{Re}\langle (ik)^2 \phi_k, \partial_y \omega_k\rangle = \mathrm{Re}\langle (ik)^2 \nabla_k \phi_k, \nabla_k\partial_y \phi_k \rangle - \mathrm{Re}\langle (ik)^2 \nabla_k \phi_k, \nabla_k \partial_y \phi_k\rangle = 0,
\end{split}
\end{equation}
so that Lemma \ref{hypo_lin} is unchanged if $\omega_k$ solves \eqref{beta_linearized_system} rather than \eqref{linearized_system}. Similarly, Lemma \ref{disip_lin} remains unchanged, since by integration by parts:
\begin{equation}
    \begin{split}
        & \mathrm{Re}\langle \mathfrak{J}_k \omega_k, ik \phi_k \rangle + \mathrm{Re}\langle \mathfrak{J}_k ik \phi_k, \omega_k \rangle = -\mathrm{Re}\langle \mathfrak{J}_k  \nabla_k \phi_k , ik \nabla_k \phi_k\rangle + \mathrm{Re}\langle \mathfrak{J}_k  \nabla_k \phi_k, ik \nabla_k \phi_k \rangle = 0,
    \end{split}
\end{equation}
and similarly
$$\mathrm{Re}\langle \mathfrak{J}_k \partial_y \omega_k, ik \partial_y \phi_k \rangle + \mathrm{Re}\langle \mathfrak{J}_k ik \partial_y \phi_k, \partial_y \omega_k \rangle = 0.$$
Note that we have crucially used the fact that $\mathfrak{J}_k$ and $\partial_y$ commute in the planar case. Thus, the $\mathfrak{b}ik \phi_k$ term vanishes from all estimates at the linear level. Furthermore, the nonlinearity for \eqref{beta_system} is the standard Euler nonlinearity in \eqref{couette_system}. Thus, the nonlinear argument in Section \ref{nonlin} is entirely unaffected, and the same proof used to establish Theorem \ref{main_theorem} applies equally to Corollary \ref{beta_corollary}.

\section{Technical Lemmas}\label{technical_lemmas}

Here we collect a series of minor technical lemmas used throughout the proof of Lemma \ref{bootstrap} in Section \ref{nonlin}. The proofs rely on basic inequalities of analysis together with the definitions \eqref{introduce_energy} and \eqref{different_dissipations_definition}, and in several cases are simple extensions of lemmas found in Section 5.1 of \cite{bedrossian2023stability}. However, there are additional aspects relating to the control of low frequencies terms.

We begin with two lemmas which enable us to control integrals of terms involving $||k \phi_k||_\infty$ and $||\partial_y \phi_k||_\infty$ in terms of the dissipations $\mathcal{D}_\tau$ and $\mathcal{D}_{\tau \alpha}$, respectively.

\begin{lemma}\label{keyphiestimate}
    For any $\theta \in (0,1/2]$, the following holds:
    \begin{subequations}
    \begin{equation}\label{all_freqs}
        \int_{\R} \left(\int_0^t \langle  c \lambda_{k} s\rangle^{2J}||k^{1+\theta} \phi_k||_\infty^2 ds \right)^{1/2} dk \lesssim_\theta \mathcal{D}_\tau^{1/2}.
    \end{equation}
    In place of the $\theta = 0$ case, we have the following:
    \begin{equation}\label{some_freqs}
        \int_{|k| \gtrsim \nu} \left(\int_0^t \langle  c \lambda_{k} s\rangle^{2J}||k \phi_k||_\infty^2 ds\right)^{1/2} dk \lesssim (1+\ln(1/\nu)^{1/2})\mathcal{D}_\tau^{1/2}.
    \end{equation}
    Additionally,
\begin{equation}\label{variant_of_key_phi_estimate}
        \left(\int_0^t \int_{\R} \langle  c \lambda_{k} s\rangle^{2J} \langle k \rangle^{2m}|k|^3|| \phi_k||_\infty^2 dk ds \right)^{1/2} \lesssim \mathcal{D}_\tau^{1/2}.
    \end{equation}
    \end{subequations}
\end{lemma}

\begin{proof}

    First, we note that if we split the integral in \eqref{all_freqs} into $|k| \geq 1$ and $|k| \leq 1$, then the $|k| \geq 1$ terms obey (for any $\theta \in (0,1/2]$):
\begin{equation}\label{integrating_high}
        \begin{split}       
        \int_{|k| \geq 1} \left(\int_0^t \langle  c \lambda_{k} s\rangle^{2J}||k^{1+\theta} \phi_k||_\infty^2 ds \right)^{1/2} dk&\lesssim \int_{|k| \geq 1} \left(\int_0^t \langle  c \lambda_{k} s\rangle^{2J}||k^{3/2} \phi_k||_\infty^2 ds \right)^{1/2}dk\\
        &\lesssim \int_{|k| \geq 1} \left( \int_0^t \langle  c \lambda_{k} s\rangle^{2J}|k|^{3/2}||\partial_y\phi_k||_2||\phi_k||_2 ds\right)^{1/2} dk\\
        &\lesssim \int_{|k| \geq 1} \left( \int_0^t \langle  c \lambda_{k} s\rangle^{2J} |k|||\nabla_k \phi_k||_2^2 \right)^{1/2} dk\\
        &\lesssim \left(\int_{|k| \geq 1} \langle k \rangle^{-2m}dk\right)^{1/2} \left(\int_0^t \int_{|k| \geq 1}\langle  c \lambda_{k} s\rangle^{2J} \langle k\rangle^{2m} |k|^2 ||\nabla_k \phi_k||_2 dk\right)^{1/2}\\
        &\lesssim \mathcal{D}_\tau^{1/2},
        \end{split}
    \end{equation}
    where we have used Gagliardo-Nirenberg-Sobolev and that $m > 1/2$. Next, we integrate over frequencies $|k| \leq 1$:
    \begin{equation}\label{eta_eqn}
        \begin{split}       
        \int_{|k|\leq 1} \left(\int_0^t\langle  c \lambda_{k} s\rangle^{2J} |||k|^{1+\theta} \phi_k||_\infty^2 ds \right)^{1/2} dk& \lesssim \int_{|k|\leq 1} \left(\int_0^t \langle  c \lambda_{k} s\rangle^{2J}|k|^{3 + 2\theta - 1}||\partial_y\phi_k||_2||\phi_k||_2^{1/2} ds \right)^{1/2} dk\\
        &\lesssim \int_{|k| \leq 1} |k|^{\theta - 1/2} \left(\int_0^t\langle  c \lambda_{k} s\rangle^{2J}|k|^2||\nabla_k \phi_k||_2^2 ds\right)^{1/2} dk\\
        &\lesssim \left(\int_{|k| \leq 1} |k|^{2 \theta-1} dk\right)^{1/2} \left(\int_0^t\int_{|k| \leq 1}  \langle  c \lambda_{k} t\rangle^{2J}|k|^2||\nabla_k \phi_k||_2^2 dk ds \right)^{1/2}\\
        &\lesssim \left(\int_{|k| \leq 1} |k|^{2 \theta-1}dk \right)^{1/2} \mathcal{D}_\tau^{1/2}.
        \end{split}
    \end{equation}
    Integrating the final factor in \eqref{eta_eqn}, depending on $\theta$, completes the proof of \eqref{all_freqs}. For the $\theta = 0$ estimate \eqref{some_freqs}, we split into frequencies $|k| \geq 1$ and $\nu \lesssim |k| \leq 1$. The $|k| \geq 1$ estimates proceeds exactly as in \eqref{integrating_high}. Meanwhile while the $\nu \lesssim |k| \leq 1$ estimate is similar to \eqref{eta_eqn}, but only integrating over frequencies $\nu \lesssim |k| \leq 1$, giving rise to the factor of $\ln(1/\nu)^{1/2}.$ Lastly, we note that the the variant \eqref{variant_of_key_phi_estimate} is proved in a similar but simpler manner to the above, relying on Gagliardo-Nirenberg-Sobolev. There is no need to employ H\"older's inequality or introduce additional factors in this case.
\end{proof}

A similar proof to that of Lemma \ref{keyphiestimate} gives the related lemma:

\begin{lemma}\label{keydyphiestimate}
    For any $\theta \in (0,1/2]$, the following holds:
    \begin{equation}
        \int_{|k| \gtrsim \nu} \left(\int_0^t \langle  c \lambda_{k} s\rangle^{2J} ||k^{2/3+\theta} \partial_y\phi_k||_\infty^2ds\right)^{1/2} dk \lesssim_\theta \nu^{-1/3}\mathcal{D}_{\tau\alpha}^{1/2}.
    \end{equation}
    In place of the $\theta = 0$ case, we have the following:
    \begin{equation}
        \int_{|k| \gtrsim \nu} \left(\int_0^t \langle  c \lambda_{k} t\rangle^{2J} ||k^{2/3} \partial_y\phi_k||_\infty^2ds \right)^{1/2} dk \lesssim \nu^{-1/3}(1+\ln(1/\nu)^{1/2})\mathcal{D}_{\tau\alpha}^{1/2}.
    \end{equation}
    Additionally,
    \begin{equation}
        \left(\int_0^t \int_{|k| \gtrsim \nu} \langle  c \lambda_{k} t\rangle^{2J} \langle k \rangle^{2m} ||k^{2/3} \partial_y\phi_k||_\infty^2 dk ds\right)^{1/2} \lesssim \nu^{-1/3}\mathcal{D}_{\tau\alpha}^{1/2}.
    \end{equation}

\end{lemma}
Note that in principle, there exist variants of Lemma \ref{keydyphiestimate} where the domain of integration includes low frequencies, meaning $|k| \lesssim \nu$. However, we will not have occasion to use these estimates in this paper.\\

Recalling the definition of $A(k) = \sqrt{\alpha}$, we state the following lemma, which is a variant of Lemma 5.2 from \cite{bedrossian2023stability}.

\begin{lemma}\label{beta_and_gamma}

    There hold the following estimates:

    \begin{subequations}
        \begin{equation}\label{beta_gamma_1}
            \int_0^t \int_{|k| \gtrsim \nu}  \langle  c \lambda_{k} t\rangle^{2J} \langle k\rangle^{2m} |k| ||\omega_k||_2^2 dk ds\lesssim \nu^{-1/2} \mathcal{D}_\gamma^{1/2}\mathcal{D}_\beta^{1/2},
        \end{equation}
        and
        \begin{equation}\label{beta_gamma_2}
            \int_0^t \int_{\R} \langle  c \lambda_{k} t\rangle^{2J} \langle k\rangle^{2m} \left(A(k) |k|\right)^{2} ||\omega_k||_2^2 dk ds \lesssim \mathcal{D}_\gamma^{1/2} \mathcal{D}_\beta^{1/2} + \nu \mathcal{D}_\beta.
        \end{equation}
    \end{subequations}
\end{lemma}

\begin{proof}
    This lemma is an extension of Lemma 5.2 from \cite{bedrossian2023stability}, wherein the authors included (essentially) a proof of \eqref{beta_gamma_1} at frequencies $|k| \geq 1$ via a simple interpolation trick. This same interpolation still works over the larger range $|k| \gtrsim \nu$. We demonstrate this by using a similar trick in the proof of \eqref{beta_gamma_2}. For $|k| \gtrsim \nu$, $A(k) = \nu^{1/3} |k|^{-1/3}$, and so
    \begin{equation}
        \begin{split}
            \int_0^t\int_{|k| \gtrsim \nu}& \langle  c \lambda_{k} t\rangle^{2J} \langle k\rangle^{2m} \nu^{2/3} |k|^{4/3} ||\omega_k||_2^2 dk\\  &\lesssim \int_0^t \int_{|k| \gtrsim \nu} \left(\langle  c \lambda_{k} t\rangle^{J} \langle k\rangle^{m} \nu^{1/2}|k| ||\omega_k||_2 \right)\left(\langle  c \lambda_{k} t\rangle^{J} \langle k\rangle^{m} \nu^{1/6}|k|^{1/3} ||\omega_k||_2\right) dk ds\\
            &\lesssim \left(\int_0^t \int_{|k| \gtrsim \nu} \langle  c \lambda_{k} t\rangle^{2J} \langle k\rangle^{2m} \nu ||k\omega_k||_2^2 dk ds\right)^{1/2}\left(\int_0^t \int_{|k| \gtrsim \nu} \langle  c \lambda_{k} t\rangle^{2J} \langle k\rangle^{2m}  \nu^{1/3}|k|^{2/3} ||\omega_k||_2^2 dk ds\right)^{1/2}\\
            &\lesssim \mathcal{D}_\gamma^{1/2} \mathcal{D}_\beta^{1/2}.
        \end{split}
    \end{equation}
    Meanwhile, \eqref{beta_gamma_2} at frequencies $|k| \lesssim \nu$ follows immediately from the definition of $\mathcal{D}_\beta$ \eqref{different_dissipations_definition} in terms of $D_{\beta,k}$ \eqref{k_by_k_dissipation}.
\end{proof}

\begin{lemma}\label{dy_dy_phi_to_energy}

The following functional estimates hold.
$$ \int_{|k| \lesssim \nu} \langle  c \lambda_{k} t\rangle^{J}||\partial_y^2 \phi_k||_\infty dk \lesssim \nu^{1/2}\mathcal{E}^{1/2},$$
$$ \left(\int_0^t \left(\int_{\R} \langle  c \lambda_{k} t\rangle^{J} ||\partial_y^2 \phi_k||_\infty dk \right)^4 ds \right)^{1/4} \lesssim \nu^{-1/4}\D_\gamma^{1/4}\sup_{s \in [0,t]}\mathcal{E}^{1/4}.$$
\end{lemma}

\begin{proof}
    The first estimate is immediate by Gagliardo-Nirenberg-Sobolev, elliptic regularity, and H\"older:
    \begin{equation}
        \begin{split}
            \int_{|k| \lesssim \nu} \langle  c \lambda_{k} t\rangle^{J}||\partial_y^2 \phi_k||_\infty dk &\lesssim \nu^{1/2}\left(\int_{|k| \lesssim \nu} ||\partial_y^2 \phi_k||_\infty^2 dk\right)^{1/2} \\
           &\lesssim \nu^{1/2}\left(\int_{|k| \lesssim \nu} ||\partial_y \omega_k||_2^2 + ||\omega_k||_2^2 dk\right)^{1/2} \lesssim \nu^{1/2} \mathcal{E}^{1/2}.
        \end{split}
    \end{equation}
    The second estimate is a simple interpolation argument. By elliptic regularity, Gagliardo-Nirenberg-Sobolev, H\"older, and $m > 1/2$:
    \begin{equation}
        \begin{split}
            \left( \int_0^t \left(\int_{\R} \langle  c \lambda_{k} t\rangle^{J}||\partial_y^2 \phi_k||_\infty dk \right)^{4} ds \right)^{1/4} &\lesssim \left( \int_0^t \left(\int_{\R} \langle  c \lambda_{k} t\rangle^{J}||\partial_y \omega_k||_2^{1/2} ||\omega_k||_2^2 dk \right)^{4} ds \right)^{1/4} \\
            &\lesssim \left( \int_0^t \biggl(\int_{\R}  \langle k \rangle^{-m} \biggl( \langle  c \lambda_{k} t\rangle^{J/2} ||\langle k \rangle^m \nabla_k \omega_k||_2^{1/2}\right)\\
            &\quad \quad \left( \langle c \lambda_{k} t\rangle^{J/2} ||  \langle k \rangle^m \omega_k||_2^{1/2} \right)dk \biggr)^{4} ds \biggr)^{1/4}\\
            &\lesssim \nu^{-1/4} \D_\gamma^{1/4} \sup_{s \in [0,t]}\E^{1/4}(s).
        \end{split}
    \end{equation}
\end{proof}
Additionally, we have the following lemmas related to the control of the stream function at low frequencies.

\begin{lemma}\label{low_frequency_lemma_phi}

Fix $p \in [0,1/2]$. Then the following holds:
\begin{subequations}
\begin{equation}\label{low_freq_first}
    \int_{|k| \lesssim \nu} ||\partial_y \phi_k||_\infty dk \lesssim \nu^{1/2} \E^{1/2}.
\end{equation}
Additionally,
\begin{equation}\label{low_freq_second}
    \int_{|k| \lesssim \nu} \left(\int_0^t ||k \phi_k||_\infty^2 ds \right)^{1/2} dk \lesssim \nu^{1/2}\mathcal{D}_{\tau,2}^{1/2}.
\end{equation}
Lastly, we note the similar results
\begin{equation}\label{low_freq_third}
    \int_{|k| \leq 1} ||\partial_y \phi_k||_\infty dk \lesssim  \E_2^{1/2},  \quad \int_{|k| \leq 1} \left(\int_0^t ||k \phi_k||_\infty^2 ds \right)^{1/2} dk \lesssim \mathcal{D}_{\tau,2}^{1/2}.
\end{equation}
\end{subequations}
\end{lemma}

\begin{proof}
    Using Gagliardo-Nirenberg-Sobolev, introducing a factor $|k|^{-1/2} |k|^{1/2}$ and applying elliptic regularity, the integramd of \eqref{low_freq_first} satsifies:
    \[
            ||\partial_y \phi_k||_\infty \lesssim |k|^{-1/2} |k|^{1/2} ||\partial_y^2 \phi_k||_2^{1/2} ||\partial_y \phi_k||_2^{1/2} \lesssim |k|^{-1/2} ||\omega_k||_2 .\]
    Then \eqref{low_freq_first} follows by H\"older's inequality. The proof of \eqref{low_freq_second} is similarly straightforward. By Gagliardo-Nirenberg-Sobolev and H\"older:
    \begin{equation}
        \begin{split}
            \int_{|k| \lesssim \nu} \left(\int_0^t ||k \phi_k||_\infty^2 ds \right)^{1/2} dk &\lesssim \int_{|k| \lesssim \nu} |k|^{-1/2} |k| \left(\int_0^t ||\partial_y \phi_k||_2 ||k \phi_k||_2 \right)^{1/2} dk\\
            &\lesssim \left(\int_{|k|\lesssim \nu} |k|^{-1/2} dk\right) \sup_{|k| \lesssim \nu}\left(\int_0^t |k|^2 ||\nabla_k \phi_k||_2^2 ds\right)^{1/2}\\
            &\lesssim \nu^{1/2}\mathcal{D}_{\tau,2}^{1/2}.
        \end{split}
    \end{equation}
    The proof of \eqref{low_freq_third} follows the above arguments, differing only in the domain of integration, which removes the factors of $\nu$ from the result.
\end{proof}

\section{Non-Linear Estimates}\label{nonlin}

The proof of Lemma \ref{nonlin} consists of combining the linear results in Proposition \ref{linear_proposition} with additional bounds on nonlinear terms proved in this section.  To make this distinction clearer, for $\omega_k$ solving \eqref{couette_system}, we define
\begin{subequations}
    \begin{equation}\label{lin_notation}
        \mathbb{L}_k \coloneqq  -ik y \omega_k + \nu \Delta_k \omega_k,
    \end{equation}
    \begin{equation}\label{nonlin_notation}
        \mathbb{NL}_k \coloneqq -(\nabla^{\perp} \phi \cdot \nabla \omega)_k = -\int_{\R} \nabla_{k-k'}^{\perp} \phi_{k-k'} \cdot \nabla_{k'} \omega_{k'} dk',
    \end{equation}
\end{subequations}
so that $\partial_t \omega_k = \mathbb{L}_k + \mathbb{NL}_k$. By the fundamental theorem of calculus and the triangle inequality, for any $t > 0$,
\begin{equation}\label{fotc_energy}
    \begin{split}
        \E(t) + \frac{1}{2} \D_2(t) &\leq \E_1(0) + \int_0^t \frac{d}{ds} \E_1(s) ds + \sup_k\left(\langle (I + c_\tau \mathfrak{J}_k)\omega_k, \omega_k\rangle(0) + \int_0^t \frac{d}{ds}\langle (I + c_\tau \mathfrak{J}_k)\omega_k, \omega_k\rangle(s) ds \right)\\
        &\quad \quad +   \sup_{k} \left(\int_0^t 2 D_{k,\gamma}(s) + \frac{1}{2}c_\tau D_{k,\tau}(s) \right) ds\\
        &\leq \E_1(0) + \int_0^t \frac{d}{ds} \E_1(s) ds + 2\sup_k\biggl(E_k(0)  + \int_0^t  \frac{d}{ds}\langle (I + c_\tau \mathfrak{J}_k)\omega_k, \omega_k\rangle(s) ds\\
        &\quad \quad \quad \quad \quad + \int_0^t   2 D_{k,\gamma}(s) + \frac{1}{2}c_\tau D_{k,\tau}(s) ds \biggr)\\
        &\leq 2\E(0) + \int_0^t \frac{d}{ds} \E_1(s) ds\\
        &\quad \quad + 2\sup_k \left| \int_0^t \frac{d}{ds}\langle (I + c_\tau \mathfrak{J}_k)\omega_k, \omega_k\rangle(s) +  2 D_{k,\gamma}(s) + \frac{1}{2}c_\tau D_{k,\tau}(s)  ds\right|.
    \end{split}
\end{equation}
We begin by addressing the time derivative falling on $\E_1$. Equipped with the notation \eqref{lin_notation} and \eqref{nonlin_notation}, we split $\frac{d}{ds}\mathcal{E}_1(s)$ into linear and non-linear contributions, plus the factors arising when the time derivative falls on the time-dependent multipliers:
\begin{equation}\label{time_deriv_of_full_energy}
    \frac{d}{ds}\mathcal{E}_1(s) = \mathcal{L}_1(s) + \mathcal{NL}_1(s) + \int_{\R} J c \lambda_k \frac{2c\lambda_k s}{\langle c \lambda_k s \rangle^2}\frac{\langle c\lambda_k s \rangle^{2J}}{M_k(s)}\langle k \rangle^{2m} E_k[\omega_k] dk- \int_\R \frac{\dot{M}_k(t)}{M_k(s)}\frac{\langle c\lambda_k s \rangle^{2J}}{M_k(s)}\langle k \rangle^{2m} E_k[\omega_k] dk,
\end{equation}
where we denote
\begin{equation}
\begin{split}
        \mathcal{L}_1 \coloneqq & \int_{\R} \frac{\langle c \lambda_k s \rangle^{2J}}{M_k(s)}\langle k \rangle^{2m}\biggl(  2\textrm{Re} \langle \omega_k, (I+c_\tau \mathfrak{J}_k) \mathbb{L}_k \rangle \\
         &\quad+ 2 c_\alpha\alpha \textrm{Re}\langle \partial_y \omega_k, (I + c_\tau \mathfrak{J}_k) \mathbb{L}_k\rangle- c_\beta \beta \left(\textrm{Re} \langle ik \omega_k, \partial_y \mathbb{L}_k\rangle + \textrm{Re}\langle ik \mathbb{L}_k, \partial_y \omega_k\rangle\right) \biggr) dk,\\
        \mathcal{NL}_1
        \coloneqq& \int_{\R} \frac{\langle c \lambda_k s \rangle^{2J}}{M_k(s)}\langle k \rangle^{2m}\biggl(  2\textrm{Re} \langle \omega_k, (I+c_\tau \mathfrak{J}_k) \mathbb{NL}_k \rangle + 2 c_\alpha\alpha \textrm{Re}\langle \partial_y \omega_k, (I + c_\tau \mathfrak{J}_k) \mathbb{NL}_k\rangle\\
    &\quad- c_\beta \beta \left(\textrm{Re} \langle ik \omega_k, \partial_y \mathbb{NL}_k\rangle + \textrm{Re}\langle ik \mathbb{NL}_k, \partial_y \omega_k\rangle\right) \biggr) dk.
\end{split}
\end{equation}
We define for later use
\begin{equation}
    \begin{split}
                T_{\gamma, \tau} &\coloneqq \int_0^t \int_{\R} \frac{\langle c \lambda_k s \rangle^{2J}}{M_k(s)}\langle k \rangle^{2m}  2\textrm{Re} \langle \omega_k, (I+c_\tau \mathfrak{J}_k)\mathbb{NL}_k \rangle dk ds,\\
                T_{\alpha, \tau\alpha} &\coloneqq \int_0^t\int_{\R} \frac{\langle c \lambda_k s \rangle^{2J}}{M_k(s)}\langle k \rangle^{2m}  2 c_\alpha\alpha \textrm{Re}\langle \partial_y \omega_k, (I + c_\tau \mathfrak{J}_k) \mathbb{NL}_k\rangle dk ds, \\
                T_\beta &\coloneqq - \int_0^t \int_{\R} \frac{\langle c \lambda_k s \rangle^{2J}}{M_k(s)}\langle k \rangle^{2m}c_\beta \beta \left(\textrm{Re} \langle ik \omega_k, \partial_y \mathbb{L}_k\rangle + \textrm{Re}\langle ik \mathbb{NL}_k, \partial_y \omega_k\rangle\right)dk ds,
    \end{split}
\end{equation}
so that $\int_0^t \mathcal{NL}_1 ds = T_{\gamma, \tau} + T_{\alpha, \tau \alpha} + T_\beta$. By Proposition \ref{linear_proposition}, 
\begin{equation}\label{implication_from_prop}
    \int_0^t \mathcal{L}_1(s) dt \leq -8 c \mathcal{D}_1(t) -\int_0^t\int_\R 8 c \lambda_k\frac{\langle c \lambda_k s \rangle^{2J}}{M_k(s)} \langle k \rangle^{2m} E_k[\omega_k] dk ds.
\end{equation}
We also note that by Young's product inequality and the defining ODE \eqref{definition_of_Mk(t)} for $\dot{M}_k(t)$ :
\begin{equation}\label{usage_of_Mk(t)}
\begin{split}
         \int_{\R} J c \lambda_k \frac{2c\lambda_k s}{\langle c \lambda_k s \rangle^2}\frac{\langle c\lambda_k s \rangle^{2J}}{M_k(s)}\langle k \rangle^{2m} E_k[\omega_k] dk &\leq \int_\R c \lambda_k\frac{\langle c\lambda_k s \rangle^{2J}}{M_k(s)}\langle k \rangle^{2m} E_k[\omega_k] dk\\
         &\hphantom{\leq \int_\R}+ \int_\R  c \lambda_k J^2 \frac{(c\lambda_k s)^2}{\langle c \lambda_k s \rangle^4}\frac{\langle c\lambda_k s \rangle^{2J}}{M_k(s)}\langle k \rangle^{2m} E_k[\omega_k]dk\\
         &= \int_\R c \lambda_k\frac{\langle c\lambda_k s \rangle^{2J}}{M_k(s)}\langle k \rangle^{2m} E_k[\omega_k] dk\\
         &\hphantom{\leq \int_\R}+ \int_\R  \frac{\dot{M}_k(s)}{M_k(s)}\frac{\langle c\lambda_k s \rangle^{2J}}{M_k(s)}\langle k \rangle^{2m} E_k[\omega_k] dk.
\end{split}
\end{equation}
Combining \eqref{implication_from_prop}, \eqref{time_deriv_of_full_energy}, and \eqref{usage_of_Mk(t)}, we observe that
\begin{equation}\label{reduced_diffble_energy}
    \int_0^t \frac{d}{ds} \mathcal{E}_1(s) \leq - 4c \mathcal{D}_1(s) + \int_0^t \mathcal{NL}_1(s).
\end{equation}
Applying \eqref{reduced_diffble_energy} to \eqref{fotc_energy}, we find
\begin{equation}\label{pre_abs_val}
    \begin{split}
        \E(t) + \frac{1}{2}\D_2(t) &\leq \E(0)  - 4c \D_1(s) + \int_0^t \mathcal{NL}_1(s) ds\\
        &\quad \quad + 2\sup_k \left| \int_0^t \frac{d}{ds}\langle (I + c_\tau \mathfrak{J}_k)\omega_k, \omega_k\rangle(s) +  2 D_{k,\gamma}(s) + \frac{1}{2}c_\tau D_{k,\tau}(s)  ds\right|.
    \end{split}
\end{equation}
Next, recall \eqref{gamma_lin_est} and \eqref{damping_of_phi}, which alongside \eqref{coeff_conditions} and boundedness of $\mathfrak{J}_k$ give
\begin{equation}\label{abs_val_eqn}
\begin{split}
    \left|\frac{d}{ds}\langle (I + c_\tau \mathfrak{J}_k)\omega_k, \omega_k\rangle(s) + 2\left(D_{k,\gamma}(s) + \frac{1}{2}c_\tau D_{k,\tau}(s) \right)\right| \leq 2 || \mathfrak{J}_k||_{L^2 \to L^2} c_\tau D_{k,\gamma} +  2 |\textrm{Re} \langle \omega_k, (I+c_\tau \mathfrak{J}_k) \mathbb{NL}_k \rangle|.
\end{split}
\end{equation}
Define now
\begin{equation}
    \begin{split}
        T_2 \coloneqq \sup_k 4 \int_0^t\biggl|  \textrm{Re} \langle \omega_k, (I+c_\tau \mathfrak{J}_k) \mathbb{NL}_k \rangle \biggr| ds.
    \end{split}
\end{equation}
Then by applying the triangle inequality and \eqref{abs_val_eqn} to \eqref{pre_abs_val}, we find
\begin{equation}
    \E(t) + \frac{1}{2} \mathcal{D}_2(t) \leq 2 \E(0) - 4c \mathcal{D}_s(t) + 2||\mathfrak{J}_k||_{L^2 \to L^2} c_\tau \sup_k \int_0^t D_{k,\gamma}(s) ds + \int_0^t \mathcal{NL}_(s) ds + T_2(t).
\end{equation}
Now as $c_\tau < \frac{1}{32^2(1+ ||\mathfrak{J}_k||_{L^2 \to L^2})}$ by \eqref{coeff_conditions}, we finally have obtained
\begin{equation}\label{big_energy_ineq}
\begin{split}
    \E(t) &\leq 2\E(0) - 4c \D_1(t) - \frac{1}{4}\D_2(t) + \int_0^t \mathcal{NL}_1(s) ds + T_2(t)\\
    &= 2\E(0) - 4c \D_1(t) - \frac{1}{4}\D_2(t) + T_{\gamma, \tau}(t) + T_{\alpha, \alpha \tau}(t) + T_{\beta}(t) + T_2(t).
\end{split}
\end{equation}
Thus, the key terms to bound are $T_{\gamma, \tau}$, $T_{\alpha, \tau \alpha}$, $T_\beta$, and $T_2$. This motivates the statement of the following Lemma, which will complete the proof of Lemma \ref{bootstrap} and hence Theorem \ref{main_theorem}:
\begin{lemma}\label{nonlinear_lemma}
    There exists a universal constant $C > 0$ depending only on $m$, $J$, and the choice of $c$ such that
    $$|T_{\gamma,\tau}| + |T_{\alpha, \tau\alpha}| + |T_\beta| + |T_2| \leq \frac{C}{\nu^{1/2}}(1+\ln(1/\nu)^{1/2}) \sup_{s \in [0,t]}\mathcal{E}^{1/2}(s)\D(t).$$
\end{lemma}
We will accomplish the proof of Lemma \ref{nonlinear_lemma} by considering each term separately. Our first splitting will be based on the expression $\nabla_{k-k'}^\perp \phi_{k-k'} \cdot \nabla_{k'} \omega_k' = \partial_y \phi_{k-k'} ik' \omega_{k'} - i(k-k') \phi_{k-k'} \partial_y \omega_{k'}$ inside of $\mathbb{NL}_k$. To that end we write
\begin{equation}
    \begin{split}
        T_{\gamma,\tau}^x &\coloneqq \int_0^t \int_{\R} \int_{\R} \frac{\langle c \lambda_k s \rangle^{2J}}{M_k(s)}\langle k \rangle^{2m}  2\textrm{Re} \langle \omega_k, (I+c_\tau \mathfrak{J}_k) i(k-k') \phi_{k-k'} \partial_y \omega_{k'} \rangle dk' dk ds,\\
        T_{\gamma,\tau}^y &\coloneqq - \int_0^t\int_{\R} \int_{\R} \frac{\langle c \lambda_k s \rangle^{2J}}{M_k(s)}\langle k \rangle^{2m}  2\textrm{Re} \langle \omega_k, (I+c_\tau \mathfrak{J}_k)\partial_y \phi_{k-k'} ik' \omega_{k'} \rangle dk' dk ds,
    \end{split}
\end{equation}
\noindent so that $T_{\gamma,\tau} = T_{\gamma, \tau}^x + T_{\gamma, \tau}^y$. We perform the same type of splitting to write $T_{\alpha, \tau \alpha} = T_{\alpha, \tau\alpha}^x + T_{\alpha, \tau \alpha}^y$, and $T_{\beta} = T_{\beta}^x + T_\beta^y$. For $T_2$, we additionally use the triangle inequality and define
\begin{equation}
\begin{split}
    |T_2| & \leq \sup_k \int_0^t | \int_{\R} 4\textrm{Re} \langle \omega_k, (I+c_\tau \mathfrak{J}_k) i(k-k') \phi_{k-k'} \partial_y \omega_{k'} \rangle dk' |ds\\
    & \quad + \sup_k \int_0^t |\int_{\R} 4\textrm{Re} \langle \omega_k, (I+c_\tau \mathfrak{J}_k)\partial_y \phi_{k-k'} ik' \omega_{k'} \rangle dk'| ds\\
    &\eqqcolon T_2^x + T_2^y.
\end{split}
\end{equation}
In general, the $x$-terms will be simpler to bound than the $y$-terms. At an abstract level, for $T_{\gamma, \tau}$, $T_{\alpha, \tau \alpha}$, and $T_\beta$ we seek to bound terms of a form similar to:
\begin{equation}\label{abstract_equation}
\int_0^t \int_\R \int_\R \frac{\langle c \lambda_k s\rangle^{2J}}{M_k(s)} \langle k \rangle^{2m} \textrm{Re}\langle S_1(k,\partial_y)\omega_k, S_2(k-k',\partial_y) \phi_{k-k'} S_3(k', \partial_y) \omega_{k'}\rangle dk' dk ds,
\end{equation}
when $S_1$, $S_2$, and $S_3$ are some differential operators, and we have ignored factors arising from $\mathfrak{J}_k$. We will split \eqref{abstract_equation} through various frequency decompositions. The exact decomposition will vary from term to term, with the number of decompositions serving as an indicator for the difficulty of each term.

In general, splitting will be based on various combinations of high, intermediate, and low frequencies, inter-mixed between the frequencies $k$, $k'$, and $k-k'$. High frequencies are those of magnitude $\geq 1$, intermediate frequencies lie between $\nu$ and $1$ in magnitude, and low frequencies have magnitude $\lesssim \nu$. For the $\gamma,\tau$, $\alpha, \tau\alpha$, and $\beta$ terms, expression where all frequencies are high (namely $|k|, |k'|, |k-k'| \geq 1$) will have bounds similar to the $T_{\neq \neq}$ expressions in \cite{bedrossian2023stability}.

We will denote different slices of frequency domain with various subscripts. One of the divisions we will use most frequently is a splitting between $|k-k'| < |k'|/2$ and $|k-k'| \geq |k'|/2$, which we term $LH$ and $HL$, respectively. These refer to how $|k-k'|$ and $|k'|$ are relatively ``low" or ``high" with respect to each other. There is some additional complexity in the frequency decompositions arising from the fact that sometimes it is important to distinguish high and intermediate frequencies, and sometimes it is not important. Our next class of subscripts refers to the size of $k-k'$, $k$, and $k'$, and consists of the symbols $H$, $M$, and $L$, which will correspond to ``high," intermediate or ``medium," and ``low" frequencies, respectively. We will also use the subscript $H'$  and $L'$. The subscript $H'$ refers to frequencies $\geq \nu$ and $L'$ refers to frequencies $<1$. Note that $H'$ and $L'$ are not complementary ranges of frequencies. Rather one can view $H'$ as ``expanded" high frequencies (treating medium frequencies as high), while $L'$ refers to ``expanded" low frequencies (treating medium frequencies as low).  For each term, we will use up to four of the preceding subscripts. They will correspond, in order, to  the $HL$ versus $LH$ decomposition, and then to the size of $k-k'$, $k$, and $k'$. For example $T_{\gamma, \tau, LH, H', M, H}^x$ would indicate: 
$$T_{\gamma, \tau, LH, H', M, H}^x = \int_0^t \int_{\nu \leq |k| < 1} \int_{|k'| \geq 1} 1_{|k'|/2 > |k-k'| \geq 1}\frac{\langle c \lambda_k s \rangle^{2J}}{M_k(s)}\langle k \rangle^{2m}  2\textrm{Re} \langle \omega_k, (I+c_\tau \mathfrak{J}_k) i(k-k') \phi_{k-k'} \partial_y \omega_{k'} \rangle dk' dk ds.$$
This exact decomposition will not be used in the $T_{\gamma, \tau}$ case, but it is an illustrative example.\\

After performing frequency decomposition and potentially integration by parts in $y$, we pass the absolute value signs into the $k$ and $k'$ integrals. Once this is done, the $\frac{1}{M_k(s)}$ can be absorbed into the constant, since $1 \leq M_k(s) \leq e^{4J^2/3}$. We then use Cauchy-Schwarz in $y$ to place two factors in $L_y^2$ and one factor in $L_y^\infty$. Next, we manipulate the frequency decomposition to pass $x$ derivatives between the three factors. We also will pass copies of the time decay term $\langle c \lambda_k s \rangle^J$. We then use Young's convolution inequality to place two factors in $L^2_k$ and one factor in $L^1_k$. At this point, we will apply H\"older's inequality in time, placing one factor in $L^\infty$ in time and the other factors in $L^2$ in time. After using Minkowski's inequality to interchange the time and frequency integrations, we apply H\"older's inequality on the $L^1_k$ factor and utilize specific structures (such as the multiplier $\langle k \rangle^m$ with $m > 1/2$) to move this expression into $L^2_k$ or $L^\infty_k$. Lastly, we use the technical lemmas in Section \ref{technical_lemmas}  and the definitions \eqref{kEnergy}, \eqref{k_by_k_dissipation}, \eqref{introduce_energy}, \eqref{introduce_dissipation}, and \eqref{different_dissipations_definition} to express the result in a form which is bounded by $\nu^{-1/2} (1+\ln(1/\nu)^{1/2}) \mathcal{D} \sup_{s \in [0,t]}\mathcal{E}^{1/2}$ up to a constant.\\

To bound $T_2$, we use a similar abstract framework, but due to the nature of the supremum norm in $k$, we have less freedom as to which $L^p_k$ space we place various terms into.  Additionally, we will need to take care as when exactly we apply H\"older's inequality in time. At the same time, there are fewer multipliers within the norm, which causes simplification.

\subsection{Bound on \texorpdfstring{$T_{\gamma, \tau}$}{Gamma and Tau terms}}\label{gamma_tau_section}

First we deal with $T_{\gamma,\tau}$:

$$ T_{\gamma, \tau} = - \int_0^t \int_\R \int_{\R} \frac{\langle  c \lambda_k s \rangle^{2J}}{M_k(s)} \langle k  \rangle^{2m} 2\textrm{Re}\langle \omega_k, (I+c_\tau  \mathfrak{J}_k)\nabla^\perp \phi_{k-k'} \cdot \nabla \omega_{k'} \rangle dk' dk ds.$$

We note that in this section, and in all of the following sections, we will consistently take absolute values, and so $M_k(s)^{-1}$ is harmless. We will therefore absorb it into the implicit constant, which is allowed to depend on $J$, without further remark. Furthermore, by Lemma \ref{boundedness_of_J_k}, $\mathfrak{J}_k$ is bounded from $L^2$ to $L^2$, and we will consistently apply this result in every computation without further mention.

\subsubsection{x derivatives}

We start by concerning ourselves with $T_{\gamma, \tau}^x$, and we introduce the following frequency decomposition:\\
\begin{equation}
\begin{split}
    T_{\gamma,\tau}^x & = \int_0^t\int_\R \int_{\R} \frac{\langle c \lambda_k s \rangle^{2J}}{M_k(s)}\langle k  \rangle^{2m} 2\textrm{Re}\langle \omega_k, (I + c_\tau \mathfrak{J}_k) i(k-k')\phi_{k-k'} \partial_y\omega_{k'} \rangle dk' dk ds\\
    &= \int_0^t \int_\R \int_{\R} \frac{\langle c \lambda_k s \rangle^{2J}}{M_k(s)}\langle k  \rangle^{2m}\left(1_{|k-k'| \leq |k'|/2} +  1_{|k-k'| > |k'|/2} 1_{|k-k'| \geq 1 } + 1_{|k-k'| > |k'|/2} 1_{|k-k'| < 1} \right) \\
    &\hphantom{= \int_\R \int_{\R}}2\textrm{Re}\langle \omega_k, (I + c_\tau \mathfrak{J}_k) i(k-k')\phi_{k-k'} \partial_y\omega_{k'} \rangle dk' dk ds\\
    &\eqqcolon T_{\gamma, \tau, LH}^x +  T_{\gamma, \tau, HL, H}^x + T_{\gamma, \tau, HL, L'}^x.
\end{split}
\end{equation}
Recall that the notation $LH$ refers to the $|k-k'|$ frequency being ``low" relative to the ``high" $|k'|$ frequency. The notation $HL$ denotes the opposite relationship, and the additional $H$ and $L' $ subscripts denote $|k-k'| \geq 1$ and $|k-k'| < 1$, respectively. We do not include any subscripts corresponding to $k$ and $k'$ for the sake of notational simplicity, as we place no direct restrictions on these frequencies.
We first concern ourselves with $T_{\gamma, \tau, LH}^x$. Using the fact that $|k-k'| \lesssim |k'|$ on the support of the integrand of $T_{\gamma, \tau, LH}^x$, we have by the triangle inequality and Cauchy-Schwarz in $y$,
\begin{equation}
    \begin{split}
        |T_{\gamma, \tau, LH}^x| \lesssim \int_0^t \int_\R \int_\R  1_{|k-k'| \leq |k'|/2}\langle c \lambda_k s\rangle^{J}\langle k  \rangle^{m} || \omega_k||_2  ||(k-k') \phi_{k-k'}||_\infty \langle c \lambda_{k'} s\rangle^J \langle k' \rangle^m ||\partial_y\omega_{k'}||_2 dk' dk ds.
    \end{split}
\end{equation}
We then use Young's convolution inequality to place the $k$ and $k'$ factors into $L^2_k$ and the $k-k'$ factors in $L^1_k$. Then we use H\"older's inequality in time and apply the definitions of $\mathcal{E}$ \eqref{introduce_energy} and $\mathcal{D}_\gamma$ \eqref{different_dissipations_definition}:
\begin{equation}
    \begin{split}
        |T_{\gamma, \tau, LH}^x| &\lesssim \int_0^t \left(\int_\R \langle c \lambda_k s\rangle^{2J}\langle k  \rangle^{2m} || \omega_k||_2^2 dk \right)^{1/2}  \left(\int_{\R }|| k\phi_{k}||_\infty dk\right) \left(\int_\R\langle c \lambda_{k} s\rangle^{2J} \langle k \rangle^{2m} ||\partial_y\omega_{k}||_2^2 dk \right)^{1/2} ds\\
        &\lesssim \sup_{s \in [0,t]}\mathcal{E}^{1/2}(s) \left(\int_0^t \left(\int_{\R }|| k\phi_{k}||_\infty dk\right)^{2} ds \right)^{1/2}\nu^{-1/2} \mathcal{D}_\gamma^{1/2}.
    \end{split}
\end{equation}
To estimate $\left(\int_0^t \left(\int_{\R }|| k\phi_{k}||_\infty dk\right)^{2} ds \right)^{1/2}$, we apply Minkowski's inequality to find
$$\left(\int_0^t \left(\int_{\R }|| k\phi_{k}||_\infty dk\right)^{2} ds \right)^{1/2} \lesssim  \int_{\R } \left(\int_0^t|| k\phi_{k}||_\infty^2 ds\right)^{1/2} ds $$
Then we use Lemma \ref{keyphiestimate} at frequencies $|k| \geq 1$ and Lemma \ref{low_frequency_lemma_phi} at frequencies $|k| < 1$ to obtain
\begin{equation}\label{gamma_LH_x}
    \begin{split}
        |T_{\gamma, \tau, LH}^x| \lesssim \sup_{s \in [0,t]}\mathcal{E}^{1/2}(s) \left(\mathcal{D}_{\tau,2}^{1/2} + \mathcal{D}_\tau^{1/2}\right) \nu^{-1/2} \mathcal{D}_\gamma^{1/2} \lesssim \nu^{-1/2} \sup_{s \in [0,t]}\mathcal{E}^{1/2}(s) \mathcal{D}.
    \end{split}
\end{equation}
We turn our attention now to the $HL$ terms. We similarly use the triangle inequality and Cauchy-Schwarz in $y$. At high $|k-k'|$ frequencies, we use $|k-k'| \geq \max(|k'|/2, 1)$ on the support of the integrand, together with Lemma \ref{keyphiestimate} (in particular the variant \eqref{variant_of_key_phi_estimate}), Young's inequality, $m> 1/2$ and H\"older's inequality:
\begin{equation}\label{gamma_HL_H_x}
    \begin{split}
        |T_{\gamma, \tau, HL, H}^x| & \lesssim \int_0^t \int_{\R} \int_{\R} 1_{|k-k'| \geq 1} 1_{|k-k'| > |k'|/2 }\langle c \lambda_k s\rangle^{J}\langle k  \rangle^{m} || \omega_k||_2 \langle c \lambda_{k-k'} s\rangle^J \langle k-k' \rangle^{m}\\
        &\; \; \; \;\;\; \; \; \; \; \; \; ||(k-k') \phi_{k-k'}||_\infty ||\partial_y \omega_{k'}||_2 dk' dk ds\\
        &\lesssim \int_0^t \left(\int_\R\langle c \lambda_k s\rangle^{2J}\langle k  \rangle^{2m} || \omega_k||_2^2 dk \right)^{1/2} \left( \int_{|k| \geq 1} \langle c \lambda_k s\rangle^{2J}\langle k  \rangle^{2m} || k \phi_k||_\infty^2 dk\right)^{1/2}\\
        &\quad \quad \left(\int_\R \frac{\langle k \rangle^m }{\langle k \rangle^m} ||\partial_y \omega_k||_2dk\right) ds\\
        &\lesssim \sup_{s \in [0,t]}\mathcal{E}^{1/2}(s) \left(\int_0^t \int_{|k| \geq 1} \langle c \lambda_k s\rangle^{2J}\langle k  \rangle^{2m} || k^{3/2} \phi_k||_\infty^2 dk ds \right)^{1/2} \left( \int_0^t\int_\R \langle k \rangle^{2m}||\partial_y \omega_k||_2^2 dk ds \right)^{1/2}\\
        &\lesssim \sup_{s \in [0,t]}\mathcal{E}^{1/2}(s) \mathcal{D}_\tau^{1/2} \nu^{-1/2} \mathcal{D}_\gamma^{1/2}.
        \end{split}
\end{equation}
Our last $x$-derivative term is $T_{\gamma, \tau, HL, L'}^x$. In this case, we have $|k'| \lesssim |k-k'|$, and so exploiting this fact gives
\begin{equation*}
    \begin{split}
        |T_{\gamma, \tau, HL, L'}^x| & \lesssim \int_0^t \int_{\R} \int_{\R}1_{|k'|/2 < |k-k'| < 1} \langle c \lambda_k s\rangle^{J} \langle k  \rangle^{m} ||\omega_k||_2 \langle c \lambda_k s\rangle^{J}\langle k  \rangle^{m} |(k-k') \phi_{k-k'}||_\infty ||\partial_y \omega_{k'}||_2 dk' dk ds\\
        & \lesssim \int_0^t \int_{\R} \int_{\R} 1_{|k'|/2 < |k-k'| < 1}\langle c \lambda_k s\rangle^{J} \langle k  \rangle^{m} ||\omega_k||_2\langle c \lambda_{k-k'} s\rangle^J\\
        &\; \; \; \;\;\; \; \; \; \; \; \; \langle k-k' \rangle^{m}||(k-k') \phi_{k-k'}||_\infty |k'|^{1/2 }|k'|^{-1/2} ||\partial_y \omega_{k'}||_2 dk' dk ds\\
        & \lesssim \int_0^t \int_{\R} \int_{\R} 1_{|k'|/2 < |k-k'| < 1}\langle c \lambda_k s\rangle^{J}\langle k  \rangle^{m} ||\omega_k||_2\langle c \lambda_{k-k'} s\rangle^J \\
        &\; \; \; \;\;\; \; \; \; \; \; \; \langle k-k' \rangle^{m}||(k-k')^{3/2} \phi_{k-k'}||_\infty |k'|^{-1/2} ||\partial_y \omega_{k'}||_2 dk' dk ds.
    \end{split}
\end{equation*}
Note that as $1 > |k-k'| > |k'|/2$, we have $|k'| \leq 2$. We now apply Young's inequality, H\"older's inequality, Minkowski's inequality, Lemma \ref{keyphiestimate}, and the definition of $\mathcal{D}_2$:
\begin{equation}\label{gamma_HL_L_x}
    \begin{split}
         |T_{\gamma, \tau, HL, L'}^x| & \lesssim \sup_{s \in [0,t]}\mathcal{E}^{1/2}(s)\left(\int_0^t \int_{|k| \leq 1}\langle c \lambda_{k} s\rangle^{2J} \langle k \rangle^{2m}|||k|^{3/2} \phi_{k}||_\infty^2 dk\right)^{1/2} \left(\int_0^t \left(\int_{|k| \leq 2} |k|^{-1/2} ||\partial_y \omega_k||_2 dk\right)^2 ds \right)^{1/2}\\
         &\lesssim \sup_{s \in [0,t]}\mathcal{E}^{1/2}(s)\mathcal{D}_\tau^{1/2} \left(\int_{|k| \leq 2} |k|^{-1/2} \left(\int_0^t ||\partial_y \omega_k||_2^2 ds \right)^{1/2} dk\right) \\
         &\lesssim \sup_{s \in [0,t]}\mathcal{E}^{1/2}(s)\mathcal{D}_\tau^{1/2} \left(\int_{|k| \leq 2} |k|^{-1/2} dk \right) \sup_k\left(\int_0^t ||\partial_y \omega_k||_2^2 ds \right)^{1/2}\\
         &\lesssim \sup_{s \in [0,t]}\mathcal{E}^{1/2}(s)\mathcal{D}_\tau^{1/2} \nu^{-1/2} \D_2^{1/2}.
    \end{split}
\end{equation}
Combining \eqref{gamma_LH_x}, \eqref{gamma_HL_H_x}, and \eqref{gamma_HL_L_x} gives
\begin{equation}
    |T_{\gamma, \tau}^x| \lesssim \nu^{-1/2} \sup_{s \in [0,t]}\mathcal{E}^{1/2}(s) \mathcal{D},
\end{equation}
which suffices for the purposes of Lemma \ref{nonlinear_lemma}.

\subsubsection{y derivatives}
We now turn our attention to $T_{\gamma, \tau}^y$:
$$T_{\gamma,\tau}^y = - \int_0^t \int_{\R} \int_{\R} \frac{\langle c \lambda_k s \rangle^{2J}}{M_k(s)}\langle k \rangle^{2m}  2\textrm{Re} \langle \omega_k, (I+c_\tau \mathfrak{J}_k)\partial_y \phi_{k-k'} ik' \omega_{k'} \rangle dk' dk ds.$$
In comparison with $T_{\gamma, \tau}^x$, this will require a larger number of frequency slices. Namely, we use the following decomposition:
\begin{equation}
    \begin{split}
        1 =& 1_{|k-k'| < |k'|/2}\left(1_{|k-k'| \geq \nu}1_{|k| \geq \nu} + 1_{|k-k'| \geq \nu}1_{|k| < \nu}+ 1_{|k-k'| < \nu}\right)\\
        &+  1_{|k-k'| \geq |k'|/2}\biggl(1_{|k-k'| \geq 1} 1_{|k'| \geq \nu} + 1_{\nu \leq |k-k'| < 1}1_{|k'| \geq \nu} + 1_{|k-k'| \geq \nu} 1_{|k| \geq \nu} 1_{|k'| \leq \nu}\\
        &\hphantom{1_{|k-k'| \geq |k'|/2}\biggl(}+ 1_{|k-k'| \geq \nu} 1_{|k| \leq \nu} 1_{|k'| \leq \nu} + 1_{|k-k'| < \nu}\biggr).
    \end{split}
\end{equation}
We write the respective breakdown of $T_{\gamma, \tau}^y$ as
\begin{equation}\label{gamma_y_decomposition}
\begin{split}
        T_{\gamma, \tau}^y &\eqqcolon \left(T_{\gamma, \tau, LH, H', H', \cdot}^y + T_{\gamma, \tau, LH, H', L, \cdot}^y+ T_{\gamma, \tau, LH, L, \cdot, \cdot}^y\right)\\
        &\hphantom{=}+ \left(T_{\gamma, \tau, HL, H, \cdot, H'}^y+ T_{\gamma, \tau, HL, M, \cdot, H'}^y  + T_{\gamma, \tau, HL, H', H', L}^y + T_{\gamma, \tau, HL, H', L, L}^y + T_{\gamma, \tau, HL, L, \cdot, \cdot}^y\right).
\end{split}
\end{equation}
The symbol $LH$ corresponds to $|k-k'| < |k'|/2$, while $HL$ corresponds to the opposite equality. The final three characters in the subscripts of \eqref{gamma_y_decomposition} refer to the $k-k'$, $k$, and $k'$ frequencies, respectively. The symbols $H$, $H'$, $M$, and $ L$ denote frequencies in magnitude $\geq 1$, $\geq \nu$, between $\nu$ and $1$, and $< \nu$, respectively, while $\cdot$ denotes no restriction.

We will begin with treating the $LH$ terms. In particular, we start with $T_{\gamma, \tau, LH, H', H', \cdot}^y$. Since $|k-k'| < |k'|/2$, we have $|k| \approx |k'|$. Furthermore, we note that $|k'| \gtrsim \nu$ on the domain of integration, since $|k'| \gtrsim |k-k'| \geq \nu$.  Hence by elliptic regularity, Gagliardo-Nirenberg-Sobolev, $m > 1/2$, H\"older's inequality, and Lemma \ref{beta_and_gamma},
\begin{align}
        |T_{\gamma, \tau, LH, H', H', \cdot}^y| &\lesssim \int_0^t \int_{|k| \geq \nu} \int_{\R} 1_{\nu \leq |k-k'| < |k'|/2} \langle  c \lambda_k s \rangle^J \langle k  \rangle^{m} ||\omega_k||_2 || \partial_y \phi_{k-k'}||_\infty \langle c \lambda_{k'} s\rangle^J\langle k'  \rangle^{m} ||  k' \omega_{k'}||_2 dk' dk ds\\
        &\lesssim \int_0^t \int_{|k| \geq \nu} \int_{\R} 1_{\nu \leq |k-k'| < |k'|/2} \langle  c \lambda_ks \rangle^J \langle k  \rangle^{m}||k^{1/3}\omega_k||_2 || \partial_y \phi_{k-k'}||_\infty \\
        &\quad \quad \quad \langle c \lambda_{k'} s\rangle^J \langle k'  \rangle^{m}||  |k'|^{2/3} \omega_{k'}||_2 dk' dk ds\\
        &\lesssim  \nu^{-1/6} \mathcal{D}_\beta^{1/2} \sup_{s \in [0,t]}\left(\int_{|k| \geq \nu} \left(1_{\nu \leq |k| \leq 1}|k|^{-1/2} + 1_{|k| > 1} \right)||\partial_y^2 \phi_k||_2^{1/2} || k \partial_y \phi_k||_2^{1/2} dk\right)\left(\nu^{-1/3} \mathcal{D}_\beta^{1/4} \mathcal{D}_\gamma^{1/4}\right)\\
        &\lesssim \nu^{-1/6} \mathcal{D}_\beta^{1/2} ( 1+ \ln(1/\nu)^{1/2}) \sup_{s \in [0,t]}\mathcal{E}^{1/2}(s) \nu^{-1/3} \mathcal{D}_\beta^{1/4} \mathcal{D}_\gamma^{1/4}.
\end{align}
\indent Instead of treating $T_{\gamma, \tau, LH, H', L, \cdot}^y$, we turn our attention of $T_{\gamma, \tau, LH, L, \cdot, \cdot}^y$. Since $|k-k| \leq \nu$ here, $k$ and $k'$ are not well-separated, leading to a loss of control from purely $L^2$ estimates. In the case of the infinite channel, we are able to control these terms using the Poincar\'e inequality, while in the planar case we employ the $L^\infty_k$ energy $\mathcal{E}_2$ and the $L^\infty_k$ dissipation $\mathcal{D}_2$. In the case of $T_{\gamma, \tau, LH, L, \cdot, \cdot}^y$, we use the fact that $\omega_k = \partial_y^2 \phi_k - k^2 \phi_k$, integration by parts in $y$, $|k'| \approx |k|$ to compute:
\begin{equation}\label{gamma_LH_L_y}
    \begin{split}
        |T_{\gamma, \tau, LH, L, \cdot, \cdot}^y| & \lesssim \int_0^t \int_{\R} \int_{ \R} 1_{|k-k'| \leq \nu, |k-k'| \leq |k'|/2}\langle c \lambda_k s\rangle^J\langle k  \rangle^{m} \langle c \lambda_{k'} s\rangle^J \langle k' \rangle^m | \langle  \omega_k, \partial_y \phi_{k-k'} k'\omega_{k'} \rangle| dk' dk ds\\
        &\lesssim \int_0^t \int_{\R} \int_{\R} 1_{|k-k'| \lesssim \nu, |k-k'| \leq |k'|/2} \langle c \lambda_k s\rangle^J\langle k  \rangle^{m} \langle c \lambda_{k'} s\rangle^J \langle k' \rangle^m \biggl(| \langle  \omega_k, \partial_y \phi_{k-k'} k' k'^2 \phi_{k'} \rangle |\\
        &\hphantom{\lesssim\int_{\R} \int_{\R}} +  |\langle  \partial_y \omega_k, \partial_y \phi_{k-k'} k' \partial_y \phi_{k'}\rangle | + | \langle  \omega_k, \partial_y^2 \phi_{k-k'} k' \partial_y \phi_{k'} \rangle |\biggr) dk' dk ds\\
        &\lesssim \int_0^t \int_{\R} \int_{\R} 1_{|k-k'| \lesssim \nu, |k-k'| \leq |k'|/2} \langle c \lambda_k s\rangle^J\langle k  \rangle^{m} \langle c \lambda_{k'} s\rangle^J \langle k' \rangle^m |k-k'|^{-1/2} |k-k'|^{1/2}\\
        &\hphantom{\lesssim\int_{\R} \int_{\R}}\biggl(||k \omega_k||_2 || \partial_y\phi_{k-k'}||_\infty |k'| || k' \phi_{k'}||_2 +  ||\partial_y \omega_k||_2 ||\partial_y \phi_{k-k'} ||_\infty |k'| || \partial_y \phi_{k'}||_2\\
        &\hphantom{\lesssim\int_{\R} \int_{\R}}+ ||\omega_k||_2 ||\partial_y^2 \phi_{k-k'}||_\infty |k'| ||\partial_y \phi_{k'}||_2\biggr) dk' dk ds.
    \end{split}
\end{equation}
The first two terms on the right-hand side of \eqref{gamma_LH_L_y} can be combined to give
\begin{equation}\label{gamma_LH_L_y_second}
\begin{split}
       |T_{\gamma, \tau, LH, L, \cdot, \cdot}^y|  &\lesssim \int_0^t\int_{\R} \int_{\R} 1_{|k-k'| \lesssim \nu, |k-k'| \leq |k'|/2} \langle c \lambda_k s\rangle^J\langle k  \rangle^{m} \langle c \lambda_{k'} s\rangle^J \langle k' \rangle^m |k-k'|^{-1/2} |k-k'|^{1/2}\\
        &\hphantom{\lesssim\int_{\R} \int_{\R}}\biggl(||\nabla_k \omega_k||_2 || \partial_y\phi_{k-k'}||_\infty |k'| || \nabla_{k'}\phi_{k'}||_2 + ||\omega_k||_2 ||\partial_y^2 \phi_{k-k'}||_\infty |k'| ||\partial_y \phi_{k'}||_2\biggr) dk' dk ds.
\end{split} 
\end{equation}
The first term in \eqref{gamma_LH_L_y_second} is handled using Young's inequality, H\"older's inequality, and Lemma \ref{low_frequency_lemma_phi}. Meanwhile for the second term, we use elliptic regularity, Gagliardo-Nirenberg-Sobolev, H\"older's inequality, and Minkowski's inequality to compute,
\begin{equation}
    \begin{split}
    \left( \int_0^t \left(\int_{|k| \leq \nu} ||\partial_y^2 \phi_k||_\infty dk\right)^{2}  ds\right)^{1/2}& \leq \int_{|k| \leq \nu} \left(\int_0^t ||\partial_y^2 \phi_k||_\infty^2 ds\right)^{1/2} dk\\
    &\leq \left(\int_{|k| \leq \nu} |k|^{-1/2} dk\right)\sup_{|k| \leq \nu}\left(\int_0^t ||\partial_y \omega_k||_2||k \omega_k||_2\right)^{1/2} \lesssim \mathcal{D}_{\gamma,2}^{1/2}.
    \end{split}
\end{equation}
Thus we obtain
\begin{equation}
    |T_{\gamma, \tau, LH, L, \cdot, \cdot}^y| \lesssim \nu^{-1/2} \mathcal{D}_\gamma^{1/2} \nu^{1/4} \sup_{s \in [0,t]}\mathcal{E}^{1/2}(s) \mathcal{D}_\tau^{1/2} + \sup_{s \in [0,t]}\mathcal{E}^{1/2}(s) \mathcal{D}_{\gamma,2}^{1/2} \mathcal{D}_{\tau}^{1/2}.
\end{equation}

For the bound on $T_{\gamma, \tau, LH, H', L, \cdot}$, we note that $|k| \approx |k'|$ on the domain of integration, which implies $|k'| \lesssim \nu$, which further implies $|k-k'| \lesssim \nu$. Thus we may employ the same argument as with $T_{\gamma, \tau, LH, L, \cdot, \cdot}^y$ and we arrive at
\begin{equation}
    \begin{split}
    |T_{\gamma, \tau, LH, H', L, \cdot}^y| &\lesssim \nu^{-1/2} \mathcal{D}_\gamma^{1/2} \nu^{1/4} \sup_{s \in [0,t]}\mathcal{E}^{1/2}(s) \mathcal{D}_\tau^{1/2} + \sup_{s \in [0,t]}\mathcal{E}^{1/2}(s) \mathcal{D}_{\gamma,2}^{1/2} \mathcal{D}_{\tau}^{1/2}.
    \end{split}
\end{equation}

With the $LH$ terms concluded, we now turn to bounding the $HL$ terms, starting with $T_{\gamma, \tau, HL, H, \cdot, H'}^y$. We have $|k-k'| \geq |k'|/2$, so $|k'| \lesssim |k-k'|$. Thus we may pass derivatives in terms of $|k'|$ into derivatives in terms of $|k-k'|$. Additionally, since $|k-k'| \geq 1$, we freely have on the domain of integration that $|k-k'|^{2/3} \leq |k-k'|^{7/6}$. Then by H\"older's inequality, $m > 1/2$, and Lemma \ref{keydyphiestimate}:
\begin{equation}\label{gamma_HL_H_H'_y}
    \begin{split}
       |T_{\gamma, \tau, HL, H, \cdot, H'}^y| &\lesssim \int_0^t \int_{\R} \int_{|k'| \geq \nu} 1_{|k-k'| \geq \max(|k'|/2,1)}\langle c \lambda_k s\rangle^{2J}\langle k  \rangle^{2m} ||\omega_k||_2 || \partial_y \phi_{k-k'}||_\infty  || k' \omega_{k'}||_2 dk' dk ds\\
       &\lesssim \int_0^t \int_{\R} \int_{|k'| \geq \nu} 1_{|k-k'| \geq \max(|k'|/2,1)}\langle c \lambda_k s\rangle^J\langle k  \rangle^{m} ||\omega_k||_2 \langle c \lambda_{k-k'} s\rangle^J\\
       &\hphantom{\lesssim \int_{\R} \int_{|k'| \geq \nu}}||\langle k-k' \rangle^{m} |k-k'|^{2/3}  \partial_y \phi_{k-k'}||_\infty || |k'|^{1/3} \omega_{k'}||_2 dk' dk ds\\
        &\lesssim \sup_{s \in [0,t]}\mathcal{E}^{1/2}(s) \left( \int_0^t \int_{|k| \geq 1} \langle c \lambda_k s\rangle^{2J}\langle k \rangle^{2m} \nu^{-2/3} \nu^{2/3}||\langle k \rangle^{m} |k|^{7/3}  \partial_y \phi_{k}||_\infty^2 dk ds\right)^{1/2} \nu^{-1/6} \mathcal{D}_\beta^{1/2}\\
        &\lesssim \sup_{s \in [0,t]}\mathcal{E}^{1/2}(s) \nu^{-1/3}\mathcal{D}_{\tau \alpha}^{1/2} \nu^{-1/6} \mathcal{D}_\beta^{1/2}.
    \end{split}
\end{equation}

The intermediate frequency term $T_{\gamma, \tau, HL, M, \cdot, H'}^y$ also relies on Lemma \ref{keydyphiestimate}, but here we place the $|k-k'|$ factors in $L^1_k$. This is required to gain the additional half of a derivative on the $|k-k'|$ factors at the cost of a $\ln(1/\nu)^{1/2}$ correction, as in Lemma \ref{keydyphiestimate}. Using also the fact that $\langle k -k'\rangle \approx 1$, we have
\begin{equation}\label{gamma_HL_M_H'_y}
    \begin{split}
       |T_{\gamma, \tau, HL, M, \cdot, H'}^y| &\lesssim \int_0^t \int_{\R} \int_{|k'| \geq \nu} 1_{1 > |k-k'| \geq \max(|k'|/2,\nu)}\langle c \lambda_k s\rangle^{2J}\langle k  \rangle^{2m} ||\omega_k||_2 || \partial_y \phi_{k-k'}||_\infty\\
       &\quad \quad \quad \quad \quad \quad|| k' \omega_{k'}||_2 dk' dk ds\\
       &\lesssim \int_0^t \int_{\R} \int_{|k'| \geq \nu} 1_{1 > |k-k'| \geq \max(|k'|/2,\nu)}\langle c \lambda_k s\rangle^J\langle k  \rangle^{m} ||\omega_k||_2 \langle c \lambda_{k-k'} s\rangle^J\\
       &\quad \quad \quad \quad \quad \quad|| |k-k'|^{2/3}  \partial_y \phi_{k-k'}||_\infty \langle k' \rangle^{m} || |k'|^{1/3} \omega_{k'}||_2 dk' dk ds\\
        &\lesssim \sup_{s \in [0,t]}\mathcal{E}^{1/2}(s) \left( \int_0^t \left( \int_{\nu < |k| < 1} \langle c \lambda_k s\rangle^{J} \nu^{-1/3} \nu^{1/3}||\langle k \rangle^{m} |k|^{2/3}  \partial_y \phi_{k}||_\infty dk \right)^2 ds\right)^{1/2}\\
        &\quad \quad \quad \quad \quad \quad\nu^{-1/6} \mathcal{D}_\beta^{1/2}\\
        &\lesssim \sup_{s \in [0,t]}\mathcal{E}^{1/2}(s) \ln(1/\nu)^{1/2}\nu^{-1/3}\mathcal{D}_{\tau \alpha}^{1/2} \nu^{-1/6} \mathcal{D}_\beta^{1/2}.
    \end{split}
\end{equation}

For $T_{\gamma, \tau, HL, H', H', L}^y$, we use the fact that $|k'| \lesssim |k|, |k-k'|$ to move derivatives off of the $k'$ factors:
\begin{equation}
    \begin{split}
       |T_{\gamma, \tau, HL, H', H', L}^y| &\lesssim \int_{|k|\geq \nu} \int_{|k'|<\nu} 1_{|k-k'| \geq \max(|k'|/2,\nu)}\langle c \lambda_k s\rangle^{2J}\langle k  \rangle^{2m} ||\omega_k||_2 || \partial_y \phi_{k-k'}||_\infty || k'\omega_{k'}||_2 dk' dk\\
       &\lesssim \int_{|k|\geq \nu} \int_{|k'|<\nu} 1_{|k-k'| \geq \max(|k'|/2,\nu)}\langle c \lambda_k s\rangle^J\langle k  \rangle^{m} ||k^{1/3}\omega_k||_2 \langle c \lambda_{k-k'} s\rangle^J\\
       &\; \; \; \;\;\; \; \; \; \; \;\langle k - k'\rangle^m|| |k-k'|^{2/3}  \partial_y \phi_{k-k'}||_\infty || \omega_{k'}||_2 dk' dk.
    \end{split}
\end{equation}
Then, we split the above into the portion where $|k-k'| \geq 1$ and the portion where $\nu \leq |k-k'|<1$ in order to apply Lemma \ref{keydyphiestimate} (similar to what we did explicitly in \eqref{gamma_HL_H_H'_y} and \eqref{gamma_HL_M_H'_y}). For $|k-k'| \geq 1$, we place the $|k-k'|$ factors in $L^2_k$, and for $ \nu \leq |k-k'| < 1$ we place them in $L^1_k$, thereby obtaining:
\begin{equation}
    \begin{split}
         |T_{\gamma, \tau, HL, H', H', L}^y| \lesssim \nu^{-1/6} \mathcal{D}_\beta^{1/2} \nu^{-1/3}(1 + \ln(1/\nu)^{1/2}) \mathcal{D}_{\tau \alpha}^{1/2} \sup_{s \in [0,t]}\mathcal{E}^{1/2}.
    \end{split}
\end{equation}

Note that for $T_{\gamma, \tau, HL, H', L, L}^y$ , the fact that $|k|, |k'| < \nu$ implies that $|k-k'| \lesssim \nu$. Meanwhile for $T_{\gamma, \tau, HL, L, \cdot, \cdot}^y$, we note that since $|k-k'| < \nu$ and $|k-k'| \geq |k'|/2$, we obtain $|k'| \lesssim \nu$, which further implies $|k| \lesssim \nu$. Hence we may estimate both expression simultaneously using Gagliardo-Nirenberg-Sobolev, H\"older's inequality, and Minkowski's inequality:
\begin{align}
        |T_{\gamma, \tau, HL, H', L, L}^y| + |T_{\gamma, \tau, HL, L, \cdot, \cdot}^y| & \lesssim \int_0^t \int_{\R} \int_{\R} 1_{|k'|/2 < |k-k'| \lesssim \nu}\langle c \lambda_k s\rangle^J\langle k  \rangle^{m} ||\omega_k||_2 \langle c \lambda_{k-k'} s\rangle^J \langle k-k' \rangle^m\\ &\hphantom{\lesssim \int_{\R} \int_{\R}}||\partial_y \phi_{k-k'}||_{2} ||k'\omega_{k'}||_\infty dk' dk ds\\
        &\lesssim \int_0^t \int_{\R} \int_{\R} 1_{|k'|/2 < |k-k'| \lesssim \nu}\langle c \lambda_k s\rangle^J\langle k  \rangle^{m} || \omega_k||_2 \langle c \lambda_{k-k'} s\rangle^J \langle k-k' \rangle^m \\
        &\hphantom{\lesssim \int_{\R} \int_{\R} }||(k-k')\partial_y \phi_{k-k'}||_{2} |k'|^{-1/2} ||k'\omega_{k'}||_2^{1/2} ||\partial_y \omega_{k'}||_2^{1/2} dk' dk ds\\
        & \lesssim \sup_{s \in [0,t]}\mathcal{E}^{1/2} \mathcal{D}_{\tau}^{1/2} \left(\int_0^t\left( \int_{|k| \lesssim \nu} |k|^{-1/2} ||\nabla_k \omega_k ||_2 dk\right)^{2} ds \right)^{1/2}\\
        &\lesssim \sup_{s \in [0,t]}\mathcal{E}^{1/2}(s) \mathcal{D}_{\tau}^{1/2} \int_{|k| \lesssim \nu} |k|^{-1/2}\left(\int_0^t || \nabla_k \omega_k||_2^2 ds\right)^{1/2} dk\\
        & \lesssim \sup_{s \in [0,t]}\mathcal{E}^{1/2}(s) \mathcal{D}_{\tau}^{1/2} \nu^{-1/2} \mathcal{D}_{\gamma, 2}^{1/2}.
\end{align}
Note than rather than place $\partial_y \phi_{k-k'}$ in $L^\infty_y$, we placed $\omega_{k'}$ in $L_y^\infty$. We have now completed the $T_{\gamma,\tau}$ portion of Lemma \ref{nonlinear_lemma}.

\subsection{Bound on \texorpdfstring{$T_{\alpha, \tau \alpha}$}{Alpha and Tau Alpha terms}}\label{alpha_tau_alpha_plane}

We have 

$$ T_{\alpha, \tau\alpha} = - \int_0^t \int_{\R} \int_{\R} \frac{\langle  c \lambda_{k} s \rangle^{2J}}{M_k(s)}\langle k  \rangle^{2m} 2c_\alpha A(k)^2 \textrm{Re}\langle (I+c_\tau \mathfrak{J}_k)\partial_y \omega_k , \partial_y (\nabla^\perp \phi_{k-k'} \cdot \nabla \omega_{k'}) \rangle dk' dk ds.$$

In both the $x$-derivative and $y$-derivative portion of the proof, we shall use integration by parts and self-adjointness of $\mathfrak{J}_k$ to obtain a term of the form $S(k)|| \partial_y( I + c_\tau \mathfrak{J}_k)\partial_y\omega_k||_2$, for some multiplier $S(k).$ We note that $\mathfrak{J}_k$ is uniformly bounded in $k$ by Lemma \ref{boundedness_of_J_k}, and in the planar case, $\mathfrak{J}_k$ commutes with $\partial_y$. Hence we will freely use that $||\partial_y(c_\alpha I + c_\tau \mathfrak{J}_k) \partial_y \omega_k||_2 \lesssim ||\partial_y^2 \omega_k||_2 \lesssim ||\nabla_k \partial_y \omega_k||_2$.
\subsubsection{x derivatives}

We start with an integration by parts in $y:$
\begin{equation}
    \begin{split}
        T_{\alpha, \tau\alpha}^x &= \int_0^t \int_{\R} \int_{\R} \frac{\langle  c \lambda_{k} s \rangle^{2J}}{M_k(s)}\langle k  \rangle^{2m} c_\alpha A(k)^2 2\textrm{Re}\langle (I+c_\tau \mathfrak{J}_k)\partial_y \omega_k , \partial_y (i(k-k')\phi_{k-k'} \partial_y \omega_{k'}) \rangle dk' dk ds\\
        &= - \int_0^t \int_{\R} \int_{\R} \frac{\langle  c \lambda_{k} s \rangle^{2J}}{M_k(s)}\langle k  \rangle^{2m} c_\alpha A(k)^2 2\textrm{Re}\langle \partial_y(I+c_\tau \mathfrak{J}_k)\partial_y \omega_k ,  (i(k-k')\phi_{k-k'} \partial_y \omega_{k'}) \rangle dk' dk ds.
    \end{split}
\end{equation}
We use the frequency decomposition
\begin{equation}
    \begin{split}
        1 =& 1_{|k-k'| < |k'|/2}\\
        &+ 1_{|k-k'| \geq |k'|/2}\left(1_{|k| \geq 1} + 1_{\nu < |k| < 1}1_{|k'| \geq \nu} + 1_{\nu < |k| < 1}1_{|k'| < \nu}   + 1_{|k| < \nu} 1_{|k'| > \nu} + 1_{|k| < \nu}1_{|k'| < \nu}\right),
    \end{split}
\end{equation}
and hence we write
\begin{equation}
    \begin{split}
        T_{\alpha, \tau\alpha}^x & \eqqcolon T_{\alpha, \tau\alpha, LH}^x\\
        &\hphantom{\coloneqq}+ \left(T_{\alpha, \tau\alpha, HL, \cdot, H, \cdot}^x + T_{\alpha, \tau\alpha, HL, \cdot, M, H'}^x + T_{\alpha, \tau\alpha, HL, \cdot, M, L}^x + T_{\alpha, \tau\alpha, HL, \cdot, L, H'}^x + T_{\alpha, \tau\alpha, HL, \cdot, L, L}^x\right).
    \end{split}
\end{equation}

We begin with the simpler $LH$ term. When $|k-k'| < |k'|/2$, we have $|k| \approx |k'|$, and so
\begin{equation}
    \begin{split}
        |T_{\alpha, \tau\alpha, LH}^x| &\lesssim \int_0^t \int_\R \int_\R 1_{|k-k'| < |k'|/2} \langle c\lambda_k s \rangle^{2J} \langle k \rangle^{2m} A(k)^2 ||\nabla_k \partial_y \omega_k||_2 || (k-k') \phi_{k-k'}||_\infty ||\partial_y \omega_k||_2 dk' dk ds\\
        &\lesssim  \int_0^t \int_\R \int_\R 1_{|k-k'| < |k'|/2} \langle c\lambda_k s \rangle^{J} \langle k \rangle^{m} A(k)||\nabla_k \partial_y \omega_k||_2 || (k-k') \phi_{k-k'}||_\infty\\
        &\hphantom{\int_\R \int_\R}\quad\langle c\lambda_{k'} s \rangle^{J} \langle k' \rangle^{m} A(k')||\partial_y \omega_k||_2 dk' dk ds.
    \end{split}
\end{equation}
Then we use Young's convolution inequality, H\"older's and Minkowskis inequality (as done in other sections), and Lemmas \ref{keyphiestimate} and \ref{low_frequency_lemma_phi} to conclude
\begin{equation}
    |T_{\alpha, \tau\alpha, LH}^x| \lesssim \nu^{-1/2} \mathcal{D}_{\alpha}^{1/2}\left(\mathcal{D}_{\tau}^{1/2} + \mathcal{D}_{\tau, 2}^{1/2}\right) \sup_{s \in [0,t]}\mathcal{E}^{1/2}(s).
\end{equation}

For the $HL$ terms, we begin with $|k| \geq 1$. On the domain of the integrand of $T_{\alpha, \tau\alpha, HL, \cdot, H, \cdot}$, we have $|k| \geq 1$ and $|k-k'| \geq |k'|/2$. This implies $|k| \lesssim |k-k'|$ and $|k-k'| \gtrsim 1 \geq \nu$, so using $m > 1/3$ we compute
\begin{equation}\label{m>1/3}
    \begin{split}
        \langle k \rangle^m \nu^{1/3}|k|^{-1/3} &\lesssim \nu^{1/3} |k-k'|^{m-1/3} \lesssim \langle k-k' \rangle^m \nu^{1/3}|k-k'|^{-1/3}\\
        &\lesssim \langle k -k'\rangle^m \left(1_{|k'| < \nu} + 1_{|k'| \geq \nu} \nu^{1/3} |k'|^{-1/3} \right) = \langle k-k' \rangle^m A(k').
    \end{split}
\end{equation}
Note also that as $|k-k'| \gtrsim 1$, we have $|k-k'| \lesssim |k-k'|^{3/2}.$ Thus, using \eqref{m>1/3}, Lemma \ref{keyphiestimate}, and $m > 1/2$:
\begin{equation}
    \begin{split}
        | T_{\alpha, \tau\alpha, HL, \cdot,  H, \cdot }^{x}| &\lesssim \int_0^t \int_{|k| \geq 1} \int_\R 1_{|k'|/2 \leq |k-k'|} \langle c\lambda_k s \rangle^{2J} \langle k \rangle^{2m} A(k)^2 ||\nabla_k \partial_y \omega_k||_2 || (k-k') \phi_{k-k'}||_\infty\\
        &\hphantom{\lesssim \int_{k \geq 1} \int_{\R}} \quad ||\partial_y \omega_k||_2 dk' dk ds\\
        & \lesssim \int_0^t \int_{k \geq 1} \int_{\R} 
             1_{|k'|/2 \leq |k-k'| }\langle  c \lambda_{k} s \rangle^{J}||\langle k  \rangle^{m} \nu^{1/3} |k|^{-1/3}  \nabla_k \partial_y \omega_k ||_2\\
        & \hphantom{\lesssim \int_{k \geq 1} \int_{\R}}\langle  c \lambda_{k-k'} s \rangle^{J}||\langle k-k'\rangle^m (k-k')^{3/2}\phi_{k-k'}||_\infty A(k')||  \partial_y \omega_{k'} ||_2 dk' dk ds \\
        &\lesssim \nu^{-1/2} \mathcal{D}^{1/2}_\alpha \mathcal{D}_\tau^{1/2} \sup_{s \in [0,t]}\mathcal{E}^{1/2}(s).
    \end{split}
\end{equation}

Next, we deal with the more difficult $T_{\alpha, \tau\alpha, HL, \cdot,  M, H'}^x$. When $|k-k'| \geq |k'|/2$, we also have $|k-k'| \gtrsim |k|$. Hence, $1 = |k|^{-1/6} |k|^{1/6}\lesssim |k|^{-1/6} |k-k'|^{1/6}$. Furthermore, $1 = |k'|^{-1/3}|k'|^{1/3} \lesssim |k'|^{-1/3}|k-k'|^{1/3}$ and $\langle k \rangle \sim 1$. Altogether, these estimates imply that on the domain of integration
\begin{equation}\label{middle_freq_decomp_alpha}
    \nu^{2/3} |k|^{-2/3} |k-k'| \lesssim |k|^{-1/2} \nu^{1/3} |k|^{-1/3} |k-k'|^{3/2} \nu^{1/3} |k'|^{-1/3}.
\end{equation}
Thus by Young's convolution inequality, Lemma \ref{keyphiestimate}, Cauchy-Schwarz, and \eqref{middle_freq_decomp_alpha}:
\begin{equation}\label{alpha_HL_M_H'_x}
    \begin{split}
         |T_{\alpha, \tau\alpha, HL, \cdot,  M, H' }^x| 
         &\lesssim \int_0^t \int_{\nu < |k| < 1} \int_{|k'| \geq \nu} 1_{ |k-k'| \geq  |k'|/2 }\langle  c \lambda_{k} s \rangle^{2J} \langle k  \rangle^{2m}\nu^{2/3}|k|^{-2/3}||\nabla_k\partial_y \omega_k ||_2\\
         &\hphantom{\lesssim \int_{\nu < |k| < 1} \int_{|k'| \geq \nu}}||(k-k')\phi_{k-k'}||_\infty ||\partial_y \omega_{k'} ||_2 dk' dk ds\\
         &\lesssim \int_0^t \int_{\nu < |k| < 1} \int_{|k'| \geq \nu} 1_{ |k-k'| \geq  |k'|/2 }|k|^{-1/2}\langle  c \lambda_{k} s \rangle^{J} \langle k \rangle^{m} \nu^{1/3} |k|^{1/3} ||\nabla_k\partial_y \omega_k ||_2\langle  c \lambda_{k-k'} s \rangle^{J}\\
        & \hphantom{\lesssim \int_{\nu < |k| < 1} \int_{|k'| \geq \nu}}\langle k-k'\rangle^m|||k-k'|^{3/2}\phi_{k-k'}||_\infty \nu^{1/3} |k'|^{-1/3}|| \partial_y \omega_{k'} ||_2 dk' dk ds\\
        &\lesssim \left(\int_0^t\left(\int_{\nu < |k| < 1}|k|^{-1/2}  \langle  c \lambda_{k} s \rangle^{J} \langle k \rangle^m\nu^{1/3}  |k|^{-1/3} ||\nabla_k\partial_y \omega_k ||_2 dk\right)^2 ds\right)^{1/2}\\
        &\quad \quad \quad \mathcal{D}_\tau^{1/2} \sup_{s \in [0,t]}\mathcal{E}^{1/2}(s)\\
        &\lesssim \left(\int_{\nu < |k| < 1} |k|^{-1/2} dk\right)^{1/2} \left(\int_0^t \int_{\nu < |k| < 1} \langle  c \lambda_{k} s \rangle^{2J} \langle k \rangle^{2m} \nu^{2/3} |k|^{-2/3} ||\nabla_k\partial_y \omega_k ||_2^2 dk ds\right)^{1/2}\\
        &\quad \quad \quad \mathcal{D}_\tau^{1/2} \sup_{s \in [0,t]}\mathcal{E}^{1/2}(s)\\
        &\lesssim \ln(1/\nu)^{1/2} \nu^{-1/2} \mathcal{D}_\alpha^{1/2}\mathcal{D}_\tau^{1/2} \sup_{s \in [0,t]}\mathcal{E}^{1/2}(s).
    \end{split}
\end{equation}
We note that in the above, we broke from our usual pattern, and we placed the $k$ factors into $L^1_k$. This enabled used to remove the extraneous $|k|^{-1/2}$ factor via Cauchy-Schwarz at the cost of a logarithmic correction. It is precisely the need to deal with extraneous inverse powers of $|k|$ that makes the intermediate $k$ frequencies complicated.

For $T_{\alpha, \tau\alpha, HL, \cdot,  M, L}^x$, we note that $\nu < |k| < 1$ and $|k'| < \nu$ implies $|k-k'| \lesssim 1$. Hence $\langle k-k'\rangle \sim 1$. Additionally, $|k-k'| \gtrsim |k| > \nu$. We therefore will place the $k-k'$ factors in $L^1_k$ and apply Lemma \ref{keyphiestimate} (after suitably using H\"older's and Minkowski's inequalities). Lastly, we use that $\nu^{1/3} |k|^{-1/3} \leq 1$, and compute:
\begin{equation}
    \begin{split}
        |T_{\alpha, \tau\alpha, HL, \cdot,  M, L }^x| 
         &\lesssim \int_0^t \int_{\nu < |k| < 1} \int_{|k'| < \nu} 1_{ |k-k'| \geq  |k'|/2 }\langle  c \lambda_{k} s \rangle^{2J} \langle k  \rangle^{2m}\nu^{2/3} |k|^{-2/3}||\nabla_k\partial_y \omega_k ||_2\\
         &\hphantom{\lesssim \int_{\nu < |k| < 1} \int_{|k'| < \nu}}||(k-k')\phi_{k-k'}||_\infty ||\partial_y \omega_{k'} ||_2 dk' dk ds\\
         &\lesssim \int_0^t \int_{\nu < |k| < 1} \int_{|k'| < \nu} 1_{ |k-k'| \geq  |k'|/2 }\langle  c \lambda_{k} s \rangle^{J} \langle k  \rangle^{m}\nu^{1/3} |k|^{-1/3}||\nabla_k\partial_y \omega_k ||_2\\
         &\hphantom{\lesssim \int_{\nu < |k| < 1} \int_{|k'| < \nu}}\langle  c \lambda_{k-k'} s \rangle^{J}||(k-k')\phi_{k-k'}||_\infty \langle k'  \rangle^{m}||\partial_y \omega_{k'} ||_2 dk' dk ds\\
         & \lesssim \nu^{-1/2} \mathcal{D}_\alpha^{1/2} (1+ \ln(1/\nu))^{1/2} \mathcal{D}_\tau^{1/2} \sup_{s \in [0,t]}\mathcal{E}^{1/2}(s).
    \end{split}
\end{equation}

Our next $T_{\alpha, \tau\alpha}^x$ term is $T_{\alpha, \tau\alpha, HL, \cdot, L, H'}$. On the domain of this integrand, $|k'| > \nu$ and so $|k-k'| \gtrsim \nu$. We also have $1 = \nu^{-1/2} \nu^{1/2 }|k'|^{-1/3} |k'|^{1/3} \lesssim \nu^{-1/2} |k-k'|^{1/2} \nu^{1/3} |k'|^{-1/3}$. Then using Lemma \ref{keyphiestimate}, using Young's inequality to place the low-frequency $k$-term in $L^1_k((-\nu, \nu))$, and using H\"older's inequality to gain a power of $\nu^{1/2}$, we arrive at
\begin{equation}
    \begin{split}
        |T_{\alpha, \tau\alpha, HL, \cdot,  L, H' }^x| 
         &\lesssim \int_0^t \int_{|k| < \nu} \int_{|k'| > \nu} 1_{ |k-k'| \geq  |k'|/2 }\langle  c \lambda_{k} s \rangle^{2J} \langle k  \rangle^{2m}||\nabla_k\partial_y \omega_k ||_2 ||(k-k')\phi_{k-k'}||_\infty ||\partial_y \omega_{k'} ||_2 dk' dk ds\\
         &\lesssim \int_0^t \int_{|k| < \nu} \int_{|k'| > \nu} 1_{ |k-k'| \geq  |k'|/2 }\langle  c \lambda_{k} s \rangle^{J}||\nabla_k\partial_y \omega_k ||_2 \langle  c \lambda_{k-k'} s \rangle^{J} \langle k-k'\rangle^m\\
         &\hphantom{\lesssim \int_0^t \int_{|k| < \nu} \int_{|k'| > \nu} } \nu^{-1/2}|||k-k'|^{3/2}\phi_{k-k'}||_\infty \langle k' \rangle^m \nu^{1/3}|k'|^{-1/3} ||\partial_y \omega_{k'} ||_2 dk' dk ds\\
        &\lesssim \left(\int_0^t\left(\int_{|k| < \nu}\langle  c \lambda_{k} s \rangle^{J}||\nabla_k\partial_y \omega_k ||_2 dk\right)^{2} ds\right)^{1/2} \nu^{-1/2} \mathcal{D}_\tau^{1/2} \sup_{s \in [0,t]}\mathcal{E}^{1/2}(s)\\
        &\lesssim \nu^{1/2 - 1/2} \mathcal{D}_\alpha^{1/2} \nu^{-1/2} \mathcal{D}_\tau^{1/2} \mathcal{E}^{1/2}.
    \end{split}
\end{equation}

The final $T_{\alpha, \tau\alpha}^x$ term is $T_{\alpha, \tau\alpha, HL, \cdot, L, L}.$. For this term, we use that $|k-k'| \gtrsim |k'|$ on the domain of integration, Gagliardo-Nirenberg-Sobolev, elliptic regularity, H\"older's inequality, and Minkowski's inequality:
\begin{equation}
    \begin{split}
        |T_{\alpha, \tau\alpha, HL, \cdot,  L, L }^x| 
         &\lesssim \int_0^t\int_{|k| < \nu} \int_{|k'| < \nu} 1_{ |k-k'| \geq  |k'|/2 }\langle  c \lambda_{k} s \rangle^{2J} \langle k  \rangle^{2m}||\nabla_k\partial_y \omega_k ||_2 ||(k-k')\phi_{k-k'}||_\infty\\
         &\hphantom{\int_0^t \lesssim \int_{|k| < \nu} \int_{|k'| > \nu} } ||\partial_y \omega_{k'} ||_2 dk' dk ds\\
         &\lesssim \int_0^t \int_{|k| < \nu} \int_{|k'| < \nu} 1_{ |k-k'| \geq  |k'|/2 }\langle k \rangle^m \langle  c \lambda_{k} s \rangle^{J}||\nabla_k\partial_y \omega_k ||_2 \langle  c \lambda_{k-k'} s \rangle^{J} \\
         &\hphantom{\int_0^t \lesssim \int_{|k| < \nu} \int_{|k'| > \nu} }||(k-k')^{3/2}\phi_{k-k'}||_\infty \langle k' \rangle^m |k'|^{-1/2} ||\partial_y \omega_{k'} ||_2 dk' dk ds\\
        &\lesssim \nu^{-1/2} \mathcal{D}_\alpha^{1/2} \sup_{s \in [0,t]}\left(\int_{\R}\langle  c \lambda_{k} s \rangle^{2J} \langle k \rangle^{2m} |k|^{3} ||\phi_k||_\infty^2 dk\right)^{1/2} \left(\int_0^t\left(\int_{|k| < \nu} |k|^{-1/2} ||\partial_y \omega_k||_2dk\right)^2 ds \right)^{1/2}\\
        &\lesssim \nu^{-1/2} \mathcal{D}_\alpha^{1/2} \sup_{s \in [0,t]}\left(\int_{\R}\langle  c \lambda_{k} s \rangle^{2J} \langle k \rangle^{2m}  ||k \partial_y \phi_k||_2 ||k^2 \phi_k||_2 dk\right)^{1/2} \\
        &\quad \quad \quad \left(\int_{|k| < \nu} |k|^{-1/2} dk \right) \sup_{k}\left(\int_0^t ||\partial_y \omega_k||_2^2 ds\right)^{1/2}\\
        &\lesssim\nu^{-1/2} \mathcal{D}_\alpha^{1/2} \sup_{s \in [0,t]}\mathcal{E}^{1/2}(s) \nu^{1/2} \mathcal{D}_{\gamma, 2}^{1/2},
    \end{split}
\end{equation}
thereby finishing the $x$ derivatives.

\subsubsection{y
derivatives}

After an initial integration by parts, we seek to bound
$$T_{\alpha, \tau\alpha}^y = \int_0^t \int_{\R} \int_{\R} \frac{\langle  c \lambda_{k} s \rangle^{2J}}{M_k(s)}\langle k  \rangle^{2m} 2c_\alpha A(k)^2 \textrm{Re} \langle \partial_y (I+c_\tau \mathfrak{J}_k)\partial_y \omega_k , \partial_y \phi_{k-k'} i k'\omega_{k'} \rangle dk' dk ds.$$
We use the frequency decomposition
\begin{equation}
    \begin{split}
        1 = & 1_{|k-k'| < |k'|/2} + 1_{|k-k'| \geq |k'| / 2}\left(1_{|k| \geq 1} + 1_{\nu < |k| < 1}+ 1_{|k| < \nu}1_{|k'| \geq \nu} + 1_{|k| < \nu}1_{|k'| < \nu}\right),
    \end{split}
\end{equation}
and we correspondingly write
\begin{equation}
    \begin{split}
        T_{\alpha, \tau\alpha}^y &\eqqcolon  T_{\alpha, \tau\alpha, LH}^y + \left(T_{\alpha, \tau\alpha, HL, \cdot, H, \cdot}^y + T_{\alpha, \tau\alpha, HL, \cdot, M, \cdot}^y +  T_{\alpha, \tau\alpha, HL, \cdot, L, H'}^y +  T_{\alpha, \tau\alpha, HL, \cdot, L, L}^y\right).
    \end{split}
\end{equation}
We begin with $T_{\alpha, \tau\alpha, LH}^y$. On the domain of integration, $|k| \approx |k'|$, and so $A(k) \approx A(k')$. Then by Gagliardo-Nirenberg-Sobolev, elliptic regularity, $m > 1/2$, H\"older's inequality, Minkowski's inequality, and Lemmas \ref{beta_and_gamma} and \ref{low_frequency_lemma_phi},
\begin{equation}
    \begin{split}
        |T_{\alpha, \tau\alpha, LH}^y| & \lesssim \int_0^t \int_\R \int_\R 1_{|k-k'| < |k'|2} \langle  c \lambda_{k} s \rangle^{2J} \langle k  \rangle^{2m} A(k)^2|| \nabla_k\partial_y \omega_{k'}||_2 ||\partial_y \phi_{k-k'}||_\infty ||k' \omega_{k'}||_2 dk' dk ds\\
        & \int_0^t \int_\R \int_\R 1_{|k-k'| < |k'|2} \langle  c \lambda_{k} s \rangle^{J} \langle k  \rangle^{m} A(k)|| \nabla_k\partial_y \omega_k||_2 ||\partial_y \phi_{k-k'}||_\infty \langle c\lambda_{k'} s \rangle^{J} \langle k'  \rangle^{m} A(k')||k' \omega_{k'}||_2 dk' dk ds\\
        &\lesssim \nu^{-1/2} \mathcal{D}_\alpha^{1/2} \sup_{s \in [0,t]}\left(\int_{\R} (|k|^{-1/2} 1_{|k| < 1} + 1_{|k| \geq 1})  |k| || \partial_y \phi_k||_\infty dk\right) \left(\mathcal{D}_\gamma^{1/4} \mathcal{D}_\beta^{1/4} + \nu^{1/2} \mathcal{D}_\beta^{1/2}\right)\\
        &\lesssim \nu^{-1/2} \mathcal{D} \sup_{s \in [0,t]} \E(s)^{1/2} \left(\mathcal{D}_\gamma^{1/4} \mathcal{D}_\beta^{1/4} + \nu^{1/2} \mathcal{D}_\beta^{1/2}\right).
    \end{split}
\end{equation}

Now we turn to the $HL$ cases, starting with $T_{\alpha, \tau\alpha, HL, \cdot, H, \cdot}^y$. In this case, $|k-k'| \geq |k'|/2$ and $|k-k'| \gtrsim |k| \geq 1$, so by Lemma \ref{keydyphiestimate}, H\"older's inequality, and $m > 1/2$,
\begin{equation}
    \begin{split}
        |T_{\alpha, \tau\alpha, HL, \cdot, H, \cdot}^y |&\lesssim \int_0^t \int_{|k| \geq 1} \int_{\R} 1_{|k'|/2\leq |k-k'|} \langle  c \lambda_{k} s \rangle^{2J}\langle k  \rangle^{2m}\nu^{2/3}|k|^{-2/3}|| \nabla_k \partial_y \omega_k||_2 ||\partial_y \phi_{k-k'}||_\infty  || k' \omega_{k'}||_2 dk'dk ds\\
        &\lesssim \int_0^t \int_{|k| \geq 1} \int_{\R} 1_{|k'|/2\leq |k-k'|} \langle  c \lambda_{k} s \rangle^{J}\langle k  \rangle^{m}\nu^{1/3}|k|^{-1/3}|| \nabla_k \partial_y \omega_k||_2 \langle  c \lambda_{k-k'} s \rangle^{J} \\
        &\hphantom{\lesssim \int_0^t \int_{|k| \geq 1} \int_{\R}}\langle k-k' \rangle^{m}\nu^{1/3}||(k-k')\partial_y \phi_{k-k'}||_\infty  || \omega_{k'}||_2 dk'dk ds\\
        &\lesssim \int_0^t \int_{|k| \geq 1} \int_{\R} 1_{|k'|/2\leq |k-k'|} \langle  c \lambda_{k} s \rangle^{J}\langle k  \rangle^{m}\nu^{1/3}|k|^{-1/3}|| \nabla_k \partial_y \omega_k||_2 \langle  c \lambda_{k-k'} s \rangle^{J} \\
        &\hphantom{\lesssim \int_{|k| \geq 1} \int_{\R}}\langle k-k' \rangle^{m}\nu^{1/3}||(k-k')^{7/6}\partial_y \phi_{k-k'}||_\infty  || \omega_{k'}||_2 dk'dk\\
        &\lesssim \nu^{-1/2} \mathcal{D}_\alpha^{1/2} \mathcal{D}_{\tau \alpha}^{1/2} \sup_{s \in [0,t]}\mathcal{E}^{1/2}(s).
    \end{split}
\end{equation}

For $T_{\alpha, \tau\alpha, HL, \cdot, M, \cdot}^y$, we use a somewhat similar argument. On this domain, we use $|k| \lesssim |k-k'|$ to write $1 = |k|^{-1/6} |k|^{1/6} \lesssim |k|^{-1/6} |k-k'|^{1/6}$. Then using $\langle k \rangle \sim 1$ and $|k-k'| \gtrsim |k'|$ on the domain of integration, we have by Lemma \ref{keydyphiestimate} and Cauchy-Schwarz,
\begin{equation}\label{alpha_HL_M_y}
    \begin{split}
        |T_{\alpha, \tau\alpha, HL, \cdot, M, \cdot}^y |&\lesssim \int_0^t \int_{\nu < |k| < 1} \int_{\R} 1_{|k'|/2\leq |k-k'|} \langle  c \lambda_{k} s \rangle^{2J}\langle k  \rangle^{2m}\nu^{2/3}|k|^{-2/3}|| \nabla_k \partial_y \omega_k||_2\\
        &\hphantom{\lesssim \int_0^t \int_{\nu < |k| < 1} \int_{\R}} ||\partial_y \phi_{k-k'}||_\infty  || k' \omega_{k'}||_2 dk'dk ds\\
        &\lesssim \int_0^t \int_{|k| \geq 1} \int_{\R} 1_{|k'|/2\leq |k-k'|} |k|^{-1/2}\langle  c \lambda_{k} s \rangle^{J}\langle k  \rangle^{m}\nu^{1/3}|k|^{-1/3}|| \nabla_k \partial_y \omega_k||_2 \langle  c \lambda_{k-k'} s \rangle^{J} \\
        &\hphantom{\lesssim \int_0^t \int_{\nu < |k| < 1} \int_{\R}}\langle k-k' \rangle^{m}\nu^{1/3}||(k-k')^{7/6}\partial_y \phi_{k-k'}||_\infty  || \omega_{k'}||_2 dk'dk ds\\
        &\lesssim \left(\int_0^t\left(\int_{\nu < |k| < 1} |k|^{-1/2}\langle  c \lambda_{k} s \rangle^{J} \langle k \rangle^{2m} \nu^{1/3} |k|^{-1/3} ||\nabla_k\partial_y \omega_k ||_2 dk\right)^2 ds\right)^{1/2}\\
        &\quad \quad \quad \mathcal{D}_{\tau\alpha}^{1/2} \sup_{s \in [0,t]}\mathcal{E}^{1/2}(s)\\
        &\lesssim  \left(\int_0^t\int_{\nu < |k| < 1} \langle  c \lambda_{k} s \rangle^{2J} \langle k \rangle^{2m} \nu^{2/3} |k|^{-2/3} ||\nabla_k\partial_y \omega_k ||_2^2 dk ds\right)^{1/2}\\
        &\quad \quad \quad \left(\int_{\nu < |k| < 1} |k|^{-1} dk\right)^{1/2}\mathcal{D}_{\tau\alpha}^{1/2} \sup_{s \in [0,t]}\mathcal{E}^{1/2}(s)\\
        &\lesssim \ln(1/\nu)^{1/2} \nu^{-1/2} \mathcal{D}_\alpha^{1/2}\mathcal{D}_{\tau \alpha}^{1/2} \sup_{s \in [0,t]}\mathcal{E}^{1/2}(s).
    \end{split}
\end{equation}
Similar to $T_{\alpha, \tau \alpha, HL, \cdot, M, H'}^x$ \eqref{alpha_HL_M_H'_x}, we have placed the $k$ factors of \eqref{alpha_HL_M_y} into $L^1((-1,\nu) \cup (\nu, 1))$ in order to apply Cauchy-Schwarz to deal with the $|k|^{-1/2}$ factor. We remark that the logarithmic correction in \eqref{alpha_HL_M_H'_x} and \eqref{alpha_HL_M_y} arising from intermediate $k$ frequencies appears fundamentally difficult to remove.

We turn now to the $HL$ terms with $|k| < \nu$. When $|k'| \geq \nu$ ($T_{\alpha, \tau\alpha, HL, \cdot, L, H'}^y$), we use that $|k-k'| \gtrsim \nu$, $\langle k \rangle \sim 1$, together with Lemma \ref{keydyphiestimate} and H\"older's inequality:

\begin{equation}
    \begin{split}
        |T_{\alpha, \tau\alpha, HL, \cdot, L, H'}^y| & \lesssim \int_{k < \nu} \int_{k' > \nu} 1_{|k-k'| \geq |k'|/2}\langle  c \lambda_{k} s \rangle^{2J} \langle k  \rangle^{2m} || \nabla_k\partial_y \omega_k||_2 ||\partial_y \phi_{k-k'}||_\infty ||k' \omega_{k'}||_2 dk' dk\\
        &\lesssim \int_{k < \nu} \int_{k' \geq \nu} 1_{|k-k'| \geq |k'|/2}\langle  c \lambda_{k} s \rangle^{J} || \nabla_k\partial_y \omega_k||_2 \langle  c \lambda_{k-k'} s \rangle^{J}\\
        & \hphantom{\lesssim \int_{k < \nu} \int_{k' \geq \nu}}\langle k-k'\rangle^m \nu^{1/6} \nu^{-1/6}||(k-k')\partial_y \phi_{k-k'}||_\infty \langle k \rangle^m||\omega_{k'}||_2 dk' dk\\
        &\lesssim \int_{k < \nu} \int_{k' \geq \nu} 1_{|k-k'| \geq |k'|/2}\langle  c \lambda_{k} s \rangle^{J} || \nabla_k\partial_y \omega_k||_2 \langle  c \lambda_{k-k'} s \rangle^{J}\\
        & \hphantom{\lesssim \int_{k < \nu} \int_{k' \geq \nu}}\langle k-k'\rangle^m \nu^{-1/6}||(k-k')^{7/6}\partial_y \phi_{k-k'}||_\infty \langle k' \rangle^m||\omega_{k'}||_2 dk' dk\\
        &\lesssim \left(\int_{|k| < \nu} \langle  c \lambda_{k} s \rangle^{J} || \nabla_k\partial_y \omega_k||_2\right) \nu^{-1/2}\mathcal{D}_{\tau \alpha}^{1/2} \sup_{s \in [0,t]}\mathcal{E}^{1/2}(s)\\
        &\lesssim \mathcal{D}_\alpha^{1/2} \nu^{-1/2} \mathcal{D}_{\tau \alpha}^{1/2} \sup_{s \in [0,t]}\mathcal{E}^{1/2}(s).
    \end{split}
\end{equation}

The final $\alpha, \tau\alpha$ term is $T_{\alpha, \tau\alpha, HL, \cdot, L, L}^y$. Here, we observe that $|k-k'| \lesssim \nu$ and $|k'| \lesssim |k-k'|$, and hence by Gagliardo-Nirenberg-Sobolev, elliptic regularity, H\"older's inequality, and Minkowski's inequality:
\begin{equation}
    \begin{split}
       |T_{\alpha, \tau\alpha, HL, \cdot, L, L}^y| &\lesssim \int_0^t \int_{k <  \nu} \int_{k' < \nu} 1_{|k'|/2 \leq |k-k'| \lesssim \nu} \langle  c \lambda_{k} s \rangle^{2J}\langle k \rangle^{2m}||\nabla_k \partial_y \omega_k||_2  ||\partial_y \phi_{k-k'}||_\infty\\
       &\hphantom{\lesssim \int_0^t \int_{k <  \nu} \int_{k' < \nu}}|| k' \omega_{k'}||_2 dk' dk ds\\
        &\lesssim \int_0^t \int_{k <  \nu} \int_{k' < \nu} 1_{|k'|/2 \leq |k-k'| \lesssim \nu} \langle  c \lambda_{k} s \rangle^{J}||\nabla_k \partial_y \omega_k||_2 \langle  c \lambda_{k-k'} s \rangle^{J}\\ &
        \hphantom{\lesssim \int_0^t \int_{k <  \nu} \int_{k' < \nu}}|k-k'|^{1/2}||\partial_y \phi_{k-k'}||_\infty |k'|^{-1/2} || k' \omega_{k'}||_2 dk' dk ds\\
        &\lesssim \nu^{-1/2} \mathcal{D}_\alpha^{1/2} \left(\int_0^t \int_{|k| \lesssim \nu} \langle  c \lambda_{k} s \rangle^{2J}\langle k \rangle^{2m}||\partial_y^2 \phi_k||_2 ||k \partial_y \phi_k||_2 dk ds\right)^{1/2}\\
        &\quad \quad \quad \int_{|k| \lesssim \nu} |k|^{-1/2}\left(  \int_0^t||\nabla_k \omega_k||_2^2 ds\right)^{1/2} dk \\
        &\lesssim \nu^{-1/2} \mathcal{D}_\alpha^{1/2} \sup_{s \in [0,t]}\mathcal{E}^{1/2}(s)\left(\int_{|k| \lesssim \nu} |k|^{-1/2}\right) \sup_{k}\left(\int_0^t||\nabla_k \omega_k||_2^2 ds\right)^{1/2}\\
        &\lesssim \nu^{-1/2} \mathcal{D}_\alpha^{1/2} \sup_{s \in [0,t]}\mathcal{E}^{1/2}(s) \mathcal{D}_{\gamma,2}^{1/2},
        \end{split}
\end{equation}
completing the bounds on the $T_{\alpha,\tau\alpha}$ terms for the sake of Lemma \ref{nonlinear_lemma}.

\subsection{Bound on \texorpdfstring{$T_{\beta}$}{Beta terms}}

Written in full, the $\beta$ contribution is:
\begin{equation}\label{full_beta}
\begin{split}
        T_\beta &= c_\beta \int_0^t \int_\R \frac{\langle  c \lambda_{k} s \rangle^{2J}}{M_k(s)} \langle k \rangle^{2m} B(k)^2 \left(\textrm{Re}\langle ik \omega_k, \partial_y(\nabla^\perp \phi \cdot \nabla \omega)_k \rangle + \textrm{Re}\langle ik(\nabla^\perp \phi \cdot \nabla \omega)_k, \partial_y \omega_{k}\rangle\right) dk ds\\
    &\eqqcolon T_{\beta ; 1} + T_{\beta; 2}.
\end{split}
\end{equation}
Just as in \cite{bedrossian2023stability}, we note that integration by parts implies that $T_{\beta ;1}$ and $T_{\beta ; 2}$ can be treated similarly. Hence we will only perform the calculations for $T_{\beta ; 1}$, and in doing so, we will abuse notation and suppress the $1$, only writing $T_{\beta}$ for $T_{\beta ; 1}$.
\subsubsection{x derivatives}
We have, through integration by parts in $y$,
\begin{equation}
\begin{split}
        T_\beta^x &=  -c_\beta \int_0^t \int_\R \int_\R \frac{\langle  c \lambda_{k} s \rangle^{2J}}{M_k(t)} \langle k \rangle^{2m} B(k)^2 \textrm{Re}\langle ik \omega_k, \partial_y(i(k-k')\phi_{k-k'} \partial_y \omega_{k'}) \rangle dk' dk ds\\
    &= c_\beta \int_0^t \int_\R \int_\R \frac{\langle  c \lambda_{k} s \rangle^{2J}}{M_k(t)} \langle k \rangle^{2m} B(k)^2 \textrm{Re}\langle ik \partial_y \omega_k, i(k-k')\phi_{k-k'} \partial_y \omega_{k'} \rangle dk' dk ds.
\end{split}
\end{equation}
Our decomposition will be kept simple, only splitting into $LH$ and $HL$ terms according to
$$1 = 1_{|k-k'| < |k'|/2} + 1_{|k-k'| \geq |k'|/2},$$
which results in
$$T_\beta^x \eqqcolon T_{\beta, LH}^x + T_{\beta, HL}^x.$$
In both portions, we will use the fact that $B(k)^2|k| \lesssim A(k)$. Treating the $LH$ term first, we have by Lemmas \ref{keyphiestimate} and \ref{low_frequency_lemma_phi} and H\"older's and Minkowski's inequalities,
\begin{equation}
    \begin{split}
        |T_{\beta, LH}^x| &\lesssim \int_0^t \int_{\R} \int_{\R}1_{|k-k'| < |k'|/2} \langle  c \lambda_{k} s \rangle^{2J}\langle k  \rangle^{2m}B(k)^2|| k \partial_y \omega_k||_2  ||(k-k') \phi_{k-k'}||_\infty ||\partial_y \omega_{k'}||_2 dk' dk\\
        &\lesssim \int_0^t\int_{\R} \int_{\R}1_{|k-k'| < |k'|/2} \langle  c \lambda_{k} s \rangle^{J}\langle k  \rangle^{m}A(k)|| \partial_y \omega_k||_2  ||(k-k') \phi_{k-k'}||_\infty \\
        &\hphantom{\lesssim \int_0^t \int_{\R} \int_{\R}1_{|k-k'| < |k'|/2} }\langle  c \lambda_{k'} s \rangle^{J}\langle k'  \rangle^{m}||\partial_y \omega_{k'}||_2 dk' dk ds\\
        &\lesssim \sup_{s \in [0,t]}\mathcal{E}^{1/2}(s)\left( \int_{\R} (|k|^{-1/2}1_{|k| < 1} + 1_{|k| \geq 1}) \left( \int_0^t ||k^{3/2} \phi_{k}||_\infty^2 ds \right)^{1/2} dk\right) \nu^{-1/2} \mathcal{D}_\gamma^{1/2}\\
        &\lesssim \sup_{s \in [0,t]}\mathcal{E}^{1/2}(s) ( \mathcal{D}_{\tau,2}^{1/2} + \mathcal{D}_\tau^{1/2}) \nu^{-1/2} \mathcal{D}_\gamma^{1/2}.
    \end{split}
\end{equation}

For the $HL$ case, we notice that on the domain of integration, $1 = |k'|^{-1/2} |k'|^{1/2} 1_{|k'| < 1 } + 1_{|k'| \geq 1} \lesssim |k-k'|^{1/2}\left(|k'|^{-1/2} 1_{|k'| < 1 } + 1_{|k'| \geq 1}\right)$. Then by Lemma \ref{keyphiestimate}, H\"older's inequality, and $m > 1/2$:
\begin{equation}
    \begin{split}
        |T_{\beta, HL}^x| &\lesssim \int_0^t \int_{\R} \int_{\R}1_{|k-k'| \geq |k'|/2} \langle  c \lambda_{k} s \rangle^{2J}\langle k  \rangle^{2m}B(k)^2|| k \partial_y \omega_k||_2  ||(k-k') \phi_{k-k'}||_\infty ||\partial_y \omega_{k'}||_2 dk' dk ds\\
        &\lesssim \int_0^t \int_{\R} \int_{\R}1_{|k-k'| < |k'|/2} \langle  c \lambda_{k} s \rangle^{J}\langle k  \rangle^{m}A(k)|| \partial_y \omega_k||_2 \langle  c \lambda_{k-k'} s \rangle^{J}\\
        & \hphantom{\lesssim \int_0^t \int_{\R} \int_{\R}} \langle k -k' \rangle^{m} ||(k-k') \phi_{k-k'}||_\infty \left(|k'|^{-1/2} |k'|^{1/2} 1_{|k'| < 1 } + 1_{|k'| \geq 1}\right) ||\partial_y \omega_{k'}||_2 dk' dk ds\\
        &\lesssim \int_0^t \int_{\R} \int_{\R}1_{|k-k'| < |k'|/2} \langle  c \lambda_{k} s \rangle^{J}\langle k  \rangle^{m}A(k)|| \partial_y \omega_k||_2 \langle  c \lambda_{k-k'} s \rangle^{J}\\
        & \hphantom{\lesssim \int_0^t \int_{\R} \int_{\R}} \langle k -k' \rangle^{m} ||(k-k')^{3/2} \phi_{k-k'}||_\infty \left(|k'|^{-1/2} 1_{|k'| < 1 } + 1_{|k'| \geq 1}\right) ||\partial_y \omega_{k'}||_2 dk' dk ds\\
        &\lesssim \sup_{s \in [0,t]}\mathcal{E}^{1/2}(s) \mathcal{D}_\tau^{1/2} \left(\int_{\R} (|k|^{-1/2}1_{|k| < 1} + \langle k \rangle^m \langle k \rangle^{-m}1_{|k| > 1}) \left( \int_0^t||\partial_y \omega_{k}||_2^2 ds\right)^{1/2} dk\right)\\
        &\lesssim \sup_{s \in [0,t]}\mathcal{E}^{1/2}(s) \mathcal{D}_\tau^{1/2} \left(\nu^{-1/2}\mathcal{D}_2^{1/2} + \nu^{-1/2} \mathcal{D}_\gamma^{1/2}\right).
    \end{split}
\end{equation}
This completes the $x$-derivative terms.

\subsubsection{y derivatives}\label{beta_y_derivs}
The $y$ derivatives are slightly more complicated, but we are able to present a simplification, compared to the analogous terms in \cite{bedrossian2023stability}, where a commutator structure was employed. Recalling the Fourier multiplier 
$$B(k) = \begin{cases} \nu^{1/6} |k|^{-2/3}, & |k| \gtrsim \nu,\\
\nu^{-1/2}, & |k| \lesssim \nu,\end{cases}$$
we write
\begin{equation}
    T_{\beta}^y = c_\beta \int_0^t \int_\R \int_{\R} \frac{\langle  c \lambda_{k} s \rangle^{2J}}{M_k(s)} \langle k \rangle^{2m} B(k)^2 \textrm{Re}\langle i k \omega_k, \partial_y \left(\partial_y \phi_{k-k'} ik' \omega_{k'} \right)\rangle dk' dk ds.
\end{equation}
We make our first splitting of $T_\beta^y$ by distributing $\partial_y$:
\begin{equation}
\begin{split}
        T_{\beta}^{y,1} + T_{\beta}^{y,2}  &\coloneqq c_\beta \int_0^t \int_\R \int_{\R} \frac{\langle  c \lambda_{k} s \rangle^{2J}}{M_k(s)}\langle k \rangle^{2m} B(k)^2 \textrm{Re}\langle  i k \omega_k, \partial_y^2 \phi_{k-k'} ik' \omega_{k'} \rangle dk' dk ds\\
    &\quad + c_\beta \int_0^t \int_\R \int_{\R} \frac{\langle  c \lambda_{k} s \rangle^{2J}}{M_k(s)} \langle k \rangle^{2m} B(k)^2 \textrm{Re}\langle  i k \omega_k, \partial_y \phi_{k-k'} ik' \partial_y \omega_{k'} \rangle dk' dk ds.
\end{split}
\end{equation}
Let us begin by estimating $T_\beta^{y,1}$. We divide $T_\beta^{y,1}$ according to the decomposition
$$1 = 1_{|k-k'| < |k'|/2}\left(1_{|k| \geq \nu} + 1_{|k| < \nu}\right) + 1_{|k-k'| \geq |k'|/2}\left(1_{|k| \geq 1}  + 1_{\nu \leq |k| < 1} + 1_{|k| < \nu}\right),$$
giving
$$T_{\beta}^{y,1} \eqqcolon T_{\beta, LH, \cdot, H, \cdot}^{y,1} + T_{\beta, LH, \cdot, L, \cdot}^{y,1} + T_{\beta, HL, \cdot, H, \cdot}^{y,1} + T_{\beta, HL, \cdot, M, \cdot}^{y,1}+ T_{\beta, HL, \cdot, L, \cdot}^{y,1}.$$

Beginning with $LH$, we note that on the support of $T_{\beta, LH, \cdot, H, \cdot}^{y,1}$, $|k| \approx |k'|$, and so $B(k) \approx B(k')$ and $|k'| \gtrsim \nu$. Then by Lemma \ref{dy_dy_phi_to_energy}, H\"older's inequality in time, and interpolation in the $k$ factor (essentially a modified form of Lemma \ref{beta_and_gamma}), we have
\begin{equation}
\begin{split}
    |T_{\beta, LH, \cdot, H, \cdot}^{y,1}| &\lesssim \int_0^t \int_{|k| \geq \nu} \int_{\R} 1_{|k-k'| < |k'|/2} \langle  c \lambda_{k} s \rangle^{2J}\langle k  \rangle^{2m} B(k)^2|| k \omega_k||_2 ||\partial_y^2 \phi_{k-k'}||_\infty ||k' \omega_{k'}||_2 dk' dk\\
    &\lesssim \int_0^t \int_{|k| \geq \nu} \int_{\R} 1_{|k-k'| < |k'|/2} \langle  c \lambda_{k} s \rangle^{J} \langle k  \rangle^{m} \nu^{1/6}|| |k|^{1/3} \omega_k||_2^{1/4} || k \omega_k||_2^{1/4} || \omega_k||_{2}^{1/2}  \\
    &\hphantom{\lesssim \int_0^t \int_{|k| \geq \nu}   \int_{\R}} ||\partial_y^2 \phi_{k-k'}||_\infty \langle c \lambda_{k'} s \rangle^{J} \langle k'  \rangle^{m} \nu^{1/6}|k'|^{1/3}|| \omega_{k'}||_2 dk' dk ds\\
    &\lesssim \nu^{1/6} \left(\int_0^t \left(\int_{|k| \geq \nu} \langle c \lambda_k s \rangle^{2J} \langle k \rangle^{2m} || |k|^{1/3} \omega_k||_2^{1/2} || k \omega_k||_2^{1/2} ||\omega_k||_2 dk\right)^2 ds\right)^{1/4}\\
    & \quad \quad \left( \int_0^t \left(\int_{\R} ||\partial_y^2 \phi_k||_\infty dk\right)^4 ds \right)^{1/4}\\
    &\quad \quad \left(\int_0^t \int_{|k| \gtrsim \nu} \langle c \lambda_k s \rangle^{2J} \langle k \rangle^{2m} \nu^{1/3} |k|^{2/3} ||\omega_k||_2^2 dk ds\right)^{1/2}\\
    &\lesssim \nu^{1/6} \nu^{-1/24} \D_\beta^{1/8} \nu^{-1/8} \D_\gamma^{1/8} \sup_{s \in [0,t]}\E^{1/4}(s) \nu^{-1/4} \D_\gamma^{1/4} \sup_{s \in [0,t]}\E^{1/4}(s) \mathcal{D}_\beta^{1/2}\\
    &\lesssim \nu^{-1/4} \D \sup_{s \in [0,t]}\E^{1/2}(s).
\end{split}
\end{equation}
To handle $T_{\beta, LH, \cdot, L, \cdot}^{y,1}$, we use that $|k| < \nu$ and $|k-k'| < |k'|/2$ implies $|k'| \lesssim \nu$ and hence $|k-k'| \lesssim \nu$, Thus we have even more simply by $|k| \approx |k'|$ and Lemma \ref{dy_dy_phi_to_energy},
\begin{equation}
\begin{split}
    |T_{\beta, LH, \cdot, L, \cdot}^{y,1}| &\lesssim \int_0^t \int_{|k| < \nu} \int_{\R} 1_{|k-k'| < |k'|/2} \langle  c \lambda_{k} s \rangle^{2J}\langle k  \rangle^{2m} B(k)^2|| k \omega_k||_2 ||\partial_y^2 \phi_{k-k'}||_\infty ||k' \omega_{k'}||_2 dk' dk ds\\
    &\lesssim \int_0^t \int_{|k| < \nu} \int_{\R} 1_{|k-k'| < |k'|/2} \langle  c \lambda_{k} s \rangle^{J} \langle k  \rangle^{m} B(k)|| k \omega_k||_2 ||\partial_y^2 \phi_{k-k'}||_\infty \\
    &\hphantom{\lesssim \int_0^t \int_{|k| < \nu}   \int_{\R}} \langle c \lambda_{k'} s \rangle^{J} \langle k'  \rangle^{m} B(k')||k' \omega_{k'}||_2 dk' dk ds\\
    &\lesssim \D_\beta^{1/2} \nu^{1/2} \sup_{s \in [0,t]}\E^{1/2}(s) \D_\beta^{1/2}.
\end{split}
\end{equation}

For the $HL$ terms, we first focus on the high $|k|$ frequency term, $T_{\beta, HL, \cdot, H, \cdot}^y$. Since $1 \leq |k| $ on the support of the integrand, we have
\begin{equation}\label{m>2/3_splitting}
    1_{|k| \geq 1}|k|^{-1/3} \lesssim 1_{|k| \geq 1} |k|^{1/6}.
\end{equation}
Now \eqref{m>2/3_splitting}, Lemma \ref{dy_dy_phi_to_energy}, interpolation in $k$, H\"older's inequality, and elliptic regularity allow us to compute
\begin{equation}\label{beta_high_high_HL_y}
\begin{split}
    |T_{\beta, HL, \cdot, H, \cdot}^{y,1}| &\lesssim \int_{|k|\geq 1} \int_{\R} 1_{ |k-k'| \geq |k'|/2} \langle  c \lambda_{k} s \rangle^{2J} \langle k  \rangle^{2m} \nu^{1/3} |k|^{-4/3} || k \omega_k||_2 ||\partial_y^2 \phi_{k-k'}||_\infty ||k' \omega_{k'}||_2 dk' dk\\
    &\lesssim \int_{|k|\geq 1} \int_{\R} 1_{|k-k'| \geq |k'|/2}\langle  c \lambda_{k} s \rangle^{J} \langle k  \rangle^{m}  \left(\nu^{1/6} || k^{1/3} \omega_k||_2\right)^{1/2} ||\omega_k||_2^{1/2} \langle  c \lambda_{k-k'} s \rangle^{J}\\
    &\hphantom{\lesssim \int_{|k|\geq 1} \int_{\R}}\langle k-k' \rangle^{m}  \nu^{1/4}|| \partial_y^2 \phi_{k-k'}||_\infty ||k' \omega_{k'}||_2 dk' dk\\
    &\lesssim \mathcal{D}_\beta^{1/4} \sup_{s \in [0,t]}\E^{1/4}(s) \D_\gamma^{1/4} \sup_{s \in [0,t]}\E^{1/4}(s) \nu^{-1/2} \D_\gamma^{1/2}.
\end{split}
\end{equation}

We now deal with $HL$ at $ \nu \lesssim |k| \leq 1$, meaning $T_{\beta, HL, \cdot, M, \cdot}^{y,1}$. For $T_{\beta, HL, \cdot, M, \cdot}^{y,1}$, we place $\omega_k$ in $L_y^\infty$, then place the $||\omega_k||_\infty$ factor in $L^1_k\left((-1,-\nu) \cup (\nu, 1)\right)$ in order to apply Cauchy-Schwarz and use the following:
\begin{equation}\label{L_infty_energy_medium}
    \begin{split}
        \sup_{s \in [0,t]}\int_{\nu \leq |k| < 1} |k|^{-2/3}||\omega_k||_\infty dk&\lesssim \sup_{s \in [0,t]}\int_{\nu \leq |k| < 1} \nu^{-1/6} |k|^{-1/2} ||\nu^{1/3} |k|^{-1/3} \partial_y \omega_k ||_2^{1/2} ||\omega_k||_2^{1/2} dk\\
        &\lesssim \nu^{-1/6} \ln(1/\nu)^{1/2} \sup_{s \in [0,t]}\mathcal{E}^{1/2}.
    \end{split}
\end{equation}
Equipped with \eqref{L_infty_energy_medium}, we estimate $T_{\beta, HL, \cdot, M, \cdot}^{y,1}$ using $|k-k'|^{-1/3} \lesssim |k|^{-1/3}$ and elliptic regularity:
\begin{equation}
\begin{split}
    |T_{\beta, HL, \cdot, M, \cdot}^{y,1}| &\lesssim \int_0^t \int_{\nu < |k| \leq 1} \int_{\R} 1_{|k-k'| \geq |k'|/2} \langle c \lambda_{k} s\rangle^{2J}\langle k  \rangle^{2m} \nu^{1/3} |k|^{-1/3} ||\omega_k||_\infty ||\partial_y^2 \phi_{k-k'}||_2||k' \omega_{k'}||_2 dk' dk ds\\
    &\lesssim \int_0^t \int_{\nu < |k| \leq 1} \int_{\R} 1_{|k-k'| \geq |k'|/2} \nu^{1/6}|k|^{-1/3} \langle c \lambda_{k} s\rangle^{J} ||\omega_k||_\infty \langle c \lambda_{k-k'}t\rangle^{J}\\
    & \hphantom{\lesssim \int_0^t \int_{\nu < |k| \leq 1} \int_{\R}} \langle k -k'  \rangle^{m}  \nu^{1/6}|k-k'|^{-1/3}|k-k'|^{1/3}||\partial_y^2 \phi_{k-k'}||_2\langle k' \rangle^{m} ||k' \omega_{k'}||_2 dk' dk ds\\
    &\lesssim \int_0^t \int_{\nu < |k| \leq 1} \int_{\R} 1_{|k-k'| \geq |k'|/2} \nu^{1/6}|k|^{-2/3} \langle c \lambda_{k} s\rangle^{J}||\omega_k||_\infty \langle c \lambda_{k-k'}t\rangle^{J}\\
    & \hphantom{\lesssim \int_0^t \int_{\nu < |k| \leq 1} \int_{\R}} \langle k -k'  \rangle^{m}  \nu^{1/6}|k-k'|^{1/3}||\partial_y^2 \phi_{k-k'}||_2\langle k'  \rangle^{m} ||k' \omega_{k'}||_2 dk' dk ds\\
    &\lesssim \ln(1/\nu)^{1/2} \sup_{s \in [0,t]}\mathcal{E}^{1/2}(s) \mathcal{D}_\beta^{1/2} \nu^{-1/2} \mathcal{D}_\gamma^{1/2}.
\end{split}
\end{equation}

For $T_{\beta, HL, \cdot, L, H'}^{y,1}$, we compute using a variant of Lemma \ref{dy_dy_phi_to_energy} (treating $||\omega_k||_\infty$ in \eqref{beta_y_HL_L_H'} similar to $||\partial_y^2 \phi_k||_\infty$ in Lemma \ref{dy_dy_phi_to_energy}), elliptic regularity, and Lemma \ref{beta_and_gamma}:
\begin{equation}\label{beta_y_HL_L_H'}
    \begin{split}
        |T_{\beta, HL, \cdot, L, H'}^{y,1}| &\lesssim \int_0^t \int_{|k|< \nu} \int_{|k'| \geq \nu} 1_{|k-k'| \geq |k'|/2} \langle c \lambda_{k} s\rangle^{2J}\langle k  \rangle^{2m} \nu^{-1} || k\omega_k||_\infty ||\partial_y^2 \phi_{k-k'}||_2||k' \omega_{k'}||_2 dk' dk ds\\
        &\lesssim \int_0^t \int_{|k|< \nu} \int_{|k'| \geq \nu} 1_{|k-k'| \geq |k'|/2} \langle c \lambda_{k} s\rangle^{J} ||\omega_k||_\infty \langle c \lambda_{k-k'} s\rangle^{J}\\
        &\hphantom{ \lesssim \int_0^t \int_{|k|< \nu} \int_{|k'| \geq \nu}} \langle k-k'\rangle^m |k-k'|^{1/3} ||\partial_y^2 \phi_{k-k'}||_2 \langle k' \rangle^m ||k'^{2/3} \omega_{k'}||_2 dk' dk ds\\
        &\lesssim \nu^{1/2} \sup_{s \in [0,t]}\mathcal{E}^{1/2}(s) \nu^{-1/6} \mathcal{D}_\beta^{1/2} \nu^{-1/3} \mathcal{D}_\gamma^{1/4} \mathcal{D}_\beta^{1/4}.
    \end{split}
\end{equation}

For the final $T_{\beta}^{y,1}$ term, namely $T_{\beta, HL, \cdot, L, L}^{y,1}$, we note that $|k|, |k'| \lesssim \nu$ implies $|k-k'| \lesssim \nu$, and so we have by Lemma \ref{dy_dy_phi_to_energy}:
\begin{equation}
    \begin{split}
        |T_{\beta, HL, \cdot, L, L}^{y,1}| &\lesssim \int_0^t \int_{|k|< \nu} \int_{|k'| < \nu} 1_{|k-k'| \geq |k'|/2} \langle c \lambda_{k} s\rangle^{2J}\langle k  \rangle^{2m} \nu^{-1} || k\omega_k||_2 ||\partial_y^2 \phi_{k-k'}||_\infty||k' \omega_{k'}||_2 dk' dk ds\\
        &\lesssim \int_0^t \int_{|k| < \nu} \int_{|k'| < \nu} 1_{|k-k'| \geq |k'|/2} \langle c \lambda_{k} s\rangle^{J} \langle k \rangle^m \nu^{-1/2}||k\omega_k||_2 \langle c \lambda_{k-k'} s\rangle^{J}\\
        &\hphantom{ \lesssim \int_{|k| < \nu} \int_{|k'| <\nu}} \langle k-k'\rangle^m  ||\partial_y^2 \phi_{k-k'}||_\infty \langle k' \rangle^m \nu^{-1/2}||k' \omega_{k'}||_2 dk' dk ds\\
        &\lesssim \mathcal{D}_\beta^{1/2} \nu^{1/2} \sup_{s \in [0,t]}\mathcal{E}^{1/2}(s) \mathcal{D}_\beta^{1/2}.
    \end{split}
\end{equation}

Now we turn to the $T_\beta^{y,2}$ term. We use the frequency decomposition
$$1 = 1_{|k-k'| < |k'|/2} + 1_{|k-k'| \geq |k'|/2}\left(1_{|k| \geq 1}  + 1_{\nu \leq |k| < 1} + 1_{|k| < \nu}1_{|k'| \geq \nu} + 1_{|k| < \nu}1_{|k'| < \nu}\right),$$
resulting in 
$$T_{\beta}^{y,2} \eqqcolon T_{\beta, LH}^{y,2} + \left(T_{\beta, HL, \cdot, H, \cdot}^{y,1} + T_{\beta, HL, \cdot, M, \cdot}^{y,2}+ T_{\beta, HL, \cdot, L, H'}^{y,2} + T_{\beta, HL, \cdot, L, L}^{y,2}\right).$$

For the $LH$ term, we have $|k-k'| < |k'|/2$, implying $|k| \approx |k'|$. Furthermore, we notice that $B(k)^2 k^2 \lesssim A(k) |k|$. Thus, by Lemmas \ref{beta_and_gamma} and \ref{low_frequency_lemma_phi}, together with Gagliardo-Nirenberg-Sobolev, elliptic regularity, and $m>1/2$:
\begin{equation}
\begin{split}
    |T_{\beta, LH}^{y,2}| &\lesssim \int_0^t \int_{\R} \int_{\R} 1_{|k-k'| < |k'|/2} \langle  c \lambda_{k} s \rangle^{2J}\langle k  \rangle^{2m} B(k)^2|| k \omega_k||_2 ||\partial_y \phi_{k-k'}||_\infty ||k' \partial_y \omega_{k'}||_2 dk' dk ds\\
    &\lesssim \int_0^t \int_{\R} \int_{\R} 1_{|k-k'| < |k'|/2} \langle  c \lambda_{k} s \rangle^{J} \langle k  \rangle^{m} B(k)^2|| k^2 \omega_k||_2 ||\partial_y \phi_{k-k'}||_\infty \\
    &\hphantom{\lesssim \int_{\R}   \int_{\R}} \langle c \lambda_{k'} s \rangle^{J} \langle k'  \rangle^{m} || \partial_y \omega_{k'}||_2 dk' dk ds\\
    &\lesssim \left(\int_0^t \int_\R\langle  c \lambda_{k} s \rangle^{J} \langle k  \rangle^{2m} (A(k) k)^2||\omega_k||_2^2 dk ds\right)^{1/2}\\
    &\quad \quad \quad \sup_{s \in [0,t]}\left(\int_{|k|< 1} ||\partial_y \phi_k||_\infty dk + \int_{|k| \geq 1} ||\partial_y \phi_k||_\infty dk\right) \nu^{-1/2}\mathcal{D}_\gamma^{1/2}\\
    &\lesssim \left(\mathcal{D}_\gamma^{1/4} \mathcal{D}_\beta^{1/4} + \nu^{1/2}\mathcal{D}_\beta^{1/2}\right) \sup_{s \in [0,t]}\left( \mathcal{E}_2^{1/2} + \left(\int_{|k| \geq 1} \langle k \rangle^{2m} ||\partial_y^2 \phi_k||_2 ||k \partial_y \phi_k||_2 dk \right)^{1/2} \right)\nu^{-1/2} \mathcal{D}_\gamma^{1/2}.\\
    &\lesssim \mathcal{D}^{1/2} \sup_{s \in [0,t]}\mathcal{E}^{1/2}(s) \nu^{-1/2} \mathcal{D}_\gamma^{1/2}.
\end{split}
\end{equation}

For the $HL$ terms, we have the by now familiar splitting into high, intermediate, and low $k$ frequencies. Beginning with the high frequency $T_{\beta, HL, \cdot, H, \cdot}^{y,2}$ term, we utilize \eqref{m>2/3_splitting} and $|k'| \lesssim |k-k'|$, alongside elliptic regularity and Gagliardo-Nirenberg-Sobolev to write
\begin{equation}\label{beta_high_high_HL_y_2}
\begin{split}
    |T_{\beta, HL, \cdot, H, \cdot}^{y,2}| &\lesssim \int_0^t \int_{|k|\geq 1} \int_{\R} 1_{ |k-k'| \geq |k'|/2} \langle  c \lambda_{k} s \rangle^{2J} \langle k  \rangle^{2m} \nu^{1/3} |k|^{-4/3} || k \omega_k||_2 ||\partial_y \phi_{k-k'}||_\infty ||k' \partial_y \omega_{k'}||_2 dk' dk ds\\
    &\lesssim \int_0^t \int_{|k|\geq 1} \int_{\R} 1_{|k-k'| \geq |k'|/2}\langle  c \lambda_{k} s \rangle^{J} \langle k  \rangle^{m}  \left((\nu^{1/6} || k^{1/3} \omega_k||_2\right)^{1/2} ||\omega_k||_2^{1/2} \langle  c \lambda_{k-k'} s \rangle^{J}\\
    &\hphantom{\lesssim \int_{|k|\geq 1} \int_{\R}}\langle k-k' \rangle^{m}  \nu^{1/4} || (k-k') \partial_y \phi_{k-k'}||_\infty ||\partial_y \omega_{k'}||_2 dk' dk ds.\\
    &\lesssim \int_0^t \int_{|k|\geq 1} \int_{\R} 1_{|k-k'| \geq |k'|/2}\langle  c \lambda_{k} s \rangle^{J} \langle k  \rangle^{m}  \left(\nu^{1/6} || k^{1/3} \omega_k||_2\right)^{1/2} ||\omega_k||_2^{1/2} \langle  c \lambda_{k-k'} s \rangle^{J}\\
    &\hphantom{\lesssim \int_{|k|\geq 1} \int_{\R}}\langle k-k' \rangle^{m}  \nu^{1/4} || \partial_y \omega_{k-k'}||_2^{1/2} ||\omega_{k-k'}||_2^{1/2} ||\partial_y \omega_{k'}||_2 dk' dk ds.\\
\end{split}
\end{equation}
Then interpolation in $k$ and in $k-k'$, alongside H\"older's inequality in time allows us to estimate \eqref{beta_high_high_HL_y_2}  as
\begin{equation}
    \begin{split}
        |T_{\beta, HL, \cdot, H, \cdot}^{y,1}| &\lesssim \D_\beta^{1/4} \sup_{s \in [0,t]}\E^{1/4}(s) \D_\gamma^{1/4} \sup_{s \in [0,t]}\E^{1/4}(s) \nu^{-1/2} \D_\gamma^{1/2}.
    \end{split}
\end{equation}

For the medium $k$ frequencies $T_{\beta, HL, \cdot, M, \cdot}^{y,2}$, we utilize that on the domain of integration $1 = |k|^{-1/6} |k|^{1/6} \lesssim |k|^{-1/6}|k-k'|^{1/6}$, alongside $|k-k'| \gtrsim |k'|$ and $\langle k \rangle \sim 1$, giving
\begin{equation}
\begin{split}
    |T_{\beta, HL, \cdot, M, \cdot}^{y,2}| &\lesssim \int_0^t \int_{\nu \leq |k| < 1} \int_{\R} 1_{ |k-k'| \geq |k'|/2} \langle  c \lambda_{k} s \rangle^{2J} \langle k  \rangle^{2m} \nu^{1/3} |k|^{-4/3} || k \omega_k||_2 ||\partial_y \phi_{k-k'}||_\infty ||k' \partial_y \omega_{k'}||_2 dk' dk ds\\
    &\lesssim \int_0^t \int_{\nu \leq |k| < 1} \int_{\R} 1_{|k-k'| \geq |k'|/2}|k|^{-1/2}\langle  c \lambda_{k} s \rangle^{J}  || \omega_k||_2 \langle  c \lambda_{k-k'} s \rangle^{J}\\
    &\hphantom{\lesssim \int_0^t \int_{\nu \leq |k| < 1} \int_{\R}}\langle k-k' \rangle^{m}  \nu^{1/3}|||k-k'|^{7/6} \partial_y \phi_{k-k'}||_\infty \langle k' \rangle^m|| \partial_y \omega_{k'}||_2 dk' dk ds.\\
\end{split}
\end{equation}
We then use Young's inequality to place the $k$ factors in $L^1$, while placing the $k-k'$ and $k'$ factors in $L^2$. Then by Lemma \ref{keydyphiestimate} and H\"older's inequality, we obtain
\begin{equation}
    |T_{\beta, HL, \cdot, M, \cdot}^{y,2}| \lesssim \ln(1/\nu)^{1/2}\sup_{s \in [0,t]}\mathcal{E}^{1/2}(s) \mathcal{D}_{\tau\alpha}^{1/2} \nu^{-1/2}\mathcal{D}_\gamma^{1/2}.
\end{equation}

Next we consider $T_{\beta, HL, \cdot, L, H'}^{y,2}$. Here we will also place the $k$ factors into $L^1$, this time in order to gain a factor of $\nu^{1/2}$ by H\"older. We also use $|k-k'| \gtrsim |k'| \geq \nu$ and Lemma \ref{keydyphiestimate}:
\begin{equation}
\begin{split}
    |T_{\beta, HL, \cdot, L, H'}^{y,2}| &\lesssim \int_0^t \int_{|k| < \nu} \int_{|k'| \geq \nu} 1_{ |k-k'| \geq |k'|/2} \langle  c \lambda_{k} s \rangle^{2J} \langle k  \rangle^{2m} \nu^{-1}|| k \omega_k||_2 ||\partial_y \phi_{k-k'}||_\infty ||k' \partial_y \omega_{k'}||_2 dk' dk ds\\
    &\lesssim \int_0^t \int_{|k| < \nu} \int_{|k'| \geq \nu} 1_{|k-k'| \geq |k'|/2}\langle  c \lambda_{k} s \rangle^{J}  || \omega_k||_2 \langle  c \lambda_{k-k'} s \rangle^{J}\\
    &\hphantom{\lesssim \int_0^t \int_{|k| < \nu} \int_{|k'| \geq \nu}}\langle k-k' \rangle^{m}  \nu^{-1/6}\nu^{1/6}|||k-k'| \partial_y \phi_{k-k'}||_\infty \langle k' \rangle^m|| \partial_y \omega_{k'}||_2 dk' dk ds.\\
    &\lesssim \int_0^t \int_{|k| < \nu} \int_{|k'| \geq \nu} 1_{|k-k'| \geq |k'|/2}\langle  c \lambda_{k} s \rangle^{J}  || \omega_k||_2 \langle  c \lambda_{k-k'} s \rangle^{J}\\
    &\hphantom{\lesssim \int_0^t \int_{|k| < \nu} \int_{|k'| \geq \nu}}\langle k-k' \rangle^{m} \nu^{-1/6} |||k-k'|^{7/6} \partial_y \phi_{k-k'}||_\infty \langle k' \rangle^m|| \partial_y \omega_{k'}||_2 dk' dk ds.\\
    &\lesssim \sup_{s \in [0,t]}\left(\int_{|k| < \nu }\langle  c \lambda_{k} s \rangle^{J}  || \omega_k||_2 dk\right) \nu^{-1/2} \mathcal{D}_{\tau \alpha}^{1/2} \nu^{-1/2}\mathcal{D}_\gamma^{1/2}\\
    &\lesssim \nu^{1/2} \sup_{s \in [0,t]}\mathcal{E}^{1/2}(s) \nu^{-1} \mathcal{D}.
\end{split}
\end{equation}
The final $\beta$ term is $T_{\beta, HL, \cdot, L, L}^{y,2}$. Here we use Lemma \ref{keydyphiestimate} and $|k|,|k'| < \nu$ to obtain
\begin{equation}
\begin{split}
    |T_{\beta, HL, \cdot, L, L}^{y,2}| &\lesssim \int_0^t \int_{|k| < \nu} \int_{|k'| < \nu} 1_{ |k-k'| \geq |k'|/2} \langle  c \lambda_{k} s \rangle^{2J} \langle k  \rangle^{2m} \nu^{-1}|| k \omega_k||_2 ||\partial_y \phi_{k-k'}||_\infty ||k' \partial_y \omega_{k'}||_2 dk' dk ds\\
    &\lesssim \int_0^t \int_{|k| < \nu} \int_{|k'| < \nu} 1_{|k-k'| \geq |k'|/2}\langle  c \lambda_{k} s \rangle^{J} \nu^{-1/2} ||  k \omega_k||_2 \langle  c \lambda_{k-k'} s \rangle^{J}\\
    &\hphantom{\lesssim \int_0^t \int_{|k| < \nu} \int_{|k'| < \nu}}\langle k-k' \rangle^{m}  |||k-k'|^{1/2} \partial_y \phi_{k-k'}||_\infty \langle k' \rangle^m|| \partial_y \omega_{k'}||_2 dk' dk ds.\\
    &\lesssim \mathcal{D}_\beta^{1/2} \sup_{s \in [0,t]}\mathcal{E}^{1/2}(s) \nu^{-1/2} \mathcal{D}_\gamma^{1/2}.
\end{split}
\end{equation}
This completes the $\beta$ terms.

\subsection{Bound on \texorpdfstring{$T_2$}{Supremum terms}}

We seek to bound
$$T_2 = \sup_{k} \int_0^t | \int_{\R} 4\textrm{Re}\langle \omega_k, (I+c_\tau \mathfrak{J}_k) \nabla^{\perp}_{k-k'} \phi_{k-k'} \cdot\nabla_{k'} \omega_{k'} \rangle dk'| ds.$$
When the $(k-k')$ derivative falls on $\phi_{k-k'}$, we compute using the triangle inequality, H\"older's inequality, Lemmas \ref{keyphiestimate} and \ref{low_frequency_lemma_phi}, and $m > 1/2$:
\begin{align}
        T_{2}^x &\lesssim \sup_{k} \int_0^t \int_{\R} ||\omega_k||_2 ||(k-k') \phi_{k-k'}||_\infty ||\partial_y \omega_k||_2 dk' ds\\
        &\lesssim \sup_{k} \int_{\R}  \sup_{s \in [0,t]}||\omega_k||_2(s) \left(\int_0^t||(k-k') \phi_{k-k'}||_\infty^2 ds \right)^{1/2} \left(\int_0^t||\partial_y \omega_k||_2^2 ds \right)^{1/2} dk'\\
        & \lesssim \sup_{k}\sup_{s \in [0,t]}(||\omega_k||_2)(s) \left(\int_\R \left(\int_0^t |k| ||\partial_y \phi_{k}||_2 ||k \phi_k||_2 ds \right)^{1/2} dk\right)\sup_{k} \left( \int_0^t||\nabla_k \omega_k||_2^2 ds \right)^{1/2}\\
        &\lesssim \sup_{s \in [0,t]}\mathcal{E}_2^{1/2}(s)\left(\int_\R\left(|k|^{-1/2} 1_{|k| \leq 1} + 1_{|k| > 1}\frac{\langle k \rangle^m}{\langle k \rangle^m}\right) \left( \int_0^t|k|^2  ||\nabla_k \phi_k||_2^2 ds \right)^{1/2} dk\right) \nu^{-1/2} \mathcal{D}_{\gamma,2}^{1/2}\\
        &\lesssim \sup_{s \in [0,t]}\mathcal{E}_2^{1/2}(s)\left(\mathcal{D}_{\tau,2}^{1/2} + \mathcal{D}_{\tau}^{1/2}\right) \nu^{-1/2} \mathcal{D}_{\gamma,2}^{1/2},
\end{align}
which handles the entire $T_2^x$ term without further decomposition. Note that we needed to treat the time integrals carefully in order to obtain the desired result. We now turn to when $\partial_y$ falls on $\phi_{k-k'}$, and we recall
\begin{equation}
    |T_2^y| \lesssim \sup_k \int_0^t |\int_\R \mathrm{Re}\langle (\omega_k, (I+c_\tau \mathfrak{J}_k) \partial_y \phi_{k-k'} ik' \omega_{k'}\rangle dk' | ds.
\end{equation}
We split into $T_{2}^y \eqqcolon T_{2, LH}^y + T_{2,HL}^y$, meaning we use the frequency decomposition
$$1 = 1_{|k-k'| < |k'|/2}+ 1_{|k-k'| \geq |k'|/2}.$$
In the $LH$ case (meaning $|k-k'| < |k'|/2$), we write $\omega_{k'} = \partial_y^2 \phi_{k'} - k'^2 \phi_{k'}$. Then we use the triangle inequality, integration by parts in $y$, and $|k| \approx |k'|$ to compute
\begin{equation}\label{T_2_LH_y}
    \begin{split}
        |T_{2,LH}^y| &\lesssim \sup_{k} \int_0^t \int_{\R} 1_{|k-k'| < |k'|/2} |\langle \omega_k, (I + c_\tau \mathfrak{J}_k)\partial_y \phi_{k-k'} k' \omega_k\rangle| dk' ds\\
        & \lesssim \sup_{k} \int_0^t \int_{\R} 1_{|k-k'| < |k'|/2} \biggl(|\langle \omega_k, (I + c_\tau \mathfrak{J}_k)\partial_y \phi_{k-k'} k'^3 \phi_{k'}\rangle|\\
        &\hphantom{\lesssim \sup_{k} \int_0^t \int_{k'}}+ |\langle \partial_y\omega_k, (I + c_\tau \mathfrak{J}_k)\partial_y \phi_{k-k'} k' \partial_y\phi_{k'}\rangle|\\
        &\hphantom{\lesssim \sup_{k} \int_0^t \int_{k'}} + |\langle \omega_k, (I + c_\tau \mathfrak{J}_k)\partial_y^2 \phi_{k-k'} k' \partial_y\phi_{k'}\rangle|\biggr) dk' ds\\
        &\lesssim \sup_{k} \int_0^t\int_{\R} 1_{|k-k'| < |k'|/2} \biggl( ||k\omega_k||_2 ||\partial_y \phi_{k-k'}||_\infty |k'|||k' \phi_{k'}||_2\\
        &\hphantom{\lesssim \sup_{k} \int_0^t\int_{k'}}+||\partial_y\omega_k||_2 ||\partial_y \phi_{k-k'}||_\infty |k'|||\partial_y\phi_{k'}||_2\\
        &\hphantom{\lesssim \sup_{k} \int_0^t\int_{k'}}+||\omega_k||_2 ||\partial_y^2 \phi_{k-k'}||_\infty |k'|||\partial_y\phi_{k'}||_2\biggr) dk'ds\\
        &\lesssim \sup_{k} \int_0^t\int_{\R} 1_{|k-k'| < |k'|/2} \biggl( ||\nabla_k\omega_k||_2 ||\partial_y \phi_{k-k'}||_\infty |k'|||\nabla_{k'} \phi_{k'}||_2\\
        &\hphantom{\lesssim \sup_{k} \int_0^t\int_{k'}}+||\omega_k||_2 ||\partial_y^2 \phi_{k-k'}||_\infty |k'|||\nabla_{k'}\phi_{k'}||_2\biggr) dk' ds.
    \end{split}
\end{equation}
We split the first term of \eqref{T_2_LH_y} into $|k-k'| \geq 1$ and $|k-k'| < 1$. When $|k-k'| \geq 1$, we have by Young's inequality, H\"older's inequality, Minkowski's inequality, Gagliardo-Nirenberg-Sobolev, and elliptic regularity:
\begin{equation}
\begin{split}
    \sup_{k} \int_0^t\int_{\R} &1_{|k-k'| < |k'|/2} 1_{|k-k'| > 1 }||\nabla_k\omega_k||_2 ||\partial_y \phi_{k-k'}||_\infty |k'|||\nabla_{k'} \phi_{k'}||_2 dk'ds \\
    &\lesssim \sup_{k} \left(\int_0^t ||\nabla_k\omega_k||_2^2 ds \right)^{1/2}\sup_k \left(\int_0^t \left(\int_{\R} || |k-k'| \partial_y \phi_{k-k'}||_\infty  |k'|||\nabla_{k'} \phi_{k'}||_2 dk'\right)^2 ds\right)^{1/2}\\
    &\lesssim \nu^{-1/2} \D_{\gamma, 2}^{1/2}  \left(\int_0^t \left(\sup_k \int_\R  || |k-k'| \partial_y \phi_{k-k'}||_\infty  |k'|||\nabla_{k'} \phi_{k'}||_2 dk'\right)^2 ds\right)^{1/2}\\
    &\lesssim \nu^{-1/2} \D_{\gamma, 2}^{1/2} \left( \int_0^t \left(\int_\R || \omega_k||_2^2 dk \right) \left( \int_\R |k|^2 || \nabla_k \phi_k||_2^2 dk\right)ds\right)^{1/2}\\
    &\lesssim \nu^{-1/2} \D_{\gamma, 2}^{1/2} \sup_{s \in [0,t]}\E^{1/2}(s) \D_\tau^{1/2}.
\end{split}
\end{equation}
Meanwhile, the frequencies with $|k-k'| < 1$ are treated with a variant of Lemma \ref{low_frequency_lemma_phi}, H\"older's inequalty, and Minkowski's inequality as follows:
\begin{equation}\label{left_side}
\begin{split}
    \sup_{k} \int_0^t\int_{\R} &1_{|k-k'| < |k'|/2} 1_{|k-k'| < 1}||\nabla_k\omega_k||_2 ||\partial_y \phi_{k-k'}||_\infty |k'|||\nabla_{k'} \phi_{k'}||_2 dk'ds\\
    &\lesssim \sup_{k} \left(\int_0^t ||\nabla_k\omega_k||_2^2 \right)^{1/2} \sup_k \int_{\R} 1_{|k-k'| < 1} \sup_{s \in [0,t]}||\partial_y \phi_{k-k'}||_\infty \left( \int_0^t |k'|^2||\nabla_{k'} \phi_{k'}||_2^2 ds\right)^{1/2} dk'
    \\
    &\lesssim \nu^{-1/2} \D_{\gamma, 2}^{1/2} \left(\int_{|k| < 1}\sup_{s \in [0,t]}||\partial_y \phi_{k}||_\infty dk\right) \sup_k\left( \int_0^t |k|^2||\nabla_{k} \phi_{k}||_2^2 ds\right)^{1/2}\\
    &\lesssim \nu^{-1/2} \D_{\gamma, 2}^{1/2}\left( \int_{|k| < 1} |k|^{-1/2} dk \right) \sup_k \sup_{s \in [0,t]}\left(|k|^{1/2} ||\partial_y \phi_k||_\infty\right) \D_{\tau,2}^{1/2}\\
    &\lesssim  \nu^{-1/2} \D_{\gamma, 2}^{1/2}  \sup_{s \in [0,t]} \E_2(s)^{1/2} \D_{\tau,2}^{1/2}.
\end{split}
\end{equation}
To treat the remaining portion of \eqref{T_2_LH_y} we start by employing H\"older's inequality, Young's inequality, and Gagliardo-Nirenberg-Sobolev to find 
\begin{equation}\label{right_side_start}
    \begin{split}
        \sup_{k} &\int_0^t \int_{\R} 1_{|k-k'| < |k'|/2}||\omega_k||_2 ||\partial_y^2 \phi_{k-k'}||_\infty |k'|||\nabla_{k'}\phi_{k'}||_2 dk' ds\\
        &\lesssim \sup_k \sup_{s\in[0,t]} ||\omega_k||_2 \sup_k \int_\R 1_{|k-k'| < |k'|/2} \left(\int_0^t ||\partial_y^2 \phi_{k-k'}||_\infty^2 ds\right)^{1/2}\left( \int_0^t |k'|^2 ||\nabla_{k'} \phi_{k'}||_2^2ds\right)^{1/2} dk'\\
        &\lesssim \sup_{s \in [0,t]}\mathcal{E}_2^{1/2}(s)\biggl(\int_\R \left(\int_0^t ||\partial_y^3 \phi_k||_2 || \partial_y^2\phi_k||_2 ds \right)^{1/2} dk \biggr)\mathcal{D}_{\tau,2}^{1/2}.
    \end{split}
\end{equation}
Then, we note that by elliptic regularity and interpolation (with $m > 1/2$):
\begin{equation}
    \begin{split}
                \int_\R \left(\int_0^t ||\partial_y^3 \phi_k||_2 || \partial_y^2\phi_k||_2 ds \right)^{1/2} dk&\lesssim\int_\R \left(\int_0^t ||\partial_y \omega_k||_2 || \omega_k||_2 ds \right)^{1/2} dk \\
                &\lesssim \int_\R (|k|^{-1/2}1_{|k| < 1} + 1_{|k| \geq 1}) \left(\int_0^t ||\nabla_k \omega_k||_2^2 ds \right)^{1/2} dk\\
                &\lesssim \nu^{-1/2}(\D_{\gamma, 2}^{1/2} + \D_\gamma^{1/2}),
    \end{split}
\end{equation}
so that
\begin{equation}\label{right_side}
    \begin{split}
        \sup_{k} &\int_0^t \int_{\R} 1_{|k-k'| < |k'|/2}||\omega_k||_2 ||\partial_y^2 \phi_{k-k'}||_\infty |k'|||\nabla_{k'}\phi_{k'}||_2 dk' ds\\
        &\lesssim \sup_{s \in [0,t]}\mathcal{E}_2^{1/2}(s) \nu^{-1/2}(\D_{\gamma,2}^{1/2} + \D_\gamma^{1/2})\D_{\tau,2}^{1/2}.
    \end{split}
\end{equation}
Combining \eqref{T_2_LH_y}, \eqref{left_side}, and \eqref{right_side} yields
\begin{equation}
    \begin{split}
        |T_{2,LH}^y|
        &\lesssim \nu^{-1/2} \mathcal{D}_{\gamma,2}^{1/2} \sup_{s \in [0,t]}\mathcal{E}^{1/2} (s)\mathcal{D}_{\tau,2}^{1/2} + \sup_{s \in [0,t]}\mathcal{E}_2^{1/2} \nu^{-1/2}(\D_{\gamma,2}^{1/2} + \D_\gamma^{1/2})\D_{\tau,2}^{1/2}.
    \end{split}
\end{equation}
Our final estimate is for $|k'|/2 \leq |k-k'|$. Since $|k'| \lesssim |k-k'|$ we have by H\"older's inequality, $m > 1/2$, Gagliardo-Nirenberg-Sobolev:
\begin{align}
        |T_{2,HL}^y| &\lesssim \sup_{k} \int_0^t \int_{\R} 1_{|k'|/2 \leq |k-k'|} ||\omega_k||_2 |||k-k'|\partial_y \phi_{k-k'}||_2||\omega_{k'}||_\infty dk'\\
        & \lesssim \sup_{k} \int_{\R} 1_{|k'|/2 \leq |k-k'| } \sup_{s \in [0,t]}||\omega_k||_2 \left(\int_0^t |||k-k'|\partial_y \phi_{k-k'}||_2^2 ds\right)^{1/2}\\
        &\quad \quad \quad \quad \left(\int_0^t|k'|^{-1}||\partial_y \omega_{k'}||_2||k' \omega_{k'}||_2 ds \right)^{1/2} dk' ds\\
        &\lesssim \sup_{s \in [0, t]}\mathcal{E}_2^{1/2}(s) \mathcal{D}_{\tau,2}^{1/2} \left(\int_{\R} (|k|^{-1/2}1_{|k| < 1} + 1_{|k| \geq 1}) \left(\int_0^t ||\nabla_k \omega_k||_2^2 ds \right)^{1/2} dk\right)\\
        &\lesssim \sup_{s \in [0, t]}\mathcal{E}_2^{1/2}(s) \mathcal{D}_{\tau,2}^{1/2}\nu^{-1/2}(\mathcal{D}_{\gamma, 2}^{1/2} + \mathcal{D}_{\gamma}^{1/2}).
\end{align}
This completes the estimates on $T_2$, and hence the proofs of Lemmas \ref{nonlinear_lemma} and \ref{bootstrap}.

\section{Modifications for the Half-Plane}\label{half_modifications}

The proof of Theorem \ref{half_theorem} is almost identical to that of Theorem \ref{main_theorem}. The decay rate remains $\lambda_{k} = \lambda^{pl}(\nu, k)$, and we retain the $L^\infty_k$ frequency control. The definitions of $E_k$, $D_k$, $\mathcal{E}$, and $\mathcal{D}$ remain unchanged from Section \ref{outline_section}, barring that the $y$-norms and inner products now have domain $D = [0,\infty)$. Correspondingly, the linearized system is now
\begin{equation}\label{linear_half_system}
    \begin{cases}
        \partial_t \omega_k + ik y \omega_k - \nu \Delta_k \omega_k = 0,\\
        -\Delta_k \phi_k = \omega_k,\\
        \phi_k (y = 0) = \omega_k (y = 0) = 0.
    \end{cases}
\end{equation}
The only true difference is the definition of $\mathfrak{J}_k$, which we now re-define as
$$\mathfrak{J}_k [ f](y) \coloneqq |k| \textrm{p.v.} \frac{k}{|k|} \int_{0}^\infty \frac{1}{2i(y-y')} G_k(y,y') f(y') dy',$$
with
\begin{equation}
    \begin{split}
        G_k(y,y') &\coloneqq -\frac{1}{k}\left(e^{-|k||y-y'|} - e^{-|k|(y+y')}\right)\\
        &=-\frac{1}{|k|} \begin{cases}
        e^{-|k|y'} \sinh(ky), & y \leq y',\\
        e^{-|k|y} \sinh(ky'), & y' \leq y.\end{cases}
    \end{split}
\end{equation}
Importantly, $\mathfrak{J}_k$ is no longer a Fourier multiplier, and is more similar to the operator $\mathfrak{J}_k$ on the channel. We shall prove that $\mathfrak{J}_k$ remains a bounded operator $L^2([0,\infty)) \to L^2([0,\infty))$ which is uniformly bounded in $k$. Additionally, we will show that $\mathfrak{J}_k$ maps $H^1$ functions to $H^1$ functions, and that although $\mathfrak{J}_k$ has a non-trivial commutator with $\partial_y$, the commutator obeys the simple bound $||[\partial_y, \mathfrak{J}_k]||_{L^2 \to L^2} \lesssim |k|$. These statements, and their proofs, parallel much of the work done in \cite{bedrossian2023stability} for $\mathfrak{J}_k$ on the finite channel. 

\begin{lemma}\label{boundedness_of_J_k_half}
    For all $k \in \R \setminus \{0\}$, the operator $\mathfrak{J}_k$ extends to a bounded linear operator $L^2([0,\infty)) \to L^2([0,\infty))$ satisfying the uniform bound
    $$||\mathfrak{J}_k||_{L^2 \to L^2} \lesssim 1,$$
    where the implicit constant is uniform in $k$.
\end{lemma}

\begin{proof}

    We decompose as in \cite{bedrossian2023stability}:
    \begin{align}
            \frac{1}{|k|}\mathfrak{J}_k f(y) &= \frac{k}{|k|}\int_{[0,\infty)\setminus (y-\frac{1}{|k|},y+\frac{1}{|k|})} G_k(y,y') \frac{f(y')}{2i(y-y')}dy'\\
            &\quad+ \frac{k}{|k|} G_k(y,y) \textrm{p.v.} \int_{\max(y- \frac{1}{|k|},0)}^{y+\frac{1}{|k|}} \frac{f(y')}{2i(y-y')}dy'\\
            &\quad+ \frac{k}{|k|} \textrm{p.v.} \int_{\max(y- \frac{1}{|k|},0)}^{y+\frac{1}{|k|}} \left(G_k(y,y') - G_k(y,y)\right)\frac{f(y')}{2i(y-y')}dy'\\
            &\eqqcolon \mathcal{T}_1 + \mathcal{T}_2 + \mathcal{T}_3.
        \end{align}
    Notice that there is no principal value in the $\mathcal{T}_1$ term due to the removal of the singularity from the domain of integration. This splitting is to enable us to apply Schur's test on $\mathcal{T}_1$ and $\mathcal{T}_3$, while the bound on $\mathcal{T}_2$ will be based on boundedness of the Hilbert transform.
    
    We focus first on $\mathcal{T}_1$. Splitting the integral into two parts, we define
    $$\mathcal{T}_{1, \leq} \coloneqq \frac{k}{|k|} \int_0^{y - \frac{1}{k}} G_k(y,y') \frac{f(y')}{2i(y-y')} dy', \; \; \;\mathcal{T}_{1, \geq} \coloneqq \frac{k}{|k|} \int_{y + \frac{1}{k}}^{\infty} G_k(y,y') \frac{f(y')}{2i(y-y')} dy'.$$ Define the kernels
    $$K_k^{(\leq)}(y,y') \coloneqq \frac{k}{|k|} \frac{G_k(y,y')}{2i(y-y')} 1_{0 \leq y' \leq y - \frac{1}{|k|}}(y,y'), \; \; \; K_k^{(\geq)}(y,y') \coloneqq \frac{k}{|k|} \frac{G_k(y,y')}{2i(y-y')} 1_{0 \leq y + \frac{1}{k} \leq y'}(y,y'),$$
    so that by Schur's test we have
    \begin{equation}
        \begin{split}
            ||\mathcal{T}_{1, \leq} f||_2 &\lesssim \left(||K_k^{(\leq)}||_{L^\infty_y L_{y'}^1} + ||K_k^{(\leq)}||_{L^\infty_{y'} L_{y}^\infty}\right)||f||_2,\\
            ||\mathcal{T}_{1, \geq} f||_2 &\lesssim \left(||K_k^{(\geq)}||_{L^\infty_y L_{y'}^1} + ||K_k^{(\geq)}||_{L^\infty_{y'} L_{y}^\infty}\right)||f||_2.
        \end{split}
    \end{equation}

    We tackle the ``$\leq$" case first. Fix a $y$ such that $y - \frac{1}{|k|} \geq 0$ (otherwise this case is trivial). Then we have using the change of variables $z = |k|(y'-y)$:
    \begin{equation}
        \begin{split}
            ||K_k^{(\leq)}(y, \cdot)||_{L^1_{y'}} & \lesssim \int_0^\infty \frac{1}{|k|} \frac{1}{|y-y'|} e^{-|k|y} \sinh(|k| y') 1_{0 \leq y' \leq y - \frac{1}{|k|}}(y,y') dy' \\
            &\lesssim \frac{1}{|k|} \int_0^{y-\frac{1}{|k|}} \frac{e^{-|k||y-y'|}}{|y-y'|} dy'\\
            &\leq \frac{1}{|k|} \int_{-\infty}^{y-\frac{1}{|k|}} \frac{e^{-|k||y-y'|}}{|y-y'|} dy'\\
            &\leq \frac{1}{|k|} \int_{-\infty}^{-1} \frac{e^{-|z|}}{z} dz \lesssim \frac{1}{|k|},
        \end{split}
    \end{equation}
    which implies that $||K_k^{(\leq)}||_{L^\infty_{y} L^1_{y'}} \lesssim \frac{1}{|k|}$. Now for a fixed $y' \geq 0$, we compute
    \begin{equation}
        \begin{split}
            ||K_k^{(\leq)}(\cdot,y')||_{L^1_{y}} & \lesssim \int_0^\infty \frac{1}{|k|} \frac{1}{|y-y'|} e^{-|k|y} \sinh(|k| y') 1_{0 \leq y' \leq y - \frac{1}{|k|}}(y,y') dy\\
            &\lesssim \frac{1}{|k|} \int_{y' + 1/k}^\infty \frac{e^{-|k||y-y'|}}{|y-y'|} dy\\
            & = \frac{1}{|k|} \int_{1}^\infty \frac{e^{-|z|}}{|z|} dz \lesssim \frac{1}{|k|},
        \end{split}
    \end{equation}
    implying that $||K_k^{(\leq)}||_{L_{y'}^{\infty}L_y^1} \lesssim \frac{1}{|k|}$, which gives the desired estimate on $||\mathcal{T}_{1, \leq}||_{L_2 \to L^2}$.The same type of argument, based on change of variables, gives $||K_k^{(\geq)}||_{L^\infty_y L_{y'}^1} + ||K_k^{(\geq)}||_{L^\infty_{y'} L_{y}^\infty} \lesssim \frac{1}{|k|}$, and hence $||\mathcal{T}_1|| \lesssim \frac{1}{|k|}$.\\

    We turn now to estimating $\mathcal{T}_2$. Note that $||G_k(y,y)||_{L^\infty_y} \lesssim \frac{1}{|k|}$. 
    Thus by boundedness of the Hilbert transform and uniform boundedness of the truncated Hilbert transform:
    \begin{equation}
        \begin{split}
            ||\mathcal{T}_2 f||_2 &\lesssim ||G_k(y,y)||_\infty || \textrm{p.v} \int_{\max(y- \frac{1}{|k|},0)}^{y+\frac{1}{|k|}} \frac{f(y')}{y-y'} dy' ||_2 \lesssim \frac{1}{|k|} ||f||_2.
        \end{split}
    \end{equation}

    The final step is to estimate $\mathcal{T}_3$. This proceeds via Schur's test:
    \begin{equation*}
        \begin{split}
            \int_{y-1/k}^{y+1/k} 1_{[0,\infty)} \frac{|G_k(y,y') - G_k(y,y)|}{|y-y'|}dy' &\lesssim \sup_{z,y' \in [0,\infty), |z-y'| \leq \frac{1}{|k|}} \left(\frac{|G_k(z,y') - G_k(z,z)|}{|y-y'|}\right)\int_{y- \frac{1}{|k|}}^{y + \frac{1}{|k|}} dy' \lesssim \frac{1}{|k|},\\
            \int_{y'-1/k}^{y'+1/k} 1_{[0,\infty)} \frac{|G_k(y,y') - G_k(y,y)|}{|y-y'|}dy &\sup_{y,z \in [0,\infty), |y-z| \leq \frac{1}{|k|}} \left(\frac{|G_k(y,z) - G_k(y,y)|}{|y-y'|}\right)\int_{y'- \frac{1}{|k|}}^{y' + \frac{1}{|k|}} dy \lesssim \frac{1}{|k|},\\
        \end{split}
    \end{equation*} where we have used the fact that
    $$\sup_{y,y' \in [0,\infty), |y-y'| \leq \frac{1}{|k|}} \left|\frac{|G_k(y,y') - G_k(y,y)|}{|y-y'|}\right| \lesssim \sup_{y,y' \in [0,\infty), |y-y'| \leq \frac{1}{|k|}} e^{|k||y-y'|} \lesssim 1.$$    This concludes the proof of Lemma \ref{boundedness_of_J_k_half}.
\end{proof}

Another important distinction between $\mathfrak{J}_k$ on the half-plane and $\mathfrak{J}_k$ on the full plane, is that on the half-plane we have a non-trivial commutator with $\partial_y$.

\begin{lemma}\label{commutator_estimate_of_J_k_half}
    If $f \in H^1$, then $\mathfrak{J}_k(f) \in H^1$. Furthermore, we have
    $$[\partial_y, \mathfrak{J}_k]f(y) = -|k| \mathrm{p.v} \int_{0}^\infty \frac{e^{-|k|(y+y')}}{i(y-y')} f(y') dy',$$
    which obeys the bound
    $$||[\partial_y, \mathfrak{J}_k]||_{L^2 \to L^2} \lesssim |k|.$$
\end{lemma}

\begin{proof}
    
Proceeding as in Lemma 3.3 of \cite{bedrossian2023stability}, we introduce the regularizations
$$\frac{1}{|k|}\mathfrak{J}_k^\epsilon \coloneqq \int_0^\infty \frac{G_k(y,y')(y-y')}{2i(y-y')^2 + \epsilon^2} f(y') dy'$$
for $\epsilon > 0$. Then through integration by parts
\begin{equation}
    \begin{split}
        \frac{1}{|k|} \left(\mathfrak{J}_k^\epsilon(\partial_y f) - \partial_y \mathfrak{J}_k^\epsilon(f)\right) &= \int_0^\infty \frac{(y-y')}{2i(y-y')^2 + \epsilon^2}\left(-\partial_{y'}G_k(y,y') - \partial_y G_k(y,y')\right) f(y') dy'\\
        &=-\int_0^\infty \frac{(y-y')}{2i(y-y')^2 + \epsilon^2} e^{-|k|(y+y')} f(y') dy'.
    \end{split}
\end{equation}
If $\partial_y f \in L^2$, then from the above, we see $\frac{1}{|k|} \left(\mathfrak{J}_k^\epsilon(\partial_y f) - \partial_y \mathfrak{J}_k^\epsilon(f)\right)$ converges in $L^2$ as $\epsilon \to 0$. But then we have
$$\partial_y \mathfrak{J}_k^\epsilon(f) = \mathfrak{J}_k^\epsilon(\partial_y f) -\left(\mathfrak{J}_k^\epsilon(\partial_y f) - \partial_y \mathfrak{J}_k^\epsilon(f)\right).$$
Since the right-hand side converges in $L^2$, we observe that $\partial_y \mathfrak{J}_k^\epsilon(f)$ converges in $L^2$, and hence $\mathfrak{J}_k(f) \in H^1$ for $f \in H^1$. We note that in the above, we obtained the formula
 $$[\partial_y, \mathfrak{J}_k] f(y)= -|k| \textrm{p.v} \int_{0}^\infty \frac{e^{-|k|(y+y')}}{i(y-y')} f(y') dy'.$$
 To show that $||[\partial_y, \mathfrak{J}_k]||_{L^2 \to L^2} \lesssim |k|,$ we perform the same type of decomposition as in the proof of Lemma \ref{boundedness_of_J_k_half}:
     \begin{equation}
        \begin{split}
            \frac{1}{|k|}[\partial_y,\mathfrak{J}_k ]f(y) &= -\frac{k}{|k|} \int_{[0,\infty)\setminus (y-\frac{1}{|k|},y+\frac{1}{|k|})} e^{-|k|(y+y')}\frac{f(y')}{i(y-y')}dy'\\
            &- \frac{k}{|k|} e^{-2|k|y} \textrm{p.v.} \int_{\max(y- \frac{1}{|k|},0)}^{y+\frac{1}{|k|}} \frac{f(y')}{2i(y-y')}dy'\\
            &- \frac{k}{|k|} \textrm{p.v.} \int_{\max(y- \frac{1}{|k|},0)}^{y+\frac{1}{|k|}} 2 e^{-|k|y}\left(e^{-|k|y'} - e^{-|k|y}\right)\frac{f(y')}{2i(y-y')}dy'\\
            &\eqqcolon \mathcal{T}_1 + \mathcal{T}_2 + \mathcal{T}_3.
        \end{split}
    \end{equation}

    We again split $\mathcal{T}_1 = \mathcal{T}_{1, \leq } + \mathcal{T}_{1, \geq}$. This time, we will only present the proof for $\mathcal{T}_{1,\geq}$. Defining the kernel $K_k^{(\geq)}(y,y') \coloneqq \frac{k}{|k|} \frac{e^{-|k||y+y'|}}{i(y-y')} 1_{0 \leq y + \frac{1}{k} \leq y'}(y,y')$, it suffices by Schur's test to prove estimates on $K_k^{(\geq)}$. For fixed $y \in [0,\infty)$, we have
    \begin{align}
            ||K_k^{(\geq)}(y, \cdot)||_{L_{y'}^1} &\lesssim \int \frac{e^{-|k||y+y'|}}{|y-y'|}1_{0 \leq y + \frac{1}{k} \leq y'}(y,y') dy'\\
            &= \int_{y+\frac{1}{|k|}}^{\infty} \frac{e^{-|k||y+y'|}}{|y-y'|} dy'\\
            &= \int_1^\infty e^{-2|k|y }\frac{e^{-z}}{z} dz \lesssim e^{-2|k|y }.
        \end{align}
    Taking the supremum in $y \in [0,\infty)$ gives $||K_k^{(\geq)}||_{L_y^\infty L_{y'}^1}$. Now if we fix a $y' \in [0,\infty)$ (such that $y' - \frac{1}{|k|} > 0$ to avoid triviality), we observe via the change of variable $z = |k|(y'-y)$,
        \begin{align}
            ||K_k^{(\geq)}(\cdot, y')||_{L_{y}^1} &\lesssim \int \frac{e^{-|k|(y+y')}}{|y-y'|}1_{0 \leq y + \frac{1}{k} \leq y'}(y,y') dy\\
            &= \int_{0}^{y' - \frac{1}{|k|}} \frac{e^{-|k||y+y'|}}{|y-y'|} dy\\
            &=\int_{0}^{y' - \frac{1}{|k|}} \frac{e^{-|k||y+y'|}}{|y-y'|} dy\\
            &\leq \int_{1}^{|k|y'} e^{-2|k|y'} \frac{e^{z}}{z} dz\\
            &\leq e^{-|k|y'}\int_{1}^{|k|y'} \frac{1}{z} dz = e^{-|ky'|} \ln(|ky'|),
        \end{align}
    which is uniformly bounded in $k$ and $y'$. This gives the desired estimate $||\mathcal{T}_1||_{L^2 \to L^2} \lesssim 1$.\\

    Noticing that $e^{-2|k|y} \leq 1$, the bound on $\mathcal{T}_2$ follows the same argument as the bound on $\mathcal{T}_2$ in Lemma \ref{boundedness_of_J_k_half}. Turning our attention to $\mathcal{T}_3$, in the same manner as the proof of Lemma \ref{boundedness_of_J_k_half}, we observe that
    $$\sup_{y,y' \in [0,\infty), |y-y'| \leq \frac{1}{|k|}} e^{-|k|y}\left| \frac{e^{-|k|y'} - e^{-|k|y}}{y-y'}\right| \lesssim \sup_{{y,y' \in [0,\infty), |y-y'| \leq \frac{1}{|k|}}} |k| e^{-|k|(y + y')} \lesssim |k|.$$
    Thus the Schur estimates yield
    \begin{equation*}
        \begin{split}
            \int_{y-1/k}^{y+1/k} 1_{[0,\infty)} e^{-|k|y} \frac{|e^{-|k|y'} - e^{-|k|y}|}{|y-y'|}dy' &\lesssim |k| \int_{y- \frac{1}{|k|}}^{y + \frac{1}{|k|}} dy'\lesssim 1,\\
            \int_{y'-1/k}^{y'+1/k} 1_{[0,\infty)}  e^{-|k|y} \frac{|e^{-|k|y'} - e^{-|k|y}|}{|y-y'|}dy &\lesssim |k| \int_{y'- \frac{1}{|k|}}^{y' + \frac{1}{|k|}} dy \lesssim 1,
        \end{split}
    \end{equation*}
    which gives $||\mathcal{T}_3||_{L^2 \to L^2} \lesssim 1$, and hence completes the proof.
\end{proof}

We note without detailed proof that $\mathfrak{J}_k$ is in fact self-adjoint. The contents of Lemmas \ref{boundedness_of_J_k_half} and \ref{commutator_estimate_of_J_k_half}, together with self-adjointness of $\mathfrak{J}_k$, are enough to translate all linear and non-linear estimates from the planar case to the half-plane case. The only true difference is in the commutator between $\mathfrak{J}_k$ and $\partial_y$, and this only results in minor changes. For example, when deriving the estimate \eqref{damping_of_phi}, we estimated
$$ |\langle \nu \Delta_k \omega_k, \mathfrak{J}_k \omega_k \rangle| = |\langle \nu \nabla_k \omega_k, \mathfrak{J}_k \nabla_k\omega_k \rangle| \lesssim \nu ||\nabla_k \omega_k||_2^2.$$
On the half-plane, due to the non-trivial commutator, this becomes
\begin{equation}
    \begin{split}
        |\langle \nu \Delta_k \omega_k, \mathfrak{J}_k \omega_k \rangle| &= |\langle \nu \nabla_k \omega_k,  \mathfrak{J}_k \nabla_k \omega_k \rangle + \langle \nu \partial_y\omega_k,  [\partial_y, \mathfrak{J}_k] \omega_k \rangle |\\
        &\lesssim \nu ||\nabla_k \omega_k||_2 \left(|| \partial_y \omega_k||_2 + ||k \omega_k||_2\right)\\
        & \lesssim D_\gamma,
    \end{split}
\end{equation}
which is the same ultimate estimate. Similar considerations are taken for the non-linear $T_{\alpha, \tau \alpha}$ estimate, where we compute
$$|| \partial_y(c_\alpha I + c_\tau \mathfrak{J}_k)\partial_y\omega_k||_2^2 = || c_\alpha \partial_y^2 \omega_k + c_\tau \mathfrak{J}_k(\partial_y^2 \omega_k) + c_\tau [\partial_y, \mathfrak{J}_k] (\partial_y\omega_k)||_2^2 \lesssim ||\nabla_k \partial_y\omega_k||_2^2.$$
Other than these minor changes, the proof of Theorem \ref{half_theorem} is identical to that of Theorem \ref{main_theorem}.

\section{Modifications for the Infinite Channel}\label{channel_modifications}

In this section, we sketch a proof of Theorem \ref{channel_theorem} by pointing out changes and simplifications as compared to the proof of Theorem \ref{main_theorem}. We shall re-use many of the same symbols, but now adapted to the case of the infinite channel where $D = [-1,1]$. For example, we set $\lambda_k = \lambda^{ch}(\nu,k)$. 

\subsection{Linear Estimates for the Infinite Channel}\label{linear_channel_estimates}

Our linearized system is now
\begin{equation}\label{linear_channel_system}
    \begin{cases}
        \partial_t \omega_k + ik y \omega_k - \nu \Delta_k \omega_k = 0,\\
        -\Delta_k \phi_k = \omega_k,\\
        \phi_k (y = \pm 1) = \omega_k (y = \pm 1) = 0,
    \end{cases}
\end{equation}
and we define the following $k$-by-$k$ energy adapted to the linear problem:
\begin{equation}\label{channel_kEnergy}
    E_k[\omega_k] \coloneqq \begin{cases}
        ||\omega_k||_2^2 + c_\alpha \alpha ||\partial_y \omega_k||_2^2 - c_\beta \beta \textrm{Re}\langle ik \omega_k, \partial_y \omega_k \rangle + c_\tau  \textrm{Re}\langle \mathfrak{J}_k \omega_k, \omega_k \rangle +  c_\tau c_\alpha \alpha\textrm{Re}\langle \mathfrak{J}_k \partial_y \omega_k, \partial_y \omega_k\rangle, & |k| \geq \nu,\\
        ||\omega_k||_2^2 + c_\alpha \alpha ||\partial_y \omega_k||_2^2 + c_\tau  \textrm{Re}\langle \mathfrak{J}_k \omega_k, \omega_k \rangle +  c_\tau c_\alpha \alpha\textrm{Re}\langle \mathfrak{J}_k \partial_y \omega_k, \partial_y \omega_k\rangle, & |k| < \nu,
    \end{cases}
\end{equation}
where $\alpha$ is the same as in the planar case, but now
\begin{equation}
    \beta \coloneqq \frac{\nu^{1/3}}{|k|^{4/3}}.
\end{equation}
Note that we drop the $\beta$ term from the energy functional at low frequencies, since the Poincar\'e inequality provides stronger estimates than the hypocoercive scheme. The singular integral operator $\mathfrak{J}_k$ is now defined as:
$$\mathfrak{J}_k [ f](y) \coloneqq |k| \textrm{p.v.} \frac{k}{|k|} \int_{-1}^1 \frac{1}{2i(y-y')} G_k(y,y') f(y') dy',$$
with
$$ G_k(y,y') \coloneqq -\frac{1}{k \sinh(2k)} \begin{cases}
        \sinh(k(1-y')) \sinh(k(1+y)), & y \leq y',\\
        \sinh(k(1-y)) \sinh(k(1+y')), & y' \leq y.\end{cases}$$
Just as in the planar case, $\mathfrak{J}_k$ is uniformly bounded in $k$ as a linear operator $L^2([-1,1]) \to L^2([-1,1])$ (see Lemma \ref{boundedness_of_J_k_channel}). Unlike in the planar case, $\mathfrak{J}_k$ no longer commutes with $\partial_y$ (see Lemma \ref{commutator_estimate_of_J_k}). We define the $k$-by-$k$ dissipation functional
\begin{equation}\label{channel_k_by_k_dissipation}
    \begin{split}
        D_k[\omega_k]&\coloneqq \nu ||\nabla_k \omega_k||_2^2 + c_\alpha \alpha || \nabla_k \partial_y \omega_k||_2^2 + c_\beta \lambda_k ||\omega_k||_2^2 + c_\tau  |k|^2 ||\nabla_k \phi_k||_2^2 + c_\tau c_\alpha \alpha  |k|^2 ||\nabla_k \partial_y \phi_k||_2^2\\
        &\eqqcolon D_{k,\gamma} + c_\alpha D_{k,\alpha} + c_\beta D_{k,\beta} + c_\tau D_{k,\tau} +  c_\tau c_\alpha D_{k,\tau \alpha}.
    \end{split}
\end{equation}

Although there is no $\beta$ term for frequencies $|k|< \nu$ in the energy \eqref{channel_kEnergy}, we do still have a $\beta$ term in the dissipation at frequencies $|k| < \nu$. We further note that $\lambda_k = \beta |k|^2$ for $|k| \geq \nu.$ The analogue of Proposition \ref{linear_proposition} for the infinite channel is entirely the same, if one only makes the substitution that $\omega_k$ is an $H^1$ solution of \eqref{linear_channel_system}. Before discussing the proof of this proposition, we state two lemmas for $\mathfrak{J}_k$. These statements about $\mathfrak{J}_k$ were proved for discrete frequencies $k \in \mathbb{Z} \setminus \{0\}$ in \cite{bedrossian2023stability}, and which we now describe for general wavenumbers.

\begin{lemma}\label{boundedness_of_J_k_channel}
    For all $k \in \R \setminus \{0\}$, the operator $\mathfrak{J}_k$ extends to a bounded linear operator $L^2([-1,1]) \to L^2([-1,1])$ satisfying the uniform bound \[||\mathfrak{J}_k||_{L^2 \to L^2} \lesssim 1,\] where the implicit constant is uniform in $k$.
\end{lemma}

\begin{proof}
    This is a very slight extension of Lemma 3.1 in \cite{bedrossian2023stability}. For $ |k| > \frac{1}{2}$, the proof is in fact identical. For $|k| < \frac{1}{2}$, one follows an even simpler argument. The general structure is the same as the proof of Lemma \ref{boundedness_of_J_k_half}, but the corresponding $\mathcal{T}_1$ term disappears, creating an even simpler proof. One can check that all of the important estimates still hold at these low $|k|$ frequencies.
\end{proof}

By Lemma 3.3 in \cite{bedrossian2023stability}, if $f \in H^1$, then $\mathfrak{J}_k[f] \in H^1$, and we can explicitly compute the commutator
$$[\partial_y, \mathfrak{J}_k] [f] = |k| \textrm{p.v} \int_{-1}^1 \frac{H_k(y,y')}{2i(y-y')} f(y') dy',$$
where
$$H_k(y,y) \coloneqq -\frac{\sinh(k(y+y'))}{\sin(2k)}.$$
Additionally, we have the following estimate:

\begin{lemma}\label{commutator_estimate_of_J_k}
    For all $k \in \R \setminus \{0\}$, $||[\partial_y, \mathfrak{J}_k] ||_2 \lesssim |k|$, where the implicit constant is uniform in $k$.
\end{lemma}

\begin{proof}

    This is a restatement of Lemma 3.4 in \cite{bedrossian2023stability}. Just as in Lemma \ref{boundedness_of_J_k_channel}, the original proof decomposes $\frac{1}{|k|}[\partial_y, \mathfrak{J}_k] f$ into three terms $\mathcal{T}_i$, $i = 1,2,3$. This proof remains valid for $|k| \geq 1/2$. Hence we will only cover the $|k| < 1/2$ case, which simplifies to
    \begin{equation*}
        \begin{split}
            \frac{1}{|k|}[\partial_y,f] &= \frac{k}{|k|} H_k(y,y) \textrm{p.v.} \int_{y-\frac{1}{k}}^{y+\frac{1}{k}}1_{[-1,1]} \frac{f(y')}{2i(y-y')}dy'\\
            &\quad+\frac{k}{|k|} \textrm{p.v.} \int_{y-\frac{1}{k}}^{y+\frac{1}{k}} 1_{[-1,1]}\left(H_k(y,y') - H_k(y,y)\right)\frac{f(y')}{2i(y-y')}dy'\\
            &= \mathcal{T}_2 + \mathcal{T}_3.
        \end{split}
    \end{equation*}

    The bound on $\mathcal{T}_2$ proceeds directly as in Lemma \ref{boundedness_of_J_k_channel}, together with $||H_k(y,y)||_\infty \lesssim 1$. For $\mathcal{T}_3$, we use Schur's test, but even more simply than in the original proof:
        \begin{equation*}
        \begin{split}
            \int_{y-1/k}^{y+1/k} 1_{[-1,1]} |k|\frac{|H_k(y,y') - H_k(y,y)|}{k|y-y'|}dy' &\lesssim \sup_{y,y' \in [-1,1]} \left(\frac{|H_k(y,y') - H_k(y,y)|}{|y-y'|}\right)\int_{-1}^1 dy' \lesssim 1,\\
            \int_{y'-1/k}^{y'+1/k} 1_{[-1,1]} \frac{|H_k(y,y') - H_k(y,y)|}{|y-y'|}dy &\lesssim \sup_{y,y' \in [-1,1]} \left(\frac{|H_k(y,y') - H_k(y,y)|}{|y-y'|}\right) \int_{-1}^1 dy \lesssim 1,
        \end{split}
    \end{equation*}
    which completes the desired bounds.
\end{proof}

We further note that it was proved in \cite{bedrossian2023stability} that $\mathfrak{J}_k$ is self-adjoint, and we will not reproduce this here. The replacement for Lemmas \ref{basic_hypo_est} and \ref{disip_lin} is the following:

\begin{lemma}\label{channel_disip_lin}

Under the hypotheses of Proposition \ref{linear_proposition}, the following estimates hold, with implicit constants independent of $\nu$ and $k$. For any value of $k \neq 0$:

\begin{subequations}\label{channel_hypo_lin}
    \begin{equation}\label{channel_gamma_lin_est}
    \frac{1}{2} \frac{d}{dt} ||\omega_k||_2^2 + \nu ||\nabla_k \omega_k||_2^2 \leq 0,
\end{equation}
\begin{equation}\label{channel_alpha_lin_est}
    \frac{1}{2} \frac{d}{dt} \alpha ||\nabla_k\partial_y\omega_k||_2^2  + \nu \alpha ||\nabla_k \partial_y \omega_k||_2^2 \lesssim D_{k,\gamma} +  \lambda_k||\omega_k||_2^2,
\end{equation}
\begin{equation}\label{channel_damping_of_phi}
    \frac{d}{dt} \mathrm{Re}\langle \omega_k, \mathfrak{J}_k \omega_k\rangle + \frac{1}{2}D_{k,\tau} \lesssim D_{k,\gamma},
\end{equation}
\begin{equation}\label{channel_damping_of_dy_phi}
    \frac{d}{dt} \alpha \mathrm{Re}\langle \omega_k, \mathfrak{J}_k \omega_k\rangle + \frac{1}{2}D_{k,\tau\alpha} \lesssim D_{k,\alpha} + (\lambda_k ||\omega_k||_2^2)^{1/2}D_{k,\gamma}^{1/2}.
\end{equation}
\end{subequations}

Additionally, when $|k| \geq \nu$:
\begin{equation}\label{channel_beta_lin_est}
    -\frac{d}{dt} \beta \mathrm{Re}\langle ik \omega_k, \partial_y \omega_k \rangle + \lambda_k ||\omega_k||_2^2 \lesssim D_{k,\gamma}^{1/2} D_{k,\alpha}^{1/2}.
\end{equation}

\end{lemma}

The proof of this lemma parallels the proofs of Lemmas \ref{basic_hypo_est} and \ref{disip_lin} as well as work in \cite{bedrossian2023stability}. In particular, versions of \eqref{channel_damping_of_phi} and \eqref{channel_damping_of_dy_phi} were proved to hold on a more general class of $\omega_k$, at least for discrete frequencies. However, because we are working with Couette flow, the proofs of \eqref{channel_damping_of_phi} and \eqref{channel_damping_of_dy_phi} are almost identical with the proof of Lemma \ref{disip_lin}. The only true modification is that $\partial_y$ and $\mathfrak{J}_k$ do not commute. However, the resulting factor of $|[\partial_y, \mathfrak{J}_k] \lesssim k$ can be absorbed by the $D_\gamma$ terms, and is therefore harmless.

Using Lemma \ref{channel_disip_lin}, one can prove the version of Proposition \ref{linear_proposition} at frequencies $|k| \geq \nu$, obtaining choices of $c_\alpha, c_\beta,$ and $c_\tau$ which produce the desired $c_0$ and $c_1$. The low frequency case is similar, once we note that $|k| \leq \nu \leq 1$, and that Eqs. \eqref{channel_gamma_lin_est} and \eqref{channel_alpha_lin_est} together with Poincar\'e imply:
\begin{equation}
\begin{split}
        \frac{1}{2} \frac{d}{dt} ||\omega_k||_2^2 + \frac{\nu}{3} ||\nabla_k \omega_k||_2^2 &+ \frac{\nu}{3} ||\omega_k||_2^2 + \frac{\nu}{3} ||\omega_k||_2^2\leq 0,\\
    \frac{1}{2} \frac{d}{dt} ||\nabla_k\partial_y\omega_k||_2^2  + \frac{\nu}{2} ||\nabla_k \partial_y \omega_k||_2^2 &+ \frac{\nu}{2}||\partial_y \omega_k||_2^2 \lesssim D_{k,\gamma} +  \lambda_k||\omega_k||_2^2,
\end{split}
\end{equation}
noting that the Poincar\'e constant on $[-1,1]$ is bounded above by $1$. Then the argument proceeds in a manner similar to the high frequency case. Potentially shrinking any constants as needed (for example, taking $c_\beta < 1/3$), we obtain 
$$\frac{d}{dt}E_k + c_0(c_\tau, c_\alpha, c_\beta) \lambda_k E_k + c_1(c_\tau, c_\alpha, c_\beta) D_{k} \leq 0$$
at low frequencies as well. Although there is no $\beta$ term in the energy at these low frequencies, there is still the dissipation term $D_{k,\beta}$.

\subsection{Non-Linear Estimates for the Infinite Channel}

We now introduce the non-linear energy as
\begin{equation}\label{channel_energy}
    \mathcal{E} \coloneqq \int_{\R} e^{2 c \lambda_k t} \langle k \rangle^{2m} E_k dk,
\end{equation}
where $c$ is a sufficiently small constant. We further introduce the dissipation
\begin{equation}\label{channel_dissipation}
    \mathcal{D} \coloneqq \int_{\R} e^{2 c \lambda_k t} \langle k \rangle^{2m} E_k dk.
\end{equation}
Importantly, we observe that as written here, $\D$ in the channel case is \textit{not} integrated in time, in contrast to $\D$ in the planar case. One can express the channel case using a time-integrated dissipation, but it is simplest to use the non-time-integrated dissipation.

Note that $\mathcal{E}^{1/2}(0)$, $\mathcal{E}^{1/2}$, and $\left( \int_0^\infty \mathcal{D} dt\right)^{1/2}$ are $\approx$ to the norms under consideration in Theorem \ref{channel_theorem}. We additionally define the quantities $\mathcal{D}_*$, $* \in \{\gamma, \tau, \alpha, \tau \alpha, \beta\}$ in an analogous manner to the planar case. We do not need a version of the multiplier $M_k(t)$, since we are working with exponential decay rather than polynomial. Furthermore, there is no need for the $L^\infty_k$ energy $\mathcal{E}_2$. This enables us to express the bootstrap in a purely differential form. The primary bootstrap lemma is:

\begin{lemma}[Bootstrap Lemma on the Channel]\label{channel_bootstrap}
There exists a constant $C >0$ depending only on $m$ and the choice of the sufficiently small constant $c \in (0,1)$ such that
\begin{equation}
    \frac{d}{dt} \mathcal{E} \leq - 4c\mathcal{D} + \left(C \frac{(\ln(1/\nu)^{1/2} + 1)^2}{\nu}\mathcal{E} \right)^{1/2}\mathcal{D},
\end{equation}
for $\mathcal{E}$ as in \eqref{channel_energy}, $\mathcal{D}$ as in \eqref{channel_dissipation}, and $\omega$ an $H^1$ solution to \eqref{couette_system} such that $\mathcal{E}(0)< \infty$.
\end{lemma}

Just as Lemma \ref{bootstrap} implies Theorem \ref{main_theorem}, it is clear that proving Lemma \ref{channel_bootstrap} will suffice to prove Theorem \ref{channel_theorem}. By a similar (but simpler) argument to that in Section \ref{nonlin} and considering the linear estimates in Subsection \ref{linear_channel_estimates}, we see that one can show
$$\frac{d}{dt}\mathcal{E} \leq - 4c \mathcal{D} + \mathcal{NL},$$
where 
\begin{equation}
    \begin{split}
    \mathcal{NL} &\coloneqq \int_{\R} e^{ 2c\lambda_k t}\langle k \rangle^{2m}\biggl(  2\textrm{Re}\langle \omega_k, (I+c_\tau \mathfrak{J}_k) \mathbb{NL}_k \rangle\\
    &\quad\hphantom{\int_{\R}}+ 2 c_\alpha\alpha \textrm{Re}\langle \partial_y \omega_k, (I + c_\tau \mathfrak{J}_k) \mathbb{NL}_k\rangle\\
    &\quad\hphantom{\int_{\R}}- 1_{|k| \geq \nu}c_\beta \beta \left(\textrm{Re}\langle ik \omega_k, \partial_y \mathbb{NL}_k\rangle + \textrm{Re}\langle ik \mathbb{NL}_k, \partial_y \omega_k\rangle\right) \biggr)dk\\
    &\eqqcolon T_{\gamma, \tau}  +  T_{\alpha, \tau. \alpha} + T_\beta,
    \end{split}
\end{equation}
with $\mathbb{NL}_k \coloneqq -(\nabla^{\perp} \phi \cdot \nabla \omega)_k = -\int_{\R} \nabla_{k-k'}^{\perp} \phi_{k-k'} \cdot \nabla_{k'} \omega_{k'} dk'$. Then the proof reduces to showing the existence of a constant $C>0$ independent of $\nu$ so that
\begin{equation}
    |T_{\gamma, \tau}| + |T_{\gamma, \tau \alpha}| + |T_{\beta}| \leq \frac{C}{\nu^{1/2}}(1+\ln(1/\nu)^{1/2})\mathcal{E}^{1/2} \mathcal{D}.
\end{equation}

In the planar case, we accomplished this by breaking up into various frequency decompositions and employing lemmas from Section \ref{technical_lemmas}. Critically, we were able to estimate terms involving $\nabla^\perp_{k-k'}\phi_{k-k'}$ at low $|k-k'|$ frequencies using the $L^\infty_k$ estimates from Lemma \ref{low_frequency_lemma_phi}. The usage of such bounds was to introduce additional derivatives onto terms such as $||\partial_y \phi_{k-k'}||_\infty$ and $||(k-k') \phi_{k-k'}||_\infty$. However, in the infinite channel case, a similar effect can be achieved by applying Gagliardo-Nirenberg-Sobolev and then the Poincar\'e inequality thereby replacing Lemma \ref{low_frequency_lemma_phi} and the aspects of Lemma \ref{keyphiestimate} at low frequencies. Furthermore, there are no time integral estimates needed in the planar case. Hence the corresponding versions of the Lemmas from Section \ref{technical_lemmas} have the time integrals removed in the obvious manner. We demonstrate this by showing the proof for $T_{\gamma, \tau}$.

\subsection{Bound on \texorpdfstring{$T_{\gamma, \tau}$}{Gamma and Tau terms} for the Infinite Channel}\label{channel_gamma}

We recall the notation

$$T_{\gamma,\tau}^x \coloneqq -\int_\R \int_{\R}  e^{2c\lambda_k t}\langle k  \rangle^{2m} 2\textrm{Re}\langle \omega_k, (I + c_\tau \mathfrak{J}_k) i(k-k')\phi_{k-k'} \partial_y\omega_{k' }\rangle dk' dk.$$ Unlike in the plane, where we used a threefold frequency decomposition, we are able to estimate $T_{\gamma, \tau}^x$ on the infinite channel without any splitting. Using the embedding of $\dot{H}_0^1([-1,1])$ in $L^\infty([-1,1])$, 
\begin{align}
        T_{\gamma, \tau}^x & \lesssim \int_{\R} \int_{\R} e^{c\lambda_k t} \langle k  \rangle^{m} ||\omega_k||_2 \langle k  \rangle^{m}||(k-k') \phi_{k-k'}||_\infty  e^{c\lambda_{'}k t}||\partial_y \omega_{k'}||_2 dk' dk\\
        &\lesssim \int_{\R} \int_{\R} e^{2c\lambda_k t} \langle k  \rangle^{m} ||\omega_k||_2 \left(\langle k-k'\rangle^m + \langle k' \rangle^m\right)\\
        &\hphantom{\lesssim \int_{\R} \int_{\R}} e^{c \lambda_{k-k'}t}||(k-k') \phi_{k-k'}||_\infty  e^{c\lambda_{k'} t}||\partial_y \omega_{k'}||_2 dk' dk\\
        &\lesssim \int_{\R} \int_{\R} e^{c\lambda_k t} \langle k  \rangle^{m} ||\omega_k||_2 \left(\langle k-k'\rangle^m + \langle k' \rangle^m\right)\\
        &\hphantom{\lesssim \int_{\R} \int_{\R}} e^{c \lambda_{k-k'}t}(k-k')|| \partial_y \phi_{k-k'}||_2 e^{c\lambda_{k'} t}||\partial_y \omega_{k'}||_2 dk' dk\\
        &\lesssim \mathcal{E}^{1/2} \mathcal{D}_\tau^{1/2}\nu^{-1/2} \mathcal{D}_\gamma^{1/2}.
    \end{align}
 This embedding can be used in multiple sections of the proof to handle the $||(k-k') \phi_{k-k'}||_\infty$ term, and provides the general estimate \begin{equation}\label{phi_on_channel}
    \int_{\R} e^{c \lambda_{k} t} ||k \phi_k||_\infty dk \lesssim \mathcal{D}_\tau, \end{equation} which replaces Lemma \ref{low_frequency_lemma_phi} in several low frequency estimates. Note that we also exploited the subbadditivity of $k \mapsto \lambda_k$ to write 
    $e^{c \lambda_k t} \leq e^{c \lambda_{k-k'}t} e^{c \lambda_{k'} t},$ which is key to propagating the stronger exponential decay. We now turn to $T_{\gamma, \tau}^y$, which we recall is given by:
$$T_{\gamma,\tau}^y \coloneqq -\int_\R \int_{\R}  e^{2c\lambda_k t}\langle k  \rangle^{2m} 2\textrm{Re}\langle \omega_k, (I + c_\tau \mathfrak{J}_k) \partial_y\phi_k ik'\omega_k \rangle dk' dk.$$
In Section \ref{gamma_tau_section}, we used the following frequency decomposition of $T_{\gamma, \tau}^y$, which we use again on the infinite channel. We write the breakdown of $T_{\gamma, \tau}^y$ as 
\begin{equation}\label{y_decomp_T_2}
\begin{split}
        T_{\gamma, \tau}^y \eqqcolon& \left(T_{\gamma, \tau, LH, H', H', \cdot}^y + T_{\gamma, \tau, LH, H', L, \cdot}^y+ T_{\gamma, \tau, LH, L, \cdot, \cdot}^y\right)\\
        &+ \left(T_{\gamma, \tau, HL, H, \cdot, H'}^y+ T_{\gamma, \tau, HL, M, \cdot, H'}^y  + T_{\gamma, \tau, HL, H', H', L}^y + T_{\gamma, \tau, HL, H', L, L}^y + T_{\gamma, \tau, HL, L, \cdot, \cdot}^y\right).
        \end{split}
\end{equation}
All of the terms in the decomposition \eqref{y_decomp_T_2} have the same proof as in the planar case (modulo the handling of the time integration) except for $T_{\gamma, \tau, LH, L, \cdot, \cdot}^y$, $T_{\gamma, \tau, HL, H', L, L}^y$, and $T_{\gamma, \tau, HL, L, \cdot, \cdot}^y$. These are the terms where $|k-k'| \lesssim \nu$. Using the $\dot{H}_0^1([-1,1])$ in $L^\infty([-1,1])$ embedding and elliptic regularity, together with $|k-k'| \lesssim \nu$ on the domain of integration, H\"older's inequality, and Poincar\'e's inequality:
\begin{align}
\displaybreak[0]
        |T_{\gamma, \tau, LH, L, \cdot, \cdot}^y| + |T_{\gamma, \tau, HL, H', L, L}^y|& + |T_{\gamma, \tau, HL, L, \cdot, \cdot}^y|\\ & \lesssim \int_{\R} \int_{ \R} 1_{|k-k'| \lesssim \nu}e^{2c \lambda_k t}\langle k  \rangle^{2m}  ||\omega_k||_2 ||\partial_y \phi_{k-k'}||_\infty ||k'\omega_{k'}||_2 dk' dk\\
        &\lesssim \int_{\R} \int_{ \R} 1_{|k-k'| \lesssim \nu}e^{c \lambda_k t}\langle k  \rangle^{m}  ||\partial_y \omega_k||_2  e^{c\lambda_{k-k'} t} \langle k-k'\rangle^m\\
        &\hphantom{\lesssim \int_{\R} \int_{ \R}}||\partial_y^2 \phi_{k-k'}||_2 e^{c \lambda_{k'}t}\langle k' \rangle^m||k'\omega_{k'}||_2 dk' dk\\
        &\lesssim \nu^{-1/2} \mathcal{D}_\gamma^{1/2} \left(\int_{|k| \lesssim \nu} e^{c \lambda_k t}\langle k  \rangle^{m}  ||\omega_k||_2\right)\nu^{-1/2} \mathcal{D}_\gamma^{1/2}\\
        &\lesssim \nu^{-1/2} \mathcal{D}_\gamma^{1/2} \nu^{1/2}\mathcal{E}^{1/2} \nu^{-1/2} \mathcal{D}_\gamma^{1/2}.
    \end{align}
Indeed, we have the general estimate
    \begin{equation}\label{dyphi_to_energy_channel}
    \begin{split}
        \int_{\R} e^{c \lambda_k t} ||\partial_y \phi_k||_\infty dk \lesssim \mathcal{E}^{1/2},
    \end{split}
\end{equation}which replaces Lemma \ref{low_frequency_lemma_phi} in several instances. These bounds on $T_{\gamma, \tau}$ are representative of the type of estimates needed for $T_{\alpha, \tau\alpha}$ and $T_{\beta}$, and so we will not go into as much detail in the following sections.
\subsection{Bound on \texorpdfstring{$T_{\alpha, \tau \alpha}$}{Alpha and Tau Alpha terms} for the Infinite Channel}

 Just as in the $T_{\gamma, \tau}$ case, the arguments in the regions where $|k-k'| \gtrsim \nu$ are largely the same on the channel as on the plane, and so we do not go into detail, other than noting that we do not need to consider any subtleties related to integration in time.
 
 Meanwhile, the bounds in regions where $|k-k'| \lesssim \nu$ follow from $\dot{H}_0^1([-1,1])$ in $L^\infty([-1,1])$ embedding (or equivalently Poincar\'e and Gagliardo-Nirenberg-Sobolev), and elliptic regularity. This is similar in spirit to Section \ref{channel_gamma}, and so we do not go into detail. There is a slight difference in the preliminary stages of the proof however. We still need to deal with terms of the form $|| \partial_y( I + c_\tau \mathfrak{J}_k)\partial_y\omega_k||_2$, but now $\mathfrak{J}_k$ and $\partial_y$ do not commute. However, we note that by Lemmas \ref{boundedness_of_J_k_channel} and \ref{commutator_estimate_of_J_k}, we have
$$|| \partial_y(c_\alpha I + c_\tau \mathfrak{J}_k)\partial_y\omega_k||_2^2 = || c_\alpha \partial_y \omega_k + c_\tau \mathfrak{J}_k(\partial_y \omega_k) + c_\tau [\partial_y, \mathfrak{J}_k] (\partial_y\omega_k)||_2^2 \lesssim ||\nabla_k \partial_y\omega_k||_2^2.$$
Hence we freely use that $|| \partial_y(c_\alpha I + c_\tau \mathfrak{J}_k)\partial_y\omega_k||_2 \lesssim ||\nabla_k \partial_y\omega_k||_2$, which is the same estimate ultimately used in Subsection \ref{alpha_tau_alpha_plane}.
\subsection{Bound on \texorpdfstring{$T_{\beta}$}{Beta terms} for the Infinite Channel}
While the high $k$-frequency terms in $T_\beta$ are still the most complicated, they are dealt with on the channel in the same manner as on the plane, so long as $|k-k'| \gtrsim \nu$. We also note that $T_\beta$ has no component with $|k| \leq \nu$ on the channel, which slightly simplifies the computations. As a final note, recall that we split $T_{\beta} \coloneqq T_{\beta ;1} + T_{\beta ; 2}$, and saw that $T_{\beta ;1}$ and $T_{\beta ; 2}$ have similar treatments. We do the same here, and write $T_\beta$ for $T_{\beta;1}$.

For $T_{\beta}^x$, we can combine both of the cases from the planar equations into a single case by \eqref{phi_on_channel} and $B(k)^2|k| \lesssim A(k)$: \begin{align}
        |T_\beta^x| &\lesssim \int_{|k| \geq \nu} \int_\R e^{2c \lambda_k t} \langle k \rangle^{2m} B(k)^2 ||k \partial_y \omega_k||_2 ||(k-k') \phi_{k-k'}||_\infty ||\partial_y \omega_{k'}||_2 dk' dk\\
        &\lesssim \int_{|k| \geq \nu} \int_\R e^{c \lambda_k t} \langle k \rangle^{m}  A(k)|| \partial_y \omega_k||_2 \left(\langle k-k'\rangle^m + \langle k' \rangle^m\right) e^{c\lambda_{k-k'} t} ||(k-k') \phi_{k-k'}||_\infty e^{c\lambda_{k'} t} ||\partial_y \omega_{k'}||_2 dk' dk\\
        &\lesssim \mathcal{E}^{1/2} \mathcal{D}_\tau^{1/2} \nu^{-1/2} \mathcal{D}_\gamma^{1/2}.
    \end{align}

For $T_\beta^{y,1}$, we do not need to make any significant changes from the planar case, since the stream function term is of the form $\partial_y^2 \phi_{k-k'}$, removing the need for additional derivatives. For $T_\beta^{y,2}$, since $|k| \geq \nu$, all of the $HL$ terms with $|k-k'| \geq |k'|/2$ have $|k-k'| \gtrsim |k| \geq \nu$. Hence the $HL$ terms are the same as on the plane. Meanwhile the single $LH$ case follows from \eqref{dyphi_to_energy_channel}, $B(k)^2 k^2 \lesssim A(k)|k|$, Lemma \ref{beta_and_gamma}, and $|k| \approx |k'|$ on the domain of integration:\begin{align}
        |T_{\beta, LH}^{y,2}| &\lesssim \int_{|k| \geq \nu} \int_\R \int_\R e^{2c \lambda_k t} \langle k \rangle^{2m} B(k)^2 ||k \partial_y \omega_k||_2 ||\partial_y \phi_{k-k'}||_\infty ||k' \partial_y \omega_{k'}||_2 dk' dk\\
        &\lesssim \int_{|k| \geq \nu} \int_\R e^{c \lambda_k t} \langle k \rangle^{m} A(k) |k| || \partial_y \omega_k||_2 \left(\langle k-k'\rangle^m + \langle k' \rangle^m\right)e^{c \lambda_{k-k'} t} || \partial_y \phi_{k-k'}||_\infty e^{c\lambda_{k'} t} ||\partial_y \omega_{k'}||_2 dk' dk\\
        &\lesssim \mathcal{D}_\gamma^{1/4}\mathcal{D}_\beta^{1/4} \mathcal{E}^{1/2} \nu^{-1/2} \mathcal{D}_\gamma.
    \end{align}
This concludes the proof of Lemma \ref{channel_bootstrap}, and hence of Theorem \ref{channel_theorem}.

\addtocontents{toc}{\protect\setcounter{tocdepth}{0}}

\section*{Declarations}

\subsubsection*{Acknowledgements}This material is based upon work supported by the National Science Foundation Graduate Research Fellowship Program under Grant No. DGE-2034835. Any opinions, findings, and conclusions or recommendations expressed in this material are those of the authors and do not necessarily reflect the views of the National Science Foundation. J.B was supported by NSF Award DMS-2108633.

\subsubsection*{Conflicts of Interest}
The authors have no competing interests to declare that are relevant to the content of this article.

\subsubsection*{Data Availability}
Data sharing not applicable to this article as no datasets were generated or analyzed during the current study.

\addtocontents{toc}{\protect\setcounter{tocdepth}{1}}

\bibliographystyle{plain} 
\bibliography{citations} 

\end{document}